\documentclass[11pt,oneside]{amsart}

\usepackage{amsmath,ifthen, amsfonts, amssymb,
srcltx, amsopn, color, pstricks-add, overpic}

\usepackage[sc]{mathpazo}
\usepackage[T1]{fontenc}
\usepackage{yfonts}

\usepackage{graphicx}

\newcommand{\showcomments}{no}

\newsavebox{\commentbox}
{\ifthenelse{\equal{\showcomments}{yes}}%
{\footnotemark
        \begin{lrbox}{\commentbox}
        \begin{minipage}[t]{1.25in}\raggedright\sffamily\tiny
        \footnotemark[\arabic{footnote}]}
{\begin{lrbox}{\commentbox}}}%
{\ifthenelse{\equal{\showcomments}{yes}}%
{\end{minipage}\end{lrbox}\marginpar{\usebox{\commentbox}}}
{\end{lrbox}}}

\DeclareMathOperator{\visual}{\partial_{\infty}}

\newtheorem{thm}{Theorem}[section]
\newtheorem{lem}[thm]{Lemma}

\newtheorem{cor}[thm]{Corollary}
\newtheorem{conj}[thm]{Conjecture}

\newtheorem{prop}[thm]{Proposition}

\theoremstyle{definition}
\newtheorem{defn}[thm]{Definition}
\newtheorem{rem}[thm]{Remark}
\newtheorem*{remi}{Remark}
\newtheorem{exmp}[thm]{Example}

\newtheorem{notation}[thm]{Notation}

\newtheorem{claim}{Claim}

\newtheorem*{corj}{Corollary}

\newtheorem{claim*}{Claim}

\DeclareMathOperator{\dimension}{dim}

\DeclareMathOperator{\image}{im}

\DeclareMathOperator{\Aut}{Aut}

\DeclareMathOperator{\stabilizer}{Stab}
\DeclareMathOperator{\diam}{diam}

\newcommand{\scname}[1]{\text{\sf #1}}
\newcommand{\area}{\scname{Area}}

\newcommand{\field}[1]{\mathbb{#1}}
\newcommand{\integers}{\ensuremath{\field{Z}}}

\newcommand{\naturals}{\ensuremath{\field{N}}}
\newcommand{\reals}{\ensuremath{\field{R}}}
\newcommand{\Euclidean}{\ensuremath{\field{E}}}

\makeatletter

\newcommand{\Rmnum}[1]{\mathbf{{\expandafter\@slowromancap\romannumeral #1@}}}

\newcommand{\simp}{\ensuremath{\partial_{_{\triangle}}}}

\newcommand{\dive}[2]{\ensuremath{{\mathbf{div}}({#1},{#2})}}

\makeatother

\newcommand{\ascone}[4]{\ensuremath{\mathbf{Cone}}_{#4}({#1},{#2},{#3})}
\newcommand{\lefth}[1]{\ensuremath{\overleftarrow{#1}}}
\newcommand{\righth}[1]{\ensuremath{\overrightarrow{#1}}}
\newcommand{\lrhalf}[1]{\ensuremath{\widehat{\mathcal{#1}}}}
\newcommand{\divergence}[5]{\ensuremath{\mathbf{div}_{{#4},{#5}}({#1},{#2},{#3})
}}
\newcommand{\Divergence}[3]{\ensuremath{\mathbf{Div}^{{#1}}_{{#2},{#3}}}}

\let\oldmarginpar\marginpar
\renewcommand\marginpar[1]{\-\oldmarginpar[\raggedleft\footnotesize #1]%
{\raggedright\footnotesize #1}}

\newcommand{\Xomega}{{\mathbf X}_{\omega}}

\newcommand{\done}{\dot{d}}
\newcommand{\dtwo}{\ddot{d}}

\newcommand{\co}{\colon}

\newcounter{enumitemp}
\newenvironment{enumeratecontinue}{
  \setcounter{enumitemp}{\value{enumi}}
  \begin{enumerate}
  \setcounter{enumi}{\value{enumitemp}}
}
{
  \end{enumerate}
}

\long\def\Restate#1#2#3#4{
\medskip\par\noindent
{\bf #1 \ref{#2} #3} {\it #4}\par\medskip }

\begin{document}
\title[Thickness, relative hyperbolicity, and simplicial boundaries]
{Cubulated groups: thickness, relative hyperbolicity, and simplicial boundaries}

\author[J. Behrstock]{Jason Behrstock}\thanks{Behrstock was supported
as an Alfred P. Sloan Fellow and by NSF grant \#DMS-1006219}
\address{Lehman College and The Graduate Center, City University of New York,
U.S.A.}
\email{jason.behrstock@lehman.cuny.edu}

\author[M.~Hagen]{Mark F. Hagen}\thanks{Hagen was partially supported by the NSF under Grant Number NSF 1045119.}
\address{Dept. of Math., University of Michigan, Ann Arbor, Michigan, USA }
\email{markfhagen@gmail.com}

\date{March 10th, 2014}
\maketitle

\begin{abstract}
Let $G$ be a group acting geometrically on a CAT(0) cube complex $\mathbf X$.
We prove first that $G$ is hyperbolic relative to the collection $\mathbb P$ of
subgroups if and only if the simplicial boundary $\simp\mathbf X$ is the
disjoint union of a nonempty discrete set, together with a
pairwise-disjoint collection of subcomplexes corresponding, in the appropriate sense, to elements of
$\mathbb P$.  As a special case of this result is a new proof, in the cubical case, of a Theorem of Hruska--Kleiner regarding Tits boundaries of relatively hyperbolic CAT(0) spaces.  Second, we relate the existence of cut-points in asymptotic cones
of a cube complex $\mathbf X$ to boundedness of the 1-skeleton of $\simp\mathbf
X$.  We deduce characterizations of thickness and strong algebraic thickness of a group $G$ acting properly and cocompactly on the CAT(0) cube complex $\mathbf X$ in terms of the structure of, and nature of the $G$-action on, $\simp\mathbf X$.  Finally, we construct, for each $n\geq 0, k\geq 2$, infinitely many quasi-isometry types of group $G$ such that $G$ is strongly algebraically thick of order $n$, has polynomial divergence of order $n+1$, and acts properly and cocompactly on a $k$-dimensional CAT(0) cube complex.
\end{abstract}

\setcounter{tocdepth}{1}
\tableofcontents

\section*{Introduction}
In this paper, we study the mutually exclusive properties of relative
hyperbolicity and thickness of groups, in the context of groups acting
properly and cocompactly on CAT(0) cube complexes.  Since being first introduced as an interesting family of CAT(0) spaces by Gromov in~\cite{Gromov87}, the class of CAT(0) cube complexes has been recognized as being sufficiently rich to warrant a theory encompassing more than just CAT(0) geometry.  Applications of this theory range from their use by Charney-Davis to resolve the $K(\pi,1)$ problem for hyperplane complements~\cite{CharneyDavis95b} to the recent resolution of the virtual fibering conjecture~\cite{WiseIsraelHierarchy} and virtual Haken conjectures~\cite{AgolVirtualHaken}, among many others.  In the setting of cubical complexes, we study the properties of relative
hyperbolicity and thickness, the latter of which is a powerful obstruction to the former. We show that despite these two properties being antithetical, surprisingly, they admit similar characterizations using the boundary of the space.  This is the only setting in which such a close relationship between these two properties is known.

A CAT(0) cube complex has a highly organized combinatorial structure that yields an associated space, the
\emph{simplicial boundary}, which encodes much of the large-scale structure of
the cube complex.  Our results show that, to a large extent, both relative
hyperbolicity and thickness of a group $G$ acting geometrically on a cube
complex $\mathbf X$ correspond to simple properties of the
simplicial boundary of $\mathbf X$ and the natural action of $G$ on the simplicial boundary of $\mathbf X$.

The definition of a \emph{thick metric space} and the attendant
quasi-isometry invariant, \emph{order of thickness}, were introduced
by Behrstock--Dru\c{t}u--Mosher \cite{BDM} as an obstruction to relative hyperbolicity and as a tool to study geometric commonalities between several
classes of groups, notably mapping class groups of surfaces, outer automorphism
groups of finitely generated free groups, and $SL_n(\integers)$.  We review
thick metric spaces in detail in Section~\ref{sec:thickprelim}.

The order of thickness of $M$, defined below, is intimately related to the divergence function
of $M$.  The relevant notion of divergence of a metric space originates in work
of
Gromov and Gersten~\cite{Gromov:Asymptotic,
GerstenDivergence2,GerstenDivergence}, and, roughly speaking,
estimates how far one must travel in $M$ from a point $a$ to a point $b$,
avoiding a specified ball centered at a third point $c$.  Divergence can be
studied via asymptotic cones of $M$.  In particular,
Dru\c{t}u, Mozes, and Sapir proved that, if $M$ is quasi-isometric to a
finitely generated group, then $M$ has linear divergence if and only if it is
wide~\cite{DrutuMozesSapir}.  Furthermore, the first author and Dru\c{t}u
proved in~\cite{BD} that the divergence of $M$ is bounded above by a polynomial
of order $n+1$ when $M$ is a metric space that is strongly thick of order $n$.

The order of thickness of a metric space $M$ is defined inductively. First, $M$
is [strongly] thick of order 0 if $M$ is \emph{unconstricted} [\emph{wide}],
which means that some [any] asymptotic cone of $M$ has no cut-point.  $M$ is
\emph{[strongly] thick of order at most $n\geq 1$} if there is a collection of
quasiconvex \emph{thickly connecting} subspaces $\{S_i\}$ that coarsely cover
$M$, with the additional property that each $S_i$ is [strongly] thick of order
at most $(n-1)$.  Being \emph{thickly connected} means that for any $p,q\in M$,
there is a sequence $S_{i_1},\ldots,S_{i_k}$ with $p\in S_{i_1},q\in S_{i_k}$
and $\diam(S_{i_j}\cap S_{i_{j+1}})=\infty$ for all $j$. An important variation
on this notion occurs when $M$ is quasi-isometric to a finitely generated group,
and the sets $S_i$ are cosets of a finite collection of quasi-convex subgroups,
each of which is \emph{(strongly) algebraically thick of order $n-1$}.  In this
case, $M$ is \emph{(strongly) algebraically thick of order $n$}.  \emph{Algebraically thick of order 0} means unconstricted, and \emph{strongly algebraically thick of order 0} means wide.

CAT(0) cube complexes are a generalization of trees in two fundamental ways.  First, the class of
graphs that are 1-skeleta of CAT(0) cube complexes is precisely the class of
median graphs, of which trees are a special case, as was established independently by Chepoi and by
Roller~\cite{Chepoi2000,Roller98}.  Second, CAT(0) cube complexes contain large
collections of convex subspaces with exactly two complementary components.
These convex subspaces are the \emph{hyperplanes}; in the 1-dimensional case, hyperplanes are
midpoints of edges.  A detailed discussion of basic properties of CAT(0) cube complexes occurs in Section~\ref{sec:cubecomplex}.

Just as cube complexes generalize trees, the theory of groups acting on trees
generalizes, yielding a theory of groups acting on cube complexes.  See Sageev~\cite{Sageev95} and later developments in work of Chatterji-Niblo, Haglund-Paulin, Hruska-Wise, and
Nica~\cite{ChatterjiNiblo04,HaglundPaulin98,HruskaWiseAxioms,NicaCubulating04}.
 The class of groups known to be \emph{cubulated} --- i.e., to admit a
metrically
proper action by isometries on a CAT(0) cube complex --- is ever-growing and
contains many Coxeter groups~\cite{NibloReeves03}, right-angled Artin
groups~\cite{CharneyDavis94}, Artin groups of finite
type~\cite{CharneyDavis95b}, groups satisfying sufficiently strong
small-cancellation conditions~\cite{WiseSmallCanCube04}, random groups at
sufficiently low density in Gromov's model~\cite{OllivierWiseDensity},
appropriately-chosen subgroups of fundamental groups of nonpositively-curved
graph manifolds~\cite{LiuGraphManifold,PrzytyckiWiseGraphManifold}, certain
graphs of cubulated groups~\cite{HsuWiseCubulatingMalnormal}, and many others.

The connection between relative hyperbolicity and thickness for cube complexes results from the fact that these two properties of a group acting geometrically on a CAT(0) cube complex can both be detected by examining the action of the group on the simplicial boundary of the cube complex.  The \emph{simplicial boundary} $\simp\mathbf X$ was introduced
by Hagen \cite{HagenBoundary} as a combinatorial analogue of the Tits boundary of $\mathbf X$. The simplicial boundary is a simplicial complex that is an invariant of the median graph $\mathbf
X^{(1)}$, obtained by taking the 1-skeleton of $\mathbf X$, or, equivalently, of the hyperplanes and how they interact.  In the
event of a proper, cocompact action on $\mathbf X$, the two
boundaries are quasi-isometric in a strong sense discussed in Section~\ref{sec:titschar}~\cite[Section~3.5]{HagenBoundary}. Simplices of $\simp\mathbf X$
are represented by set of hyperplanes in $\mathbf X$ modeled on the set of
hyperplanes separating some basepoint from a collection of points at infinity,
and since an isometric action of a group $G$ on $\mathbf X$ preserves the set of
hyperplanes, such an action induces an action of $G$ on $\simp\mathbf X$ by
simplicial automorphisms.
A more discussion of the simplicial boundary is provided in Section~\ref{sec:simplicialboundary}.

\subsection*{Relative hyperbolicity}
Relatively hyperbolic cubulated groups form a rich family. For instance, by
recent work of Wise~\cite{WiseIsraelHierarchy}, if $M$ is a finite-volume cusped hyperbolic
3-manifold with a geometrically finite incompressible surface, then $M$ has a
finite cover $\widehat M$ such that $\pi_1\widehat M$ is the fundamental group
of a compact nonpositively-curved cube complex.  The simplicial boundary of the
universal cover of such a cube complex is described by Theorem~\ref{thm:relhyp} below.

Given a group $G$
acting properly and cocompactly on a CAT(0) space $Y$, it is natural to search for characterizations of hyperbolicity of $G$ relative to a collection of subgroups.  A result of
Hruska-Kleiner achieves this in the special case in which each peripheral subgroup is free abelian; they prove that $G$ is hyperbolic relative to a collection of free abelian subgroups if and only if the Tits boundary $\partial_TY$ decomposes as
the union of an infinite set of isolated points and an infinite collection of
spheres, which are boundaries of flats in $Y$ corresponding to the peripheral
subgroups~\cite{HruskaKleiner}.
The following two results generalize Hruska-Kleiner's result in the cubical setting, by removing any assumptions on the peripheral subgroups.  Just as Hruska-Kleiner's result shows that the property of being hyperbolic relative to free abelian subgroups corresponds to the existence of a simple geometric description of the Tits boundary, the following theorems relate relative hyperbolicity of cubulated groups to the existence of a simple decomposition of the simplicial (and therefore Tits) boundary into pieces with simpler structure.

\Restate{Theorem}{thm:relhyp}{}
{Let $(G,\mathbb P)$ be a relatively hyperbolic structure and let $G$ act
properly and
cocompactly on the CAT(0) cube complex $\mathbf X$.  Then $\simp\mathbf X$
consists of an infinite collection of isolated 0-simplices, together with a pairwise-disjoint collection $\{g\simp\mathbf Y_P:P\in\mathbb P,g\in G\}$ of subcomplexes, with each $\mathbf Y_P$ the convex hull of a $P$-orbit in $\mathbf X$.
}

When each $P\in\mathbb P$ is isomorphic to $\integers^{n_P}$ for some $n_P\geq
2$, the complex $\simp\mathbf Y_P$ is isomorphic to the $(n-1)$-dimensional
hyperoctahedron, and thus homeomorphic to $\mathbb S^{n-1}$; see Corollary~\ref{cor:relhypabelian}.  Conversely, the following shows that
relative hyperbolicity can be identified by examining the action on the simplicial boundary:

\Restate{Theorem}{thm:relhypconverse}{}{
Let $G$ act properly and cocompactly on the CAT(0) cube complex $\mathbf X$.
Let $\{\mathbf S_i\}_i$ be a $G$-invariant collection of pairwise-disjoint
subcomplexes of $\simp\mathbf X$, such that $\simp\mathbf X$ consists of
$\sqcup_i\mathbf S_i$ together with a $G$-invariant collection of isolated
0-simplices.  Suppose each $\stabilizer_G(\mathbf S_i)$ acts with a quasiconvex orbit on $\mathbf X$ and has infinite index in $G$, and that $\mathbf S_i$ contains all limit simplices for the action of $\stabilizer_G(\mathbf S_i)$.  Then $G$ is hyperbolic relative to a collection of subgroups, each of which is commensurable with some $\stabilizer_G(\mathbf S_i)$.
}

Corollary~\ref{cor:relhyptits1} provides an analogue of Theorem~\ref{thm:relhyp} and Theorem~\ref{thm:relhypconverse} in terms of
 the Tits boundary. In particular, this provides a characterization of relative hyperbolicity of a group acting geometrically on a cube complex $\mathbf X$ in terms of the action of $G$ on $\partial_T\mathbf X$.

\subsection*{Thickness} Important motivating examples of cocompactly cubulated groups are the
right-angled Artin groups, see Charney--Davis \cite{CharneyDavis94}. In contrast to the fundamental groups of finite volume hyperbolic manifolds mentioned above,  right-angled Artin groups are cocompactly
cubulated groups which are not relatively hyperbolic; in fact, these groups are thick \cite{BDM}.
Behrstock--Charney showed that one-ended
right-angled Artin groups that are thick of order 0 (and thus have linear
divergence) are precisely those whose presentation graphs decompose as
nontrivial joins \cite{BehrstockCharney}.  Motivated by this result, Hagen generalized this to show that a cocompactly cubulated
groups has linear divergence if and only if it acts geometrically on a
CAT(0) cube complex whose simplicial boundary decomposes as a nontrivial
simplicial join \cite{HagenBoundary}.  Otherwise, the simplicial boundary is disconnected and contains many isolated 0-simplices corresponding to endpoints of
axes of rank-one isometries \cite[Corollary~B]{CapraceSageev}.
Accordingly, as Theorem~\ref{thm:cocompactwide} we record the fact that if a CAT(0) cube complex $\mathbf X$ admits a geometric action by a
group $G$, then $\mathbf X$ and $G$ are each thick of order 0 exactly when
the simplicial boundary of $\mathbf X$ is connected.

For proper, cocompact CAT(0) cube complexes, the
property of being thick of order 1
admits a succinct characterization in terms of the simplicial boundary. We summarize this by:

\Restate{Theorem}{thm:introthick}{(Characterization of thickness)}{
Let $G$ act properly and cocompactly by isometries on the fully visible CAT(0) cube complex
$\mathbf X$.  If $G$ is algebraically thick of order 1 relative to a collection of quasiconvex wide subgroups, then $\simp\mathbf X$ is disconnected and contains a positive-dimensional, $G$-invariant connected component.

Conversely, if $\simp\mathbf X$ is disconnected, and has a positive-dimensional $G$-invariant component, then $\mathbf X$ is thick of order 1 relative to a set of wide, convex subcomplexes, and, in particular,  $G$ is thick of order~1.
}

Moreover, we obtain the following complete description of the boundary of a cube complex admitting a geometric action by a group that is strongly algebraically thick of order 1.  This description of algebraic thickness closely parallels that of relative hyperbolicity provided by Theorem~\ref{thm:relhypconverse}.

\Restate{Theorem}{thm:introthick}{(Description of the boundary)}{
Let $G$ act properly and cocompactly on the CAT(0) cube complex $\mathbf X$.  Then $G$ is strongly algebraically thick of order 1 if and only if $\simp\mathbf X$ is disconnected and has a positive-dimensional, $G$-invariant connected subcomplex $\mathfrak C=\cup_{A\in\mathcal A,g\in G}gA$, where $\mathcal A$ is a finite collection of bounded subcomplexes such that:
\begin{enumerate}
 \item Each $\stabilizer(A)$ acts on $\mathbf X$ with a quasiconvex orbit.
 \item For each $A\in\mathcal A$, $f^{-1}(A)$ belongs to the limit set of $\stabilizer(A)$.
 \item $f^{-1}(\mathfrak C)$ is contained in the limit set of $\langle\{\stabilizer(A):A\in\mathcal A\}\rangle$.
\end{enumerate}
}

\begin{remi}
Here, $f\co\partial_{\infty}\mathbf X\rightarrow\simp\mathbf X$ is a surjection from the visual boundary to the simplicial boundary which sends each asymptotic class of CAT(0) geodesic rays to a point in the simplex of $\simp\mathbf X$ represented by the set of hyperplanes crossing some ray in the given asymptotic class; see Section~\ref{sec:thickoforderone}.  \emph{Full visibility} of $\mathbf X$ is a technical condition on $\simp\mathbf X$ saying roughly that each infinite family of nested halfspaces in $\mathbf X$ determines a combinatorial geodesic ray.

Condition~$(3)$ is used to verify that $\langle\{\stabilizer(A):A\in\mathcal A\}\rangle$ has finite index in $G$, as required by the definition of algebraic thickness.  In contrast to the situation for many other examples of thick groups (see \cite{BDM}), in the present case there does not appear to be natural choice of generators of these subgroups from which one can easily see that the collection of them generate a finite index subgroup of $G$.
\end{remi}

From Theorem~\ref{thm:introthick}, an application of Corollary~4.17 of~\cite{BD} immediately yields:

\begin{corj}\label{cor:introdiv}
Let $G$ act properly and cocompactly on the CAT(0) cube complex $\mathbf X$, and suppose that $\simp\mathbf X$ has a $G$-invariant connected proper subcomplex satisfying $(1)-(3)$ of Theorem~\ref{thm:introthick}.  Then $G$ has quadratic divergence function.
\end{corj}

Theorem~\ref{thm:introthick} and the above corollary are, respectively, equivalent to very similar statements about the $G$-action on the Tits boundary of $\mathbf X$; see Corollary~\ref{cor:thicktits} below.

A key ingredient in the proof of Theorem~\ref{thm:thickoforderone} is
Theorem~\ref{thm:unconstrictedCC}, which relates the existence of cut-points in
some asymptotic cone of a cube complex (not necessarily cocompact) to
boundedness of the 1-skeleton of the simplicial boundary.  The proof of this
theorem occupies much of Section~\ref{sec:unconstricted}, and relies in part on
the relationship between divergence and wideness discussed
in~\cite{DrutuMozesSapir} and the relationship between divergence and the
simplicial boundary discussed in~\cite{HagenBoundary}.

We show that there are many cocompactly cubulated groups that are thick of any
given order. Indeed we show this is already true for the class of groups
that act geometrically on CAT(0) square complexes.

\Restate{Theorem}{thm:thickofanyorder}{(Abundance of cubulated groups that are thick of order $n$).}{
For all $n\geq 0$, there are infinitely many quasi-isometry types of
cocompactly cubulated groups that are algebraically thick (and hence metrically
thick) of order $n$ and have polynomial divergence of order precisely $n+1$.

Furthermore, for any $k\geq 2$, there are infinitely many quasi-isometry types of such groups with the additional condition that the groups act properly and
cocompactly on $k$-dimensional CAT(0) cube complexes.
}

The nature of the construction and
the latter part of the proof are modeled on the construction by Behrstock--Dru\c{t}u~\cite{BD} of CAT(0) groups which are thick of order~$n$ and with polynomial divergence of degree~$n+1$. CAT(0) groups of arbitrary order of polynomial growth were also constructed
recently by Macura~\cite{MacuraDivergence}, who considered iterated HNN
extensions of $\integers^2$. Dani--Thomas recently posted a preprint in which they show that for every integer there exists a Coxeter group whose divergence is polynomial of that degree \cite{DaniThomas:divcox} --- it would be interesting to know if those Coxeter groups are each thick and to compute their simplicial boundaries.

\subsection*{Acknowledgements}
The authors thank Dani Wise for sharing his enthusiasm for cubulated groups!
The authors thank the NSF and the NSERC for funds to
cover travel to the other's home institution to work on this project.
The second author wishes to thank Juliana Nalerio
and Tim Nest for allowing the use of their couch while he visited the first
author in New York.  We are grateful to Tim Susse for carefully
reading an earlier draft and offering several helpful corrections.  We
also thank Michah Sageev, Eric Swenson, and Dani Wise for discussions
relating to an early (unfortunately, unsuccessful) approach we had to
the results in Section~\ref{sec:thickoforderone}; we remain optimistic
those ideas will bear fruit in future work.  We finally thank the
referee for their careful reading and many helpful comments and corrections.

\section{Preliminaries}\label{sec:prelim}
The summary of thick metric spaces and groups given in Section~\ref{sec:thickprelim}
is based on the discussion in~\cite{BDM}.  Section~\ref{sec:cubecomplex}
provides a brief review of CAT(0) cube complexes, and
Section~\ref{sec:divergence} recalls some facts about divergence.

\subsection{Thick spaces and groups}\label{sec:thickprelim}

\subsubsection{Asymptotic cones}
Let $(M, d)$ be a metric space and let $\omega\subset 2^{\naturals}$ be
an ultrafilter on $\naturals$.  Given a sequence
$m=(m_n\in M)_{n\in\naturals}$ of \emph{observation points} and a positive
sequence $s=(s_n)_{n\in\naturals}$ with
$s_n\stackrel{n}{\longrightarrow}\infty$,
the \emph{asymptotic cone} $\ascone{M}{m}{s}{\omega}$ is the ultralimit of the
based metric spaces
$\lim_{\omega}(M,m_n,\frac{d}{s_n})$.  More precisely, define a
pseudometric $\mathbf d_{\omega}$ on $\prod_n M$ by
letting $\mathbf d_{\omega}(y,z)=\lim_{\omega}\frac{d(y_n,z_n)}{s_n}$,
and consider the induced pseudometric on the component containing $m$, i.e.,
\[\widehat M=\left\{(y_n)_{n\in\naturals}\in\prod_n(M,\frac{d}{d_n})\,:\,\mathbf
d_{\omega}(y_n,m_n)<\infty\right\}.\]
Then $\ascone{M}{m}{s}{\omega}$ is the associated quotient metric space,
obtained from $\widehat M$ by identifying points $y$ and $z$ for which $\mathbf
d_{\omega}(y,z)=0$.  A priori, $\ascone{M}{m}{s}{\omega}$ depends on the
observation point $m$, the sequence $s$, and the ultrafilter $\omega$.

When $M$ admits an isometric action by a group $G$ such that some bounded subset
of $M$ meets every $G$-orbit, then $\ascone{M}{m}{s}{\omega}$ is independent of
the choice of observation point $m$, and it suffices to consider
$\ascone{M}{m}{s}{\omega}$, where, for some fixed basepoint $m_o$, the
observation point $m_n=m_o$ for all $n\in\naturals$.  In most of our
applications, $M$ comes equipped with a geometric group action, and thus the
asymptotic cone is independent of the choice of observation point.

\subsubsection{Unconstricted spaces and groups}
A point $c\in M$ is a \emph{cut-point} if $M-\{c\}$ has at least two connected
components.  By convention, $c$ is a cut-point of the space
$\{c\}$.

\begin{defn}[Unconstricted space, wide space]\label{defn:unconstricted}
The metric space $(M,d)$ is \emph{unconstricted} if it satisfies each of
the following:
\begin{enumerate}
\item There exists $\kappa<\infty$ such that for all $m\in M$, there exists a quasi-isometric embedding $\gamma\co\reals\rightarrow M$ such that
$d(m,\gamma)<\kappa$.
\item There exists an ultrafilter $\omega$ and a sequence $s$ such that for any
sequence $m$ of observation points in $M$, there is no cut-point in
$\ascone{M}{m}{s}{\omega}$.
\end{enumerate}
If for all ultrafilters $\omega$, all sequences $m$ of observation points, and
all scaling sequences $d$, there is no cut-point in $\ascone{M}{m}{s}{\omega}$,
then $M$ is \emph{wide}.
\end{defn}

\begin{rem}[Unconstricted group, wide group]\label{rem:unconstrictedgroup}
Let the infinite finitely-generated group $G$ act properly and cocompactly by
isometries on $(M,d)$.  It is easy to see
that Definition~\ref{defn:unconstricted}.(1) holds for $M$.  Moreover, since
$\ascone{M}{m}{s}{\omega}$ is independent of $m$,
Definition~\ref{defn:unconstricted}.(2) is satisfied exactly when at least one
asymptotic cone of $M$ does not have a cut-point.  In particular, letting $M$ be
a Cayley graph of $G$ and $d$ the associated word-metric yields the
notion of an \emph{unconstricted group} and of a \emph{wide} group.
\end{rem}

The inductive definition of a thick metric space requires the notion of a
\emph{uniformly unconstricted} family of spaces.

\begin{defn}[Uniformly unconstricted, uniformly
wide]\label{defn:uniformlyunconstricted}
The collection $(M_n,d_n)_{n\in\naturals}$ of metric spaces is
\emph{uniformly unconstricted} if there exists an ultrafilter $\omega$ and a
sequence $(s_n)_{n\in\naturals}$ of scaling constants such that, for all
observation points $m=(m_n\in M_n)_{n\in\naturals}$, the ultralimit
$\lim_{\omega}(M_n,m_n,\frac{ d_n}{s_n})$ has no cut-point.  If this
ultralimit lacks cut-points for all choices of ultrafilter, observation points,
and scaling constants, then $(M_n)_{n\in\naturals}$ is \emph{uniformly wide}.
\end{defn}

\begin{defn}[Thick space, strongly thick space]\label{defn:thickspace}
The space $(M, d)$ is \emph{thick of order 0} if it is unconstricted, and
\emph{strongly thick of order 0} if it is wide.  Let $\mathcal S$ be a
collection of subsets of $M$ which are each
(strongly) thick of order at most~$n$.  Then
$M$ is \emph{$\tau$-thick ($(\tau,\eta)$-strongly thick) of order at most
$n+1$ with
respect to $\mathcal S$} if there exists $\tau,\eta\geq 0$ such that each of the
following holds:
\begin{enumerate}
\item For each $m\in M$, there exists $S\in\mathcal S$ with $d(m,S)\leq\tau$.
\item Each $S\in\mathcal S$ is $\tau$-quasiconvex in $M$, i.e., any two points
in $S$ can be connected by a $(\tau,\tau)$-quasigeodesic in $N_{\tau}(S)$.
\item For all $S,S'\in\mathcal S$, there exists a sequence
$$S=S_0,S_1,\ldots,S_k=S', \mbox{ with } S_i\in\mathcal S$$ such that
for all $0\leq
i<k$, the subspace $\mathcal N_{\tau}(S_i\cap S_{i+1})$
is of infinite diameter and $\tau$-path-connected.  Strong thickness requires a strengthening of this condition, namely that for any $S,S'$ that both intersect $\mathcal N_{3\tau}(x)$ for some $x\in M$, the preceding sequence can always be chosen so that $k\leq\eta$ and $x\in\mathcal N_{\eta}(S_i)$ for $0\leq i\leq k$.
\end{enumerate}

Further, we say a family of metric spaces $\mathcal M$ is \emph{uniformly thick
(uniformly strongly thick) of order at most $n+1$} if it satisfies:
\begin{enumeratecontinue}
\item
\begin{enumerate}
\item There exists constants $\tau$ and $\eta$ as above such that each
$M\in\mathcal M$ is $\tau$-thick
($(\tau,\eta)$-strongly thick) of order at most $n+1$ with respect to a
collection, $\mathcal S_M$, of subsets of $M$.
\item $\bigcup_{M\in\mathcal M}\mathcal S_M$ is uniformly thick (uniformly
strongly thick) of order at most~$n$.
\end{enumerate}
\end{enumeratecontinue}

We typically drop the constants $\tau$ and $\eta$ from the notation, as the
precise constants are rarely of interest; is is usually important only that some constants exist.

If $M$ is $(\tau,\eta)$-(strongly) thick of order at most $n$ and is not
$(\tau',\eta')$-(strongly) thick of order at most $n-1$
for any $\tau',\eta'$, then $M$ is \emph{(strongly) thick of order~$n$}.
\end{defn}

Following~\cite{BDM} and~\cite{BD}, we define \emph{algebraic thickness} and \emph{strong algebraic thickness} of a group as follows.

\begin{defn}[Algebraically thick]\label{defn:algebraicallythick}
The finitely generated group $G$ is \emph{algebraically thick of
order~0} if it is unconstricted.  For $n\geq 1$, the group $G$ is
\emph{algebraically thick of order at most $n+1$} if there exists a
finite collection $\mathcal G$ of finitely generated undistorted
subgroups of $G$ such that:
\begin{enumerate}
 \item There exists a finite index subgroup $G'\leq G$ generated by a finite
subset of $\cup_{H\in\mathcal G}H$.

 \item Each $H\in\mathcal G$ is algebraically thick of order at most $n$.

 \item For all $H,H'\in\mathcal G$, there exists a finite sequence
$H=H_1,\ldots,H_m=H'$ such that each $H_i\in\mathcal G$ and $H_i\cap H_{i+1}$
is infinite for $1\leq i\leq m-1$.
\end{enumerate}
If $G$ is algebraically thick of order at most $n+1$ and is not algebraically
thick of order at most $n$, then $G$ is \emph{algebraically thick of
order $n+1$}.
\end{defn}

\begin{defn}[Strongly algebraically thick]\label{defn:stronglyalgebraicallythick}
    The finitely generated group $\Gamma$ is
    \emph{strongly algebraically thick of order 0} if it is wide.
    For $n\geq 1$, the group $\Gamma$ is \emph{strongly algebraically
thick of order at most $n+1$} if there exists a finite collection
$\mathbb G$ of finitely generated undistorted subgroups of $G$ such
that:
\begin{enumerate}
 \item There exists a finite index subgroup $\Gamma'\leq \Gamma$ generated by a finite
subset of $\cup_{H\in\mathbb G}H$.

 \item Each $H\in\mathbb G$ is strongly algebraically thick of order at most $n$.

 \item For $H,H'\in\mathbb G$, there exists a sequence
 $H=H_0,\ldots,H_n=H'$ such that $H_i\in\mathbb G$ for each $i$, and
 $H_i\cap H_{i+1}$ is infinite and $M$-path-connected for $0\leq i<n$.

 \item There exists $M\geq 0$ such that each $H\in\mathbb G$ is $M$-quasiconvex.
\end{enumerate}
 If $\Gamma$ is strongly algebraically thick of order at most $n+1$,
 but is not strongly algebraically thick of order $n$, then $\Gamma$
 is strongly algebraically thick of order $n+1$.
\end{defn}

Note that if $\Gamma$ is strongly algebraically thick of order $n$,
then $\Gamma$ is algebraically thick of order at most $n$.

\subsection{Divergence}\label{sec:divergence}
The notion of the divergence function of a metric space goes back to
Gromov and
Gersten~\cite{Gromov:Asymptotic,GerstenDivergence,GerstenDivergence2}; the
present summary
follows~\cite{BD}.

\begin{defn}[Divergence]\label{defn:divergence}
Let $(M,d)$ be a geodesic metric space and fix $\lambda\in(0,1),\mu\geq
0$.  For $a,b,c\in M$, with $d(c,\{a,b\})=r>0$, let
$\divergence{a}{b}{c}{\lambda}{\mu}$ to be the infimum of the set $\{|P|\}$,
where $P$ varies over all paths in $M$ that join $a$ to $b$ and satisfy
$d(P(t),c)\geq\lambda r-\mu$ for all $t$.

The \emph{divergence}
$\Divergence{M}{\lambda}{\mu}:\naturals\rightarrow\reals^+$ of $M$ with respect
to $\lambda,\mu$ is defined by
\[\Divergence{M}{\lambda}{\mu}(n)=\sup\{\divergence{a}{b}{c}{\lambda}{\mu}\,:\,
 d(a,b)\leq n\}.\]
\end{defn}

For any function $f\colon\naturals\rightarrow\reals^+$, the space $M$ has
divergence
at most $f$ if for some $\lambda,\mu$, and for all $n\in\naturals$, we have
$\Divergence{M}{\lambda}{\mu}(n)\leq f(n)$, and the notion of a space with
divergence at least $f$ is defined analogously.  As usual, for functions $f,g$,
we write $f\preceq g$ if for all $n$, we have $f(n)\leq Kg(Kn+K)+K$ for some
constant $K$, and $f\asymp g$ if $f\preceq g$ and $g\preceq f$.  For $d\geq 1$,
the space $M$ has divergence \emph{of order at most $d$} if
$\Divergence{M}{\lambda}{\mu}\preceq p$ for some $\lambda,\mu$, where $p$ is a
polynomial of degree $d$, and \emph{order $d$} if it has divergence of order at
most $d$ but does not have divergence of order at most $d-1$.

There are several alternative notions of divergence discussed
in~\cite[Section~3]{BD}.  In the situations of interest in this paper, $M$
admits a proper, cocompact group action and thus the various divergence
functions coincide up to $\asymp$, by~\cite[Corollary~3.2]{DrutuMozesSapir}.
Further, under the hypotheses of~\cite[Corollary~3.2]{DrutuMozesSapir},
the~$\asymp$-class of the divergence of $M$ is a quasi-isometry invariant, in
the following sense: if $q\co M\rightarrow M'$ is a quasi-isometry, then for some
$\lambda,\lambda'\in(0,1),\mu,\mu'\geq 0$, we have
$\Divergence{M}{\lambda}{\mu}\asymp\Divergence{M}{\lambda'}{\mu'}$, and in
particular the divergence order of $M$ (if it exists) is a quasi-isometry
invariant.  Hence the divergence of a finitely-generated group is well-defined,
and it is sensible to speak of groups with linear, quadratic, exponential, etc.
divergence.

In this paper, we study divergence of cocompactly cubulated groups by studying
thickness of cube complexes.  The relationship between the thickness order and
the divergence order of $M$ is not yet fully understood (see, e.g.~\cite[Question~1.2]{BD}).
One useful result that is established is the following, which we will use in
Section~\ref{sec:arbitrary}, in conjunction with lower bounds on divergence for some cocompactly cubulated groups, in order to provide lower bounds on the order of thickness.

\begin{prop}[Corollary~4.17 of~\cite{BD}]\label{prop:BDdivthick}
Let $M$ be a geodesic metric space that is strongly thick of order at most $n$.
Then
\[\Divergence{M}{\lambda}{\mu}(r)\preceq r^{n+1}\]
for all $\lambda\in(0,\frac{1}{54}),\mu\geq 0$.
\end{prop}

\subsection{CAT(0) cube complexes}\label{sec:cubecomplex}

\subsubsection{Cube complexes and hyperplanes}
A \emph{cube complex} $\mathbf X$ is a CW-complex whose cells are
Euclidean unit cubes of the form $[-\frac{1}{2},\frac{1}{2}]^d$ for $0\leq d<\infty$,
attached in such a way that any two cubes (not necessarily distinct) of $\mathbf
X$ with nonempty intersection intersect in a common face.  The \emph{dimension}
$\dimension\mathbf X$ is the supremum of the set of $d\geq 0$ for which $\mathbf
X$ contains a $d$-cube.

$\mathbf X$ is \emph{nonpositively-curved} if for each $x\in \mathbf X^{(0)}$,
the link
of $x$ is a simplicial flag complex, and \emph{CAT(0)} if it is
nonpositively-curved and simply connected.  As observed by Gromov
in~\cite{Gromov87} and in full generality by Leary~\cite{LearyInfiniteCubes},
the CAT(0) cube complex $\mathbf X$ is endowed with a CAT(0) geodesic metric,
denoted $\dtwo$, obtained by regarding each cube as a Euclidean unit cube (see also the more general results of Bridson and Moussong on the existence of CAT(0) metrics for many polyhedral complexes~\cite{BridsonThesis,MoussongThesis}).  It
is
often convenient to view the 1-cubes as unit intervals and use the combinatorial
metric $\done$ on the graph $\mathbf X^{(1)}$.

These two geometries essentially
agree when $\dimension\mathbf X<\infty$ in the sense that $(\mathbf X,\dtwo)$ is quasi-isometric to $(\mathbf X^{(1)},\done)$.  The metric $\done$ is determined by hyperplanes, as explained below, and these hyperplanes can be used to provide a nice characterization of isometric embeddedness and convexity of subcomplexes.  Since we are concerned with finite-dimensional cube complexes, we use whichever metric is most convenient in a given situation.

For $d\geq 1$, the $d$-cube $c$ has $d$ \emph{midcubes}, which are subspaces
obtained by restricting exactly one coordinate to 0.  A \emph{hyperplane} $H$ of
the CAT(0) cube complex $\mathbf X$ is a connected subspace such that for each cube $c$ of
$\mathbf X$, either $H\cap c=\emptyset$, or $H\cap c$ is a midcube of $c$.  The
\emph{carrier} $N(H)$ of $H$ is the union of all closed cubes $c$ for which
$H\cap c\neq\emptyset$.  Each hyperplane $H$ is itself a CAT(0) cube complex of
dimension at most $\dimension\mathbf X-1$, and $N(H)$ is a CAT(0) cube complex
isomorphic to $H\times[-\frac{1}{2},\frac{1}{2}]$.  Furthermore, $H$ and $N(H)$
are convex with respect to $\dtwo$, and $N(H)^{(1)}$ is convex in
$\mathbf X^{(1)}$, with respect to $\done$
(see~\cite{Chepoi2000,Sageev95}).

Crucially, Sageev showed in~\cite{Sageev95} that, for each hyperplane $H$ of
$\mathbf X$, the complement $\mathbf X-H$ has exactly two components, called
\emph{halfspaces (associated to $H$)} and denoted $\lefth H,\righth H$.  We
denote by $\mathcal H$ the set of hyperplanes in $\mathbf X$ and by $\lrhalf H$
the set of halfspaces.  If $A,B\subset\mathbf X$, then $H\in\mathcal H$
\emph{separates} $A$ and $B$ if $A\subset\lefth H$ and $B\subset\righth H$ or
vice versa.

For each 1-cube $c$ of $\mathbf X$, there is a unique hyperplane $H$ that
separates the endpoints of $c$.  $H$ is the hyperplane \emph{dual} to $c$, and
$c$ is a 1-cube \emph{dual} to $H$.  It can be shown that a path
$P\rightarrow\mathbf X^{(1)}$ is a $\done$-geodesic if and only if $P$
contains at most one 1-cube dual to each $H\in\mathcal H$.  Hence, for
$x,y\in\mathbf X^{(0)}$, the number of hyperplanes separating $x$ from $y$ is
exactly $\done(x,y)$.  Usefully, it is also true that a path
$P\rightarrow\mathbf X$ is an $\dtwo$--geodesic only if for each
$K\in\lrhalf H$, the intersection $P\cap K$ is connected.

Distinct $H_1,H_2\in\mathcal H$ \emph{contact} if $N(H_1)\cap
N(H_2)\neq\emptyset$ (equivalently, no third hyperplane separates $H_1$ from
$H_2$). This can happen in one of two ways: if $H_1\cap H_2\neq\emptyset$, then
$H_1$ and $H_2$ \emph{cross}.  Crossing is also characterized by the fact that
$\lefth H_1\cap\lefth H_2\neq\emptyset,\righth H_1\cap\lefth
H_2\neq\emptyset,\righth H_1\cap\righth H_2\neq\emptyset,\lefth H_1\cap\righth
H_2\neq\emptyset$, and by the fact that $N(H_1)\cap N(H_2)$ contains a 2-cube
whose 1-cubes are dual to $H_1$ or $H_2$.  If $H_1$ and $H_2$ contact and do not
cross, then they \emph{osculate}.

More generally, if $A\subset\mathbf X$ is a connected subspace and $H\in\mathcal
H$, then $H$ \emph{crosses} $A$ if $A\cap\righth H$ and $A\cap\lefth H$ are both
nonempty.  We denote by $\mathcal H(A)$ the set of hyperplanes crossing $A$.  A
connected full subcomplex $\mathbf Y\subseteq\mathbf X$ is \emph{isometrically
embedded} if the inclusion $\mathbf Y^{(1)}\rightarrow\mathbf X^{(1)}$ is an
isometric embedding.  Equivalently, $H\cap\mathbf Y$ is connected for each
$H\in\mathcal H(\mathbf Y)$.  Similarly, $\mathbf Y$ is \emph{convex} if, for any collection $H_1,\ldots,H_n\in\mathcal H(\mathbf Y)$ of pairwise-crossing hyperplanes, $\mathbf Y$ contains an $n$-cube of $\cap_{i=1}^nN(H_i)$.  This notion turns out to coincide with CAT(0)--convexity for subcomplexes~\cite{HaglundSemisimple}; it also equivalent to the requirement that $\mathbf Y^{(1)}$ be a convex subgraph of $\mathbf X^{(1)}$ and every cube of $\mathbf X$ whose 1-skeleton lies in $\mathbf Y$ itself lies in $\mathbf Y$.

\subsubsection{Actions on cube complexes}
By $\Aut(\mathbf X)$, we mean the group of cubical automorphisms of the CAT(0)
cube complex $\mathbf X$, and by an action of the group $G$ on $\mathbf X$, we
mean a homomorphism $G\rightarrow\Aut(\mathbf X)$.  Such an action is also an
action by $\done$-isometries on $\mathbf X^{(1)}$ and by $\dtwo$-isometries on
$\mathbf X$.

This action is \emph{proper}
if the stabilizer of any cube of $\mathbf X$ is finite, and \emph{metrically
proper} if for all infinite sequences $(g_n\in G)_{n\geq 0}$ of distinct
elements, and for all $x\in\mathbf X$, we have $\dtwo(x,g_nx)\rightarrow\infty$
as $n\rightarrow\infty$.  Generally, we are
concerned with cocompact actions, and in this situation the notions of
properness and metric properness coincide.  A proper action of $G$ on a CAT(0)
cube complex is a \emph{cubulation} of $G$, and if such an action exists, $G$ is
\emph{cubulated}.  If $G$ acts geometrically on a CAT(0) cube complex, then $G$
is \emph{cocompactly cubulated}.

Each $g\in\Aut(\mathbf X)$ acts as an isometry of both the CAT(0) space
$(\mathbf X,\dtwo)$ and the median graph $(\mathbf X^{(1)},\done)$.  According
to~\cite{HaglundSemisimple}, either $g$ fixes the barycenter
of a cube of $\mathbf X$, or there exists a combinatorial geodesic
$\gamma\co\reals\rightarrow\mathbf X$ and some $N=N(\dimension\mathbf X),\tau>0$ such that
$g^N\gamma(t)=\gamma(t+\tau)$ for all $t\in\reals$; such an element $g^N$ is
\emph{combinatorially hyperbolic} and $\gamma$ is a \emph{combinatorial axis for
$g^N$}.  Likewise, if $g$ does not fix a point of $\mathbf X$, then since
isometries of CAT(0) spaces are semisimple, $g$ acts by translations on a CAT(0)
geodesic $\alpha\co\reals\rightarrow\mathbf X$, called an \emph{axis} for $g$.  If
$\gamma$ is combinatorially rank-one (equivalently, $\alpha$ is rank-one) for
some combinatorial axis $\gamma$ (CAT(0) axis $\alpha$), then $g$ is a
\emph{rank-one isometry}.

The hyperplane $H\in\mathcal H$ is a \emph{leaf} if at least one of $\lefth
H,\righth H$ fails to contain a hyperplane.  $\mathbf X$ is \emph{essential} if
it contains no leaves.  If $G$ acts on $\mathbf X$, then $H$ is a
\emph{$G$-leaf} if there exists $r\geq 0$ such that, for $A\in\{\lefth H,\righth
H\}$ and for all $x\in\mathbf X$, $\dtwo(gx,H)\leq r$ for all $g\in G$
such that $gx\in A$.  The action of $G$ on $\mathbf X$ is \emph{essential} if
$\mathbf X$ contains no $G$-leaves.  Usually, we will assume that $G$ acts
essentially on $\mathbf X$, abetted by~\cite[Proposition~3.5]{CapraceSageev} and
Lemma~\ref{lem:boundaryofcore} below.  The former says, in particular, that if
$G$ acts geometrically on $\mathbf X$, then there is a convex, $G$-cocompact
subcomplex $\mathbf Y\subseteq\mathbf X$ on which $G$ acts essentially.  The
latter says that the simplicial boundaries of $\mathbf X$ and $\mathbf Y$
coincide.

We will occasionally need some notion of quasiconvexity of subgroups.  Since the groups under consideration are not in general hyperbolic, quasiconvexity of a subgroup depends on the choice of generating set.  However, the groups in this section come equipped with specific geometric actions on metric spaces; accordingly, we use:

\begin{defn}[Quasiconvex]\label{defn:quasiconvex}
Let the group $G$ act properly and cocompactly on the metric space $M$.  The subgroup $H\leq G$ is \emph{quasiconvex} if for some (and hence any) $m\in M$, the orbit $Hm$ is a quasiconvex subspace of $M$.
\end{defn}

This definition is not intrinsic either to $G$ or to $M$, but rather depends on the particular action of $G$ on $M$.  Note, in particular, that this property implies that for any fixed word metric on $G$, there exist uniform constants such that any pair of point in $H$ can be joined by a uniform quality quasigeodesic contained inside a uniform neighborhood of $H$.  This latter, weaker property is the one considered in~\cite{BD}, and it holds for subgroups that are quasiconvex as defined above.

\section{The simplicial boundary}\label{sec:simplicialboundary}
The definition and basic properties of the simplicial boundary of a CAT(0) cube
complex are discussed in~\cite{HagenBoundary}, and we recall these here
briefly, before establishing some simple facts about the simplicial boundary
that will be necessary in subsequent sections.

\subsection{Boundary sets}
Let $\mathbf X$ be a CAT(0) cube complex and suppose that the set $\mathcal H$
of hyperplanes contains no infinite set of pairwise-crossing hyperplanes.  This
holds for all cube complexes in this paper, since they are finite-dimensional by virtue of cocompactness.

\begin{defn}[Closed under separation]\label{defn:closedundersep}
$\mathcal U\subseteq\mathcal H$ is \emph{closed under separation} if for all
$H_1,H_2\in\mathcal U$, if some hyperplane $H_3$ separates $H_1$ from $H_2$,
then $H_3\in\mathcal U$.
\end{defn}

For example, if $A\subset\mathbf X$ is a connected subspace, then $\mathcal
H(A)$ is closed under separation.

\begin{defn}[Unidirectional]\label{defn:unidirectional}
$\mathcal U\subseteq\mathcal H$ is \emph{unidirectional} if for each
$H\in\mathcal U$, at most one of $\lefth H$ or $\righth H$ contains infinitely
many elements of $\mathcal U$.
\end{defn}

The motivating example of a set that is \emph{not} unidirectional is the set
$\mathcal H(\gamma)$, where $\gamma$ is a bi-infinite combinatorial geodesic in
a CAT(0) cube complex $\mathbf X$ in which every set of pairwise-crossing
hyperplanes is finite.

\begin{defn}[Facing triple]\label{defn:facingtriple}
A \emph{facing triple} $\{H_1,H_2,H_3\}\subseteq\mathcal H$ is a set of three
distinct hyperplanes, any two of which are contained in a single halfspace
associated to the third.  Equivalently, $\{H_1,H_2,H_3\}$ is a facing triple if
no three of the associated halfspaces are totally ordered by inclusion.
\end{defn}

\begin{defn}[Boundary set, boundary set equivalence]\label{defn:BS}
$\mathcal U\subseteq\mathcal H$ is a \emph{boundary set} if $\mathcal U$ is
infinite, unidirectional, closed under separation, and contains no facing
triple.

Let $\mathcal U_1,\mathcal U_2$ be boundary sets.  Then $\mathcal
U_1\lesssim\mathcal U_2$ if $|\mathcal U_1-\mathcal U_1\cap\mathcal
U_2|<\infty$.  If $\mathcal U_1\lesssim\mathcal U_2$ and $\mathcal
U_2\lesssim\mathcal U_1$, i.e., if $|\mathcal U_1\triangle\mathcal U_2|<\infty$,
then $\mathcal U_1$ and $\mathcal U_2$ are \emph{equivalent} boundary sets,
denoted $\mathcal U_1\sim\mathcal U_2$.  The boundary set $\mathcal U$ is
\emph{minimal} if for each boundary set $\mathcal U'$ with $\mathcal
U'\lesssim\mathcal U$, we have $\mathcal U'\sim\mathcal U$.
\end{defn}

The following lemma from~\cite{HagenBoundary} explains why we assume that sets
of pairwise-crossing hyperplanes are finite:

\begin{lem}\label{lem:containsminimal}
Any boundary set in $\mathcal H$ contains a minimal boundary set.
\end{lem}

Indeed, an infinite set of pairwise-crossing hyperplanes is, by definition, a
boundary set, but such a set is easily seen to fail to contain a minimal
boundary set.  Lemma~\ref{lem:containsminimal} is needed to prove
Proposition~\ref{prop:decomposition} (which is
\cite[Proposition~3.10]{HagenBoundary}), and this statement is in turn required
when defining the simplicial boundary.

\begin{prop}\label{prop:decomposition}
Let $\mathcal U$ be a boundary set.  Then there exists $k\leq\dimension\mathbf X$ and
pairwise-disjoint minimal boundary sets $\mathcal U_1,\ldots,\mathcal U_k$ such
that
$\bigsqcup_{i=1}^k\mathcal U_i\sim\mathcal U$ and, for each $1\leq i<j\leq k$
and each $U\in\mathcal U_j$, the set of $V\in\mathcal U_i$ such that $U\cap
V=\emptyset$ is finite.

Moreover, if $\mathcal U_1',\ldots\mathcal U'_{k'}$ are pairwise-disjoint
minimal boundary sets such that $\bigsqcup_{i=1}^{k'}\mathcal U_i'\sim\mathcal
U$, then $k=k'$ and, after relabeling, $\mathcal U_i\sim\mathcal U'_i$ for all
$i$.
\end{prop}

\subsection{Simplices at infinity}
The \emph{dimension} of the boundary set $\mathcal U$ is equal to $k-1$, where
$k$ is the number of minimal boundary sets in the decomposition of $\mathcal U$
given by Proposition~\ref{prop:decomposition}.  In particular, the minimal
boundary sets are exactly those that have dimension 0, and the dimension of any
boundary set is finite, by Proposition~\ref{prop:decomposition}, since $\mathbf
X$ has no infinite set of pairwise-crossing hyperplanes.  Note also that if
$\mathcal U\sim\mathcal U'$, then their dimensions coincide.  Accordingly, for
each $k\geq 0$, let $\mathfrak S(k)$ be the set of $\sim$-classes $u$ such that
some (and hence every) representative $\mathcal U$ of $u$ is a $k$-dimensional
boundary set.

\begin{defn}[Simplicial boundary]\label{defn:simplicialboundary}
Let $\mathbf X$ be a CAT(0) cube complex with no infinite set of
pairwise-crossing hyperplanes.  The \emph{simplicial boundary} $\simp\mathbf X$
of $\mathbf X$ is the simplicial complex whose set of $k$-simplices is
$\mathfrak S(k)$, for $k\geq 0$, with the simplex $u$ (represented by a boundary
set $\mathcal U$) a face of $v$ (represented by $\mathcal V$) exactly when
$\mathcal U\lesssim\mathcal V$.
\end{defn}

For example, it is easily verified that the simplicial boundary of an infinite
tree is a discrete set, and that the simplicial boundary of the
standard tiling of $\Euclidean^2$ by 2-cubes is a 4-cycle.
In~\cite{HagenBoundary}, it is shown that $\simp\mathbf X$ is a flag complex,
every simplex of $\simp\mathbf X$ is contained in a finite-dimensional maximal
simplex.

\subsection{Visibility and cubical flats}
The motivating example of a boundary set is the set $\mathcal H(\gamma)$ of
hyperplanes that cross the (combinatorial or CAT(0)) geodesic ray
$\gamma$, but there are boundary sets not of this type:
see~\cite[Example~3.17]{HagenBoundary}.  Following this example, a simplex $v$
is called \emph{visible} if there exists a combinatorial geodesic ray $\gamma$
such that $\mathcal H(\gamma)$ represents the $\sim$-class $v$.
By~\cite[Theorem~3.19]{HagenBoundary}), each maximal simplex is visible.  In this paper, $\mathbf X$ is often assumed to be \emph{fully visible}, meaning that each simplex is visible.  We believe the following is plausible and would remove the need for to hypothesis fully visible from several results in this paper, but a proof of this result appears to be tricky.

\begin{conj}\label{conj:cocompactfullyvisible}
Let $\mathbf X$ be a locally finite CAT(0) cube complex for which some
$G\leq\Aut(\mathbf X)$ acts cocompactly.  Then $\mathbf X$ is fully visible.
\end{conj}

We shall occasionally use the fact that full visibility is inherited by convex subcomplexes.

\begin{defn}[Flat, orthant, cubical flat]\label{defn:flat}
For $d\geq 0$, a $d$-\emph{flat} in $\mathbf X$ is the image of an isometric
embedding $\Euclidean^d\rightarrow(\mathbf X,\dtwo)$. An \emph{orthant}
is the image of an isometric embedding $([0,\infty)^d,\mathbf
d_{\Euclidean^d})\rightarrow(\mathbf X,\dtwo)$.  A \emph{cubical flat}
is an isometrically embedded subcomplex $\mathbf F\subseteq\mathbf X$ that is
isomorphic to the standard tiling of $\Euclidean^d$ by unit $d$-cubes for some $d\geq
0$.  A cubical orthant is defined similarly, in terms of the standard tiling of
$[0,\infty)^d$.
\end{defn}

The simplicial boundary of a $d$-dimensional cubical orthant is easily seen to
be a $(d-1)$-simplex, for $d\geq 1$.  Similarly, one checks that the simplicial
boundary of a $d$-dimensional cubical flat is isomorphic to the
$(d-1)$-dimensional \emph{spherical hyperoctahedron} $\mathbb O_d$.  This
simplicial complex is defined as follows: $\mathbb O_1$ consists of a pair of
0-simplices, and for $d\geq 1$, $\mathbb O_d$ is the simplicial join of $\mathbb
O_0$ and $\mathbb O_{d-1}$.  Under the hypothesis of full visibility, the
presence of a $d$-simplex at infinity ensures the presence of an isometric
cubical orthant; likewise, the presence of a hyperoctahedra in the boundary  yields a flat.

\begin{prop}[Theorem~3.23 of~\cite{HagenBoundary}]\label{prop:simplexflat}
Let $\mathbf X$ be fully visible and let $v\subseteq\simp\mathbf X$ be a
simplex.  Then there is a cubical orthant $\mathbf F\subseteq\mathbf X$ with
$\mathcal H(\mathbf F)$ representing $v$.
\end{prop}

It will be necessary to reach conclusions similar to that of
Proposition~\ref{prop:simplexflat}, but in the CAT(0) setting.

\begin{prop}[Simplices yield orthants]\label{prop:simplicesflats}
Let $\mathbf X$ be fully visible, and let $\mathcal V$ be a boundary set of
dimension $d\geq 1$.  Then there exists a $(d+1)$-dimensional orthant $\mathbf
O\subseteq\mathbf X$ such that $\mathcal H(\mathbf O)\sim\mathcal V$.
\end{prop}

\begin{proof}
By Proposition~\ref{prop:simplexflat}, there exists an isometric cubical orthant
$\mathbf C$ in $\mathbf X$ with $\mathcal H(\mathbf C)\sim\mathcal V$.  Let
$v_1,\ldots,v_{d+1}$ be the 0-simplices of $v$.  For $1\leq i\leq d+1$, there is
a combinatorial geodesic ray $\gamma_i$ such that the $\gamma_i$ all have common
basepoint, and $\mathcal H(\mathbf C)=\bigsqcup_i\mathcal H(\gamma_i)$, and for
$i'\neq j$, every $V\in\mathcal H(\gamma_i)$ crosses every $H\in\mathcal
H(\gamma_j)$.  As is shown in~\cite{HagenBoundary}, there exists, for each $i$,
a CAT(0) geodesic ray $\alpha_i$ in $\mathbf X$ with $\alpha_i(0)=\gamma_i(0)$
and $\mathcal H(\gamma_i)=\mathcal H(\alpha_i)$. The preceding crossing property
ensures that $\mathbf X$ contains $\prod_i\alpha_i$, which is
the desired CAT(0) orthant.
\end{proof}

\begin{defn}[Maximal orthant]\label{defn:maxorthant}
The orthant $\mathbf O\subseteq\mathbf X$ is \emph{maximal} if for all orthants
$\mathbf O'$ that coarsely contain $\mathbf O$, $\dimension\mathbf
O'=\dimension\mathbf O$.
\end{defn}

\begin{prop}[Orthants yield simplices]\label{prop:orthantsyieldsimplices}
Let $\mathbf X$ be fully visible and let $\mathbf O\subseteq\mathbf X$ be a
$d$-dimensional maximal orthant or cubical orthant.  Then $\mathcal H(\mathbf
O)$ represents a $(d-1)$-simplex of $\simp\mathbf X$.
\end{prop}

\begin{proof}
Let $\mathcal V=\mathcal H(\mathbf O)$, and let $\mathcal
V=\bigcup_{i=1}^e\mathcal V_i$ be a decomposition into minimal boundary sets
such that, for all $i\neq j$, if $H\in\mathcal V_i$ and $V\in\mathcal V_j$, then
$H$ crosses $V$.  Now, $e\geq d$ since $\mathbf O$ is a $d$-flat.  On the other
hand, the proof of Proposition~\ref{prop:simplexflat} shows that $\mathbf O$ is
contained in an $e$-dimensional cubical orthant, whence $d=e$.  Thus $\mathcal V$ represents a
$(d-1)$-simplex.
\end{proof}

\begin{rem}
The conclusion of Proposition~\ref{prop:orthantsyieldsimplices} fails in the
absence of maximality.  This is roughly because, while an isometric embedding
$\mathbf Y\rightarrow\mathbf X$ induces an embedding of simplicial boundaries,
the image of $\simp\mathbf Y$ may not be a subcomplex if $\mathbf Y$ is not
convex.  For example, consider the geodesic ray $L$ in $\Euclidean^2$ beginning
at $(0,0)$ and containing $(1,1)$.  Let $\mathbf X$ be the standard tiling of
$\Euclidean^2$ by 2-cubes, and let $\mathbf Y$ be a combinatorial geodesic ray
whose 0-cubes are the points $(n,n),(n+1,n)\,n\geq 0$.  No two hyperplanes of $\mathbf
Y$ cross in $\mathbf Y$, so that $\simp\mathbf Y$ is a 0-simplex.  But $\mathcal
H(\mathbf Y)$ determines a 1-simplex of $\simp\mathbf X$.

Proposition~\ref{prop:orthantsyieldsimplices} also requires full visibility.
For example, if $\mathbf X$ is an \emph{eighth-flat}
(see \cite[Example~3.17]{HagenBoundary}), a maximal cubical orthant is
1-dimensional but the set of dual hyperplanes corresponds to a 1-simplex of
$\simp\mathbf X$.
\end{rem}

The following proposition characterizes hyperbolic proper, cocompact CAT(0) cube complexes using $\simp\mathbf X$.  In the fully visible case, the proof is simplified slightly by Proposition~\ref{prop:simplexflat} and Proposition~\ref{prop:orthantsyieldsimplices}.

\begin{prop}\label{prop:disconnectedhyperbolic}
Let the CAT(0) cube complex $\mathbf X$ admit a proper, cocompact group action.  $\simp\mathbf X$ is discrete if and only if $\mathbf X$ (and therefore $\mathbf X^{(1)}$) is hyperbolic.
\end{prop}

\begin{proof}
If $\simp\mathbf X$ consists entirely of isolated 0-simplices, then
$\mathbf X$ cannot contain an isometrically embedded flat of dimension $d\geq
2$: if $\Euclidean^d\cong\mathbf F\rightarrow(\mathbf X,\dtwo)$ is such
an isometric embedding, then the cubical convex hull of $\mathbf F$ contains a boundary set of positive dimension, resulting in a positive-dimensional simplex of $\simp\mathbf X$.  Hence, by the Flat Plane Theorem~\cite{BridsonHaefliger}, $\mathbf X$ is hyperbolic.  Conversely, if $v$ is a $d$-simplex with $d\geq 2$, then the intersection graph of the set of hyperplanes contains arbitrarily large complete bipartite graphs $K_{n,n}$, by the definition of a boundary set, whence $\mathbf X$ is not hyperbolic~\cite{HagenQuasiArb}.
\end{proof}

\subsection{Essential actions and the simplicial boundary}
We will require the following lemma in Section~\ref{sec:thickoforderone}.

\begin{lem}\label{lem:boundaryofcore}
Let the group $G$ act properly and cocompactly on the CAT(0) cube complex
$\mathbf X$. Let $\mathbf X_1\subseteq\mathbf X$ be a convex, $G$-cocompact
subcomplex on which $G$ acts essentially.  Then $\simp\mathbf X\cong\simp\mathbf
X_1$.
\end{lem}

\begin{proof}
By~\cite[Theorem~3.15]{HagenBoundary}, the inclusion $\mathbf
X_1\hookrightarrow\mathbf X$ induces a simplicial embedding $\simp\mathbf
X_1\rightarrow\simp\mathbf X$.  It suffices to show that this map is surjective.
 If not, there exists a 0-simplex $v$ of $\simp\mathbf X$ that does not belong
to the image of $\simp\mathbf X_1$.  This means that $v$ is represented by a
minimal boundary set $\mathcal V$ such that, for all $V\in\mathcal V$, the
intersection $V\cap\mathbf X_1=\emptyset$.  We thus have a sequence of
hyperplanes $\{V_i\in\mathcal V\}_{i\geq 0}$ such that for all $i\geq 1$, we
have $V_i\subset\righth V_{i-1}$ and $\mathbf X_1\subset\lefth V_{i-1}$.  Now,
by
cocompactness, there exists $R<\infty$ such that every point of $\mathbf X$ is
of the form $gx$, where $g\in G$ and $x$ lies in the $R$-neighborhood of some
fundamental domain $K\subset\mathbf X_1$ for the action of $G$ on $\mathbf X_1$.
 For any $j\geq 0$, we can choose $gx\in\righth V_1$ to be separated from $V_1$,
and hence from $\mathbf X_1$,
by at least $j$ of the hyperplanes $V_i$.  This contradicts the fact that $G$
stabilizes any regular neighborhood of $\mathbf X_1$.  Thus the embedding
$\simp\mathbf X_1\rightarrow\simp\mathbf X$ is surjective.
\end{proof}

Lemma~\ref{lem:boundaryofcore} will be used in conjunction
with~\cite[Proposition~3.5]{CapraceSageev} in the following way: if we wish to
make a statement about $\simp\mathbf X$, where $\mathbf X$ admits a proper,
cocompact action, then there is no harm in passing to a convex, cocompact,
essential subcomplex.

\subsection{Limit simplices, limit sets, and the visual boundary}\label{sec:limitvisual}
In this section, $\mathbf X$ is a CAT(0) cube complex admitting a proper, cocompact action by a group $G$.  Let $\visual\mathbf X$ denote the visual boundary of $(\mathbf X,\dtwo)$, endowed with the cone topology.  For a geodesic ray $\gamma\subset\mathbf X$, we denote by $[\gamma]$ the point of $\visual\mathbf X$ represented by $\gamma$.  It is shown in~\cite[Section~3]{HagenBoundary} that, when $\mathbf X$ is fully visible, there is a surjection $f\co\visual\mathbf X\rightarrow\simp\mathbf X$ such that, if $\gamma$ is a CAT(0) geodesic and $u$ is the simplex of $\simp\mathbf X$ represented by $\mathcal H(\gamma)$, then $f([\gamma])\in u$.

In the interest of an explicit, self-contained account, we now describe the map $f\co\visual\mathbf X\rightarrow\simp\mathbf X$ when $\mathbf X$ is a fully visible CAT(0) cube complex admitting a proper, cocompact action by some group $G$.  Fix a base 0-cube $x_o$, and choose for each $[\gamma]\in\visual\mathbf X$ a CAT(0) geodesic ray $\gamma$ representing $[\gamma]$, with $\gamma(0)=x_o$.  Let $u_{[\gamma]}$ be the simplex of $\simp\mathbf X$ represented by $\mathcal H(\gamma)$, which is easily seen to be a boundary set.  Note that if $\gamma'$ fellow-travels with $\gamma$, then $|\mathcal H(\gamma)\triangle\mathcal H(\gamma)|<\infty$, whence $u_{[\gamma]}=u_{[\gamma']}$.  Hence $u_{[\gamma]}$ is well-defined.  Moreover, every simplex $u$ of $\simp\mathbf X$ satisfies $u=u_{[\gamma]}$ for some $\gamma\in\visual\mathbf X$, by full visibility of $\mathbf X$.

If $[\gamma]$ has the property that $\mathcal H(\gamma)$ is a minimal boundary set, then $u_{[\gamma]}$ is a 0-simplex, and we let $f([\gamma])=u_{[\gamma]}$.

Next, let $\gamma$ be a combinatorial geodesic ray with $\gamma(0)=x_o$ and $\mathcal H(\gamma)$ a representative set for a $d$-simplex $u$ of $\simp\mathbf X$, with $d\geq 2$.  By Proposition~\ref{prop:simplicesflats}, there exists an isometrically embedded maximal flat orthant $Y\subset\mathbf X$ with $|\mathcal H(\gamma)-\mathcal H(\gamma)\cap\mathcal H(Y)|<\infty$, so that the cubical convex hull $\widehat Y$ has the property that the inclusion $\widehat Y\rightarrow\mathbf X$ induces the inclusion $\simp Y\cong u\hookrightarrow\simp\mathbf X$.

Choose a geodesic ray $\sigma\subset Y$ such that $\mathcal H(\sigma)$ and $\mathcal H(\gamma)$ have finite symmetric difference, and such that $\sigma(0)$ is the image of the origin under $[0,\infty)^D\cong Y\hookrightarrow\mathbf X$.  Let $\gamma_0,\ldots,\gamma_D$, with $D\geq d$, be a collection of CAT(0) geodesic rays such that $\mathbf Y=\gamma_0\times\ldots\times\gamma_D$, so that $u$ is spanned by the 0-simplices $f([\gamma_0]),\ldots,f([\gamma_D])$.  Then $\sigma$ is determined by a unit vector $(\alpha_i)_{i=0}^D$, where $\alpha_i$ is the projection in $\mathbf Y$ of $\gamma(1)$ to $\gamma_i$.  Let $f([\gamma])=f([\sigma])$ be the point $\sum_{i=0}^D\alpha_if([\gamma_i])$.  Note that this is well-defined: if $\gamma'$ fellow-travels with $\gamma$, then $|\mathcal H(\gamma')\triangle\mathcal H(\sigma)|<\infty$.

The map $f$ is surjective, by construction, and has the additional property that if $\mathcal H(\gamma)$ represents a simplex $u\subset\simp\mathbf X$, then $f([\gamma])\in u$, and if $f([\gamma])\in u$ for some simplex $u$, then $\mathcal H(\gamma)$ represents $u$ or one of its faces.  (A priori, for $f$ to be injective requires that any two geodesic rays representing the same 0-simplex of $\simp\mathbf X$ fellow-travel, and so there are in general many orthants that are coarsely inequivalent but represent the same simplex; each is coarsely equivalent to some orthant in the convex hull of any of them, however.  This explains the failure of $f$ to be injective; see~\cite[Proposition~3.37]{HagenBoundary}.)

\begin{defn}[Limit simplex, limit set]\label{defn:limitsequence}
Let $H\leq G$.  The simplex $a\subseteq\simp\mathbf X$ is a \emph{limit simplex} for the action of $H$ on $\mathbf X$ (and on $\simp\mathbf X$) if for some (and hence any) 0-cube $x\in\mathbf X$, there exists a sequence $(h_i\in H)$ such that the set of hyperplanes $V$ such that $V$ separates $h_ix$ from $x$ for all but finitely many $i$ is a boundary set representing $a$.  The \emph{limit complex} for $H$ is the smallest subcomplex that contains every limit simplex.

A point $p\in\mathbf X\cup\visual\mathbf X$ is in the \emph{limit set} of $H$ if for some (and hence any) $x\in\mathbf X$, there exists $(h_i\in H)_{i\geq 0}$ such that $h_ix$ converges to $p$ in the cone topology.
\end{defn}

The following lemma relates limit sets (which live in the visual boundary) to limit complexes (which live in the simplicial boundary).

\begin{lem}\label{lem:visualboundary}
Let $H\leq G$ and let $\mathbf X$ be finite-dimensional, locally finite, and fully visible, and let $u\subset\simp\mathbf X$ be a simplex.  If $f^{-1}(u)\subset\visual\mathbf X$ is contained in the limit set of $H$, then $u$ is contained in a limit simplex for $H$.
\end{lem}

\begin{proof}
Choose a (combinatorial or CAT(0)) geodesic ray $\gamma$ such that $\mathcal H(\gamma)$ represents the simplex $u$; this is possible since $\mathbf X$ is fully visible.  Since $f^{-1}(u)$ is contained in the limit set of $H$, there is a sequence $(h_i\in H)_{i\geq 0}$ such that $h_ix_o$ converges to $[\gamma]\in\visual\mathbf X$, where $x_o=\gamma(0)$.  Hence there exists $K\geq 0$ such that for all sufficiently large $i$, there exists $n_i$ such that $\ddot d(\gamma(n_i),p_i)\leq K$, where $p_i$ is the projection of $h_ix_o$ onto the (CAT(0)-metric) sphere of radius $n_i$ about $x_o$ (and $\ddot d(h_ix_o,x_o)\geq n_i$).

Let $\mathcal U$ be the set of hyperplanes $W$ such that $W$ separates $x_o$ from $h_ix_o$ for all but finitely many values of $i$.  Write $\mathcal U=\mathcal U_1\sqcup\mathcal U_2$, where $\mathcal U_1$ is the set of hyperplanes in $\mathcal U$ that separate $p_i$ from $x_o$ for all but finitely many $i$.  Since $p_i$ lies on the geodesic from $h_ix_o$ to $x_o$, we note that each $V\in\mathcal U_1$ separating $p_i$ from $x_o$ also separates $h_ix_o$ from $x_o$.

Observe that $|\mathcal U_1-\mathcal H(\gamma)|\leq K$.  Indeed, a hyperplane in $\mathcal U_1-\mathcal H(\gamma)$ must separate $\gamma(n_i)$ from $p_i$ for all sufficiently large $i$.

Conversely, suppose that $\mathcal H(\gamma)-\mathcal U_1$ is infinite.  Each $V\in\mathcal H(\gamma)-\mathcal U_1$ fails to separate $p_i$ from $x_o$ for arbitrarily large values of $i$, while separating $x_o$ from $\gamma(n_j)$ for all but finitely many $j$.  Thus $V$ separates $p_i$ from $\gamma(n_i)$ for arbitrarily large values of $i$.

Suppose $V_1,V_2,\ldots$ are hyperplanes with this property, numbered according to the order in which one encounters them while traveling along $\gamma$.  Let $M$ be the Ramsey number $R(\dimension\mathbf X+1,K+1)$.  Then $\{V_1,\ldots,V_M\}$ contains either $\dimension\mathbf X+1$ pairwise-crossing hyperplanes, which is impossible, or $K+1$ pairwise-disjoint hyperplanes.  In the latter case, renumber so that $V_1,\ldots,V_{K+1}$ are pairwise-disjoint hyperplanes.  Since $\gamma$ is a geodesic, each $V_j$ either separates $p_i$ from $\gamma(n_i)$ for all sufficiently large $i$, or $V_j$ separates $p_i$ from $p_{i'}$ for infinitely many values of $i,i'$.  However, if $V_j,V_{j'}$ are both hyperplanes of the latter type then, since they cannot cross, they separate $p_i,\gamma(n_i)$ for the same values of $i$.  Hence there exists $i$ such that $K+1$ hyperplanes separate $p_i$ from $\gamma(n_i)$, which is impossible.  Hence $|\mathcal H(\gamma)-\mathcal U|<\infty$.

Thus $|\mathcal U_1\triangle\mathcal H(\gamma)|<\infty$, i.e. $\mathcal U_1$ and $\mathcal H(\gamma)$ represent the same simplex $u$ of $\simp\mathbf X$.  Suppose that $V\in\mathcal U_1$ and $W\in\mathcal U_2$.  Then there exists $I$ such that for all $i\geq I$, the points $x_o$ and $h_ix_o$ are separated by $W$, but there are infinitely many $i$ such that $W$ separates $h_ix_o$ from $p_i$.  Hence, since $V$ separates $x_o$ from $h_ix_o$, all but finitely many such $V$ cross $W$.

Hence, if $\mathcal U_2$ is finite, then $u$ is a limit simplex for $H$.  Otherwise, by local finiteness, $\mathcal U_2$ contains a boundary set $\mathcal U'_2$ representing a simplex $v$ of $\simp\mathbf X$ such that $u\star v$ is also a simplex of $\simp\mathbf X$.  By definition, $u\star v$ is a limit simplex for $H$.
\end{proof}

\section{Relatively hyperbolic cubulated groups}\label{sec:relhyp}
Before studying cocompactly cubulated groups that are thick, we consider a
natural class of such groups that are not thick, namely those that are
relatively hyperbolic.  We saw in Proposition~\ref{prop:disconnectedhyperbolic} that if the infinite, finitely generated group $G$ acts properly and cocompactly on the CAT(0) cube complex $\mathbf X$, then $G$ is hyperbolic if and only
if $\simp\mathbf X$ is an infinite set of 0-simplices.  It is natural to ask how this extends to
relatively hyperbolic groups; in this section we shall provide a complete characterization of relatively hyperbolic cocompactly cubulated groups, in terms of the simplicial boundary.

Note that a subset of $\mathbf X^{(0)}$ is quasiconvex in $(\mathbf X,\dtwo)$ if and only if it is quasiconvex in $(\mathbf X^{(1)},\done)$.  Hence in what follows, we sometimes say that $A\subset\mathbf X$ is ``quasiconvex in $\mathbf X$'' to mean that the set of 0-cubes of $A$ is quasiconvex in $\mathbf X^{(0)}$.

\subsection{The simplicial boundary of a relatively hyperbolic cube
complex}\label{sec:relhyptosimp}
Suppose that the group $G$ acts properly and cocompactly on the CAT(0) cube
complex $\mathbf X$, and is hyperbolic relative to a collection $\mathbb P$ of
peripheral subgroups.
Now, each $P\in\mathbb P$ is the stabilizer of a single
vertex in an appropriately chosen fine hyperbolic graph for $(G,\mathbb P)$
(see~\cite{Bowditch97,SageevWiseCore}) and therefore acts on that graph with a quasiconvex orbit.  (The latter condition is called \emph{relative quasiconvexity} in~\cite{SageevWiseCore}.)  By~\cite[Theorem~1.1]{SageevWiseCore},
there exists a convex (and hence CAT(0)) $P$-invariant subcomplex $\mathbf
Y_P\subseteq\mathbf X$.  By~\cite[Theorem~3.15]{HagenBoundary}, the inclusion
$\mathbf Y_P\rightarrow\mathbf X$ induces a simplicial embedding $\simp\mathbf
Y_P\rightarrow\simp\mathbf X$.  Now, if $\mathbf Y,\mathbf Y'$ are convex,
$P$-cocompact subcomplexes, then each lies in a finite neighborhood of the
other, and it follows that $\mathcal H(\mathbf Y)$ and $\mathcal H(\mathbf Y')$
have finite symmetric
difference, so that the images of $\simp\mathbf Y$ and $\simp\mathbf Y'$ in
$\simp\mathbf X$ coincide.  We denote by $\mathcal I$ the set of isolated
0-simplices of $\simp\mathbf X$.

\begin{thm}\label{thm:relhyp}
Let $G$ be hyperbolic relative to a collection $\mathbb P$ of peripheral
subgroups, each of which has infinite index in $G$, and suppose that $G$ acts
properly and cocompactly on the CAT(0) cube

complex $\mathbf X$.  Then $\mathcal I\neq\emptyset$ and $\simp\mathbf
X\cong\mathcal I\cup\left(\bigsqcup_{P\in\mathbb P}\simp\mathbf Y_P\right)$.
\end{thm}

\begin{rem}
Note that $\simp\mathbf Y_P$ may be disconnected, and may contain simplices of
$\mathcal I$.
\end{rem}

\begin{rem}[Metric relative hyperbolicity]
Theorem~\ref{thm:relhyp} holds under more general conditions.  Namely, if $G$
acts properly and cocompactly on a CAT(0) cube complex $\mathbf X$ and there is
a family $\{\mathbf Y_P\}$ of convex subcomplexes such that $\mathbf X=\mathcal
N_{\tau}(\cup_P\mathbf Y_P)$ for some $\tau\geq 0$, no distinct $\mathbf
Y_P,\mathbf Y_Q$ have infinite coarse intersection, and the intersection graph
of the $\tau$-neighborhoods of the $\mathbf Y_P$ is fine and
$\delta$-hyperbolic for some $\delta\geq0$, then $\simp\mathbf X$ decomposes as
in the conclusion of Theorem~\ref{thm:relhyp}.
\end{rem}

\begin{rem}[Limit simplices]
If $a$ is a limit simplex for the action of $P$ on $\mathbf X$, then, fixing $y\in\mathbf Y$, we have a sequence $(p_j\in P)$ such that the set $\mathcal A$ of hyperplanes $H$ that separates $y$ from $p_jy$ for all but finitely many values of $j$ represents $a$.  Each such hyperplane separates two 0-cubes of the $P$-invariant subcomplex $\mathbf Y$, and thus crosses $\mathbf Y$.  Hence $a\subseteq\simp\mathbf Y$.  Thus each $\simp\mathbf Y_P$ contains every limit simplex for the action of $P$ on $\mathbf X$.  This verifies that each hypothesis in Theorem~\ref{thm:relhypconverse} below is necessary.
\end{rem}

\begin{proof}[Proof of Theorem~\ref{thm:relhyp}]
That $\mathcal I\neq\emptyset$ follows from the rank-rigidity
theorem~\cite[Corollary~B]{CapraceSageev} and the fact that the simplex
represented by the boundary set consisting of hyperplanes that cross a sub-ray
of an axis for a rank-one isometry is an isolated 0-simplex.  Otherwise,
$\mathbf X$ decomposes as the product of two unbounded subcomplexes and
$\mathbb P$ consists of $G$ itself.

We first show that, if $P,P'\in\mathbb P$ are distinct, then $\simp\mathbf Y_P$ and $\simp\mathbf Y_{P'}$ have disjoint images in $\simp\mathbf X$.  From this it follows that there is a simplicial embedding $\mathcal I\cup\left(\bigsqcup_{P\in\mathbb P}\simp\mathbf Y_P\right)\hookrightarrow\simp\mathbf X$.

Since $\mathbf Y_P\cap\mathbf Y_{P'}$ is the intersection of convex subcomplexes, it is convex and $P\cap P'$-cocompact, since $\mathbf Y_P$ and $\mathbf Y_{P'}$ are respectively $P$ and $P'$-cocompact.  Since $\mathbb P$ is almost-malnormal, $P\cap P'$ is finite, and $\mathbf Y_P\cap\mathbf Y_{P'}$ is therefore compact and, in particular, crossed by finitely many hyperplanes.  The same is true of the intersection of any uniform neighborhoods of $\mathbf Y_P$ and $\mathbf Y_{P'}$.  In particular, $\mathcal H(\mathbf Y_P)\cap\mathcal H(\mathbf Y_{P'})$ is finite, whence $\simp\mathbf Y_P\cap\simp\mathbf Y_{P'}=\emptyset$, as desired.

Consider a maximal simplex $v$ of $\simp\mathbf X$.  If $v$ is a 0-simplex, then
it belongs to $\mathcal I$, so suppose that the dimension of $v$ is positive.
Let $\mathbf O$ be an orthant in $\mathbf X$ such that $\mathcal H(\mathbf O)$
represents $v$.  It suffices to verify that $\mathbf O$ is coarsely contained in
some $\mathbf Y_P$, for it then follows that $v\subset\simp\mathbf Y_P$ and the
above embedding is surjective.

$\mathbf O$ is a maximal flat orthant, by maximality of $v$, and cannot have infinite
coarse intersection with more than one $\mathbf Y_P$.  Hence either $\mathbf F$
is coarsely contained in some $\mathbf Y_P$, or has finite intersection with
each $\mathbf Y_P$.  The latter case is impossible, since orthants are unconstricted, as shown in Section~\ref{sec:unconstricted}, and hence must lie near a peripheral subset by~\cite{DrutuSapir} and~\cite[Theorem~4.1,~Remark 4.3]{BDM}.  Thus $v$ belongs to a translate of some $\simp\mathbf Y_P$, and the proof is complete.
\end{proof}

When the peripheral subgroups are virtually abelian, we obtain a cubical
analogue of a result of Hruska-Kleiner~\cite[Theorem~1.2.1]{HruskaKleiner} which
states that if $X$ is a CAT(0) space admitting a proper,
cocompact action by a group that is hyperbolic relative to maximal abelian
subgroups, then the Tits boundary of $X$ is isometric to the disjoint union of
isolated points and spheres of various dimensions. This result of Hruska--Kleiner relates to the following:

\begin{cor}\label{cor:relhypabelian}
Let $G$ be hyperbolic relative to a collection $\mathbb P$ of virtually abelian
subgroups of rank at least 2.  Then for any CAT(0) cube complex $\mathbf X$ on
which $G$ acts properly and cocompactly, $\simp\mathbf X$ is the disjoint union
of a discrete set and a set of pairwise-disjoint spherical hyperoctahedra.  If
$G$ is not virtually abelian, each of these sets is infinite.
\end{cor}

\begin{proof}
By Theorem~\ref{thm:relhyp}, $\simp\mathbf X\cong\mathcal
I\sqcup\left(\bigsqcup_P\simp\mathbf Y_P\right)$.  The set of isolated 0-cubes,
and the set of $\simp\mathbf Y_P$, are obviously infinite if $G$ is not
virtually abelian.  For each maximal virtually abelian subgroup $P$, we have
$\simp\mathbf Y_P\cong \mathbb O_d$, where $d\geq 2$ is the rank of $P$,
by~\cite[Theorem~A]{HagenCrystallographicCubes}.  If $\simp\mathbf Y_{P'}$ and
$g\simp\mathbf Y_P$ have nonempty intersection, containing a common simplex $v$,
then $g\mathbf Y_P\cap\mathbf Y_{P'}$ is coarsely unbounded, since it is crossed by every
hyperplane in a boundary set representing $v$.  But then $gPg^{-1}\cap P'$ is
infinite, contradicting almost-malnormality unless $gPg^{-1}=P'$.  In the latter
case, $g\simp\mathbf Y_P=\simp\mathbf Y_{P'}$.  (If $G$ is virtually abelian,
then the above argument shows that $\simp\mathbf X$ is a single
hyperoctahedron.)
\end{proof}

Since each hyperoctahedron can be given a CAT(1) metric, in which simplices are spherical simplices with side length $\frac{\pi}{2}$, making it isometric to a sphere of the appropriate dimension (see Section~3 of~\cite{HagenBoundary}), Corollary~\ref{cor:relhypabelian} provides a new proof of the Hruska-Kleiner result in the CAT(0) cubical case.

\subsection{Peripheral structures from collections of subcomplexes of
$\simp\mathbf X$}\label{sec:simptorelhyp}
Conversely, one can recover a relatively hyperbolic structure on $G$ from a
decomposition of $\simp\mathbf X$ like that in Theorem~\ref{thm:relhyp}.
Suppose $G$ acts properly and cocompactly on the CAT(0) cube complex $\mathbf X$
and, as before, denote by $\mathcal I$ the set of isolated 0-simplices of
$\simp\mathbf X$.

\begin{defn}[Fine graph]\label{defn:fine}
The graph $\Lambda$ is \emph{fine} if for all $n\in\naturals$ and all edges $e$
of $\Lambda$, there are finitely many $n$-cycles in $\Lambda$ that contain $e$.
\end{defn}

\begin{thm}\label{thm:relhypconverse}
For some $k<\infty$, let $\mathbf S_1,\ldots,\mathbf S_k$ be subcomplexes of
$\simp\mathbf X$, with $P_i=\stabilizer(\mathbf S_i)$, and satisfying all of the
following:
\begin{enumerate}
\item $\simp\mathbf X=\mathcal I'\sqcup G\left(\bigsqcup_{i=1}^k\mathbf
S_i\right)$, where $\mathcal I'\subseteq\mathcal I$.
\item For each $i$, the subcomplex $\mathbf S_i$ contains all limit simplices for the action of $P_i$ on $\simp\mathbf X$.  Equivalently, when $\mathbf X$ is fully visible, each $f^{-1}(\mathbf S_i)$ contains the limit set of $P_i$.
\item For all $1\leq i\leq j\leq k$ and $g,h\in G$, we have $g\mathbf S_i\cap
h\mathbf S_j=\emptyset$ unless $i=j$ and $gh^{-1}\in P_i$.
\item Either $k=1$ and $P_1$ is a finite index subgroup of $G$, or each $P_i$ has infinite index in
$G$.
\item Each $P_i$ is quasiconvex.
\end{enumerate}

Then $G$ is hyperbolic relative to a collection $\{Q_i\}_{i=1}^k$ for which
$Q_i$ is commensurable with $P_i$ for each $i\leq k$.
\end{thm}

\begin{proof}
First, we assume that each $\mathbf S_i$ contains at least one
positive-dimensional simplex, for otherwise the hypotheses are satisfied by a
proper subset of $\{\mathbf S_i\}_{i=1}^k$.  If the set of $\mathbf S_i$ is
empty, then $\simp\mathbf X$ consists entirely of isolated 0-simplices whence
$G$ is hyperbolic relative to $\{1\}$ by Proposition~\ref{prop:disconnectedhyperbolic}.

In this proof, we use the metric $\done$ unless stated
otherwise.  Observe also that the hypotheses imply that each positive-dimensional component of $\simp\mathbf X$ is contained in a single $g\mathbf S_i$.

\textbf{Representing $P_i$ in $\mathbf X$:}  Fix a 0-cube $x\in\mathbf X$.  For
$1\leq i\leq k$, let $\mathbf C_i$ be the convex hull of the orbit $P_ix$.  The
subcomplex $\mathbf C_i$ is $P_i$-invariant because $\mathbf C_i$ is the largest subcomplex contained in the
intersection of all halfspaces that contain $P_ix$, the set of which is
obviously $P_i$-invariant.  Thus $P_i\leq\stabilizer_G(\mathbf C_i)$.  Each $P_i$ is quasiconvex in $G$ with respect to the action of $G$ on $\mathbf X^{(1)}$.  Hence the subcomplex $\mathbf C_i$ is contained in a uniform neighborhood of the orbit $P_ix$ and is therefore $P_i$-cocompact.  Let $Q_i=\stabilizer_G(\mathbf C_i)$.  Since $\mathbf C_i$ is contained in a finite neighbourhood of $P_ix$, the groups $P_i$ and $Q_i$ are commensurable.

\textbf{Comparing $\simp\mathbf C_i$, $\mathbf S_i$, and verifying almost-malnormality:}  The inclusion $\mathbf C_i\rightarrow\mathbf X$ induces an inclusion $\simp\mathbf C_i\hookrightarrow\simp\mathbf X$ whose image is a subcomplex.  Now, suppose that $a\subseteq\simp\mathbf C_i$ is a maximal, and therefore visible, simplex, and let $\gamma\rightarrow\mathbf C$ be a combinatorial geodesic ray such that $\mathcal H(\gamma)$ represents $a$.  Since $P_i$ acts cocompactly on $\mathbf C_i$, there exists a sequence $\{p_j\in P_i\}$ such that $\gamma$ lies at finite Hausdorff distance from $\{p_jx\}$, and therefore that the set of hyperplanes $H$ such that $H$ separates $x$ from $p_jx$ for all but finitely many values of $j$ has finite symmetric difference with $\mathcal H(\gamma)$.  Hence $a$ is a limit simplex for the action of $P_i$ on $\mathbf X$.

Under the hypothesis that each $\mathbf S_i$ contains every limit simplex for the action of its stabilizer $P_i$, this shows that $\simp\mathbf C_{i}\subseteq\mathbf S_i$.  Similarly, under the hypothesis that $f^{-1}(\mathbf S_i)$ contains the limit set for the action of $P_i$, this implies that $\simp\mathbf C_i\subseteq\mathbf S_i$.  Hence, if $g,h\in G$, then $g\simp\mathbf C_i\cap h\simp\mathbf C_j=\emptyset$ unless $i=j$ and $gh^{-1}\in P_i$.  This implies that the set of hyperplanes crossing $g\mathbf C_i$ and $h\mathbf C_j$ is finite, whence, for any $R\geq 0$, the intersection of the $R$-neighborhood of $g\mathbf C_i$ with that of $h\mathbf C_j$ is compact.

Let $i,j\leq k$ and $h\in G$, and consider $P_i^h\cap P_j$.  If this intersection is infinite, then $\mathbf C_j\cap
(h\mathbf C_i)$ contain unbounded subsets at finite Hausdorff distance, a contradiction.  Thus $\{P_i\}_{i=1}^k$ is an almost-malnormal collection, and the same is true of $\{Q_i\}$.

\textbf{A Bowditch graph:}  For any $R\in\naturals$, and any convex
subcomplex $Y\subset\mathbf X$, let $\mathfrak K_R(Y)$ be the following convex
subcomplex containing $Y$ with the property that every $x\in\mathfrak
K_R(Y)$ satisfies $\done(x,Y)\leq R$.  Let
$t_R=\frac{R}{\dimension\mathbf X}$ and let $\mathfrak K_R(Y)$ be the
convex hull of the $\done$-neighborhood of $Y$ of radius $t_R$.  Then
$Y\subseteq\mathfrak K_R(Y)$, the latter subcomplex is convex and
contained in the uniform $R$-neighborhood of $Y$ as we now quickly
show.  Any geodesic joining $y\in\mathfrak K_R(Y)$ to a closest point
of $Y$ crosses a set of hyperplanes that cross the $t_R$-neighborhood
of $Y$ but do not cross $Y$.  Further, this set of hyperplanes
contains no facing triple, and each clique has cardinality at most
$\dimension\mathbf X$.  Thus, there are at most $\dimension\mathbf X
t_r=R$ hyperplanes in the set, since otherwise we would have a
contradiction as we would obtain a nested set of more than $t_R$
hyperplanes separating $y$ from $Y$ and crossing $\mathcal
N_{t_R}(Y)$.

Since $G$ acts cocompactly, there exists $R<\infty$ such that
$\bigcup_iG\mathfrak K_R(\mathbf C_i)=\mathbf X$.  Fixing such an $R$, let
$\Gamma$ be the intersection graph of the collection of subspaces $\mathfrak
K_R(\mathbf C_i)$ and all of their translates.  More precisely, $\Gamma$ has a
vertex for each $\mathfrak K_R(g\mathbf C_i)$ and exactly one edge joining
$\mathfrak K_R(g\mathbf C_i)$ to $\mathfrak K_R(h\mathbf C_j)$ if and only if
$g\mathbf C_i\neq h\mathbf C_j$ and $\mathfrak K_R(g\mathbf C_i)\cap \mathfrak
K_R(h\mathbf C_j)\neq\emptyset$.

Since $\mathbf S_i\cap\mathbf S_j=\emptyset$ for $i\neq j$, and $\mathbf X$ is
locally finite, $\mathbf C_i\cap\mathbf C_j$ is compact, and in particular is
crossed by finitely many hyperplanes.  More strongly, the set of hyperplanes
that crosses both $\mathbf C_i$ and $\mathbf C_j$ is finite, since otherwise
$\mathcal H(\mathbf C_i)\cap\mathcal H(\mathbf C_j)$ would contain a boundary
set.  Hence finitely many hyperplanes cross $\mathfrak K_R(\mathbf C_i)\cap
\mathfrak K_R(\mathbf C_j)$, and therefore there exists a compact
convex subcomplex $B$ such that for all $g,h\in G,\,1\leq i,j\leq k$ there
exists $a\in G$ such that $\mathfrak K_R(g\mathbf C_i)\cap \mathfrak
K_R(h\mathbf C_j)\subset aB$.

By construction, $G$ acts by isometries on $\Gamma$, in such a way that the set
of vertex stabilizers is exactly the set of subgroups $Q_i$ and their
conjugates.

\textbf{Edge-stabilizers:}  Almost-malnormality of $\{Q_i\}_i$ implies that the
stabilizers of edges in $\Gamma$ are finite.

\textbf{Cofiniteness:}  There are finitely many $G$-orbits of edges in $\Gamma$.
 To see this, first observe that each $P_i$ acts cocompactly on $\mathfrak
K_R(\mathbf C_i)$.  Also, there are
clearly finitely many $G$-orbits of vertices in $\Gamma$: one for each $\mathbf
C_i$ with $1\leq i\leq k$.

For each vertex $\mathfrak v$ of $\Gamma$ (corresponding to some translate of
some $\mathfrak K_R(\mathbf C_i)$), let $E(\mathfrak v)=\{e_1,\ldots,e_q\}$ be a
set of edges of $\Gamma$ incident to $\mathfrak v$, containing exactly one edge
from each $\stabilizer_G(\mathfrak v)$-orbit.  This set is finite since $\stabilizer(\mathfrak v)$ acts cocompactly on $\mathbf C_i$.  Let $\{\mathfrak
v_1,\ldots,\mathfrak v_k\}$ contain exactly one vertex of $\Gamma$ from each
$G$-orbit.  If $\mathfrak v$ is a vertex and $e$ an incident edge, then
$(\mathfrak v, e)=(g\mathfrak v_i,gpg^{-1}e_j)$, where $g^{-1}e_j\in E(\mathfrak
v_i)$, and $g\in G$, and $p\in\stabilizer_G(\mathfrak v)$.  Thus $(\mathfrak v,
e)=g(\mathfrak v_i,pg^{-1}e_j)$ is a translate of one of the finitely many pairs
$(\mathfrak v_i,e_j)$.  Hence there are finitely many $G$-orbits of edges in
$\Gamma$.

\textbf{Conclusion:} Below we prove $\Gamma$ is fine in Lemma~\ref{lem:fine} and hyperbolic in
Lemma~\ref{lem:qi}.  Accordingly, the action of $G$ on $\Gamma$ satisfies all of the
conditions of~\cite[Definition~2]{Bowditch97} and $G$ is therefore hyperbolic
relative to $\{Q_i\}_{i=1}^k$.
\end{proof}

\begin{lem}\label{lem:fine}
$\Gamma$ is fine.
\end{lem}

\begin{proof}
Since $\Gamma$ contains no loops or bigons, every cycle has length at least 3.

\textbf{3-cycles:}  Let $A_0=\mathfrak K_R(g\mathbf C_i)$ and $A_1=\mathfrak
K_R(h\mathbf C_j)$ with $A_0\cap A_1\neq\emptyset$.  Let $e$ be the edge of
$\Gamma$ joining the vertices corresponding to $A_0$ and $A_1$.  If $A_2$ is a
subcomplex corresponding to some other vertex of $\Gamma$, and $A_0\cap
A_2\neq\emptyset$ and $A_1\cap A_2\neq\emptyset$, then $A_0\cap A_1\cap
A_2\neq\emptyset$, since each $A_i$ is convex and CAT(0) cube complexes have the
Helly property.  Now, $A_0\cap A_1$ is compact, and thus contained in some translate $aB$ of $B$.  Hence, for each $A_2$ that intersects $A_0$ and $A_1$, the mutual
intersection $A_0\cap A_1\cap A_2$ lies in $aB$.  In particular, $A_2$ intersects
$aB$.  Hence, by cocompactness, there are only finitely many $A_2$ such that the vertices in $\Gamma$ corresponding to
$A_0,A_1,A_2$ form a 3-cycle.

\textbf{4-cycles:}  As before, let $\{A_0,A_1\}$ be an edge of $\Gamma$.  Let
$A'_0,A'_1$ be vertices of $\Gamma$ (we use the same notation for the
corresponding subcomplexes of $\mathbf X$) such that $\{A_i,A_i'\}$ is an edge
of $\Gamma$ for $i\in\{0,1\}$ and $\{A'_0,A'_1\}$ is an edge of $\Gamma$.

Choose combinatorial geodesic paths $\rho_0,\rho'_0,\rho_1,\rho'_1$ such that
$\rho_i\rightarrow A_i$ and $\rho'_i\rightarrow A'_i$ for $i\in\{0,1\}$ and
$\rho_0\rho_1\rho'_1\rho'_0$ is a closed path in $\mathbf X$.  Let
$D\rightarrow\mathbf X$ be a disc diagram in $\mathbf X$ bounded by
$\rho_0\rho_1\rho'_1\rho'_0$, as in Figure~\ref{fig:finedisc1}.  Assume that $D$
has minimal area among all diagrams with that boundary path, and, moreover,
suppose that the $\rho_i$ and $\rho'_i$ are chosen among geodesic paths in the
required $A_i,A'_i$ in such a way that the resulting disc diagram $D$ is as
small as possible, in the following sense: $(\area(D),|\partial_pD|)$ is as
small as possible, where such pairs are taken in lexicographic order.

Suppose, for the moment, that $|\rho_i|,|\rho'_i|>0$ for each $i$, so that $D$
contains a dual curve emanating from each of the four named subpaths of its
boundary path.  If the dual curve $K$ emanates from $\rho_1$, then $K$ cannot
end on $\rho_1$, since that path is a geodesic.  Also, if $K_1,K_2$ are two dual
curves emanating from $\rho_1$, then they cannot cross, for otherwise, by
convexity of $A_1$, we could modify $\rho_1$ by finding a corner of a square of $A_1$ in the subdiagram bounded by $A_1$ and the arrowed path indicated in Figure~\ref{fig:finedisc1}, leading to a lower-area
diagram.  If $K$ is a leftmost (or rightmost) dual curve emanating from $\rho_1$
and ending on $\rho'_1$ (or $\rho'_0$, if $K$ is rightmost), then any dual curve
emanating from the part of $\rho_1$ subtended by $\rho'_1$ and $K$
(respectively, $\rho'_0$ and $K$) must cross $K$, and this is impossible.  Hence
$K$ is dual to the terminal (respectively, initial) 1-cube $c$ of $\rho_1$ and,
by performing a series of \emph{hexagon moves}
(see~\cite[Section~2]{WiseIsraelHierarchy}),
 we find that $\rho_1$ and $\rho'_1$ (respectively, $\rho_1$ and $\rho'_0$) have
a common 1-cube, namely $c$.  We can thus remove $c$ from $\rho_1,\rho'_1$,
resulting in a new diagram with the required properties, the same area as $D$,
and strictly shorter boundary path.  Since this is a contradiction, we conclude
that every dual curve travels from $\rho_1$ to $\rho'_0$ or from $\rho'_1$ to
$\rho_0$.  Let $\mathcal V$ be the set of hyperplanes corresponding to dual
curves of the former type, and $\mathcal W$ the set of hyperplanes corresponding
to dual curves of the latter type.  (Using this fact, the fact that geodesic
segments cross each hyperplane at most once, and the fact that hyperplanes do
not self-cross, it is easy to see that distinct dual curves in $D$ map to
distinct hyperplanes.)

\begin{figure}[h]
  \includegraphics[width=0.4\textwidth]{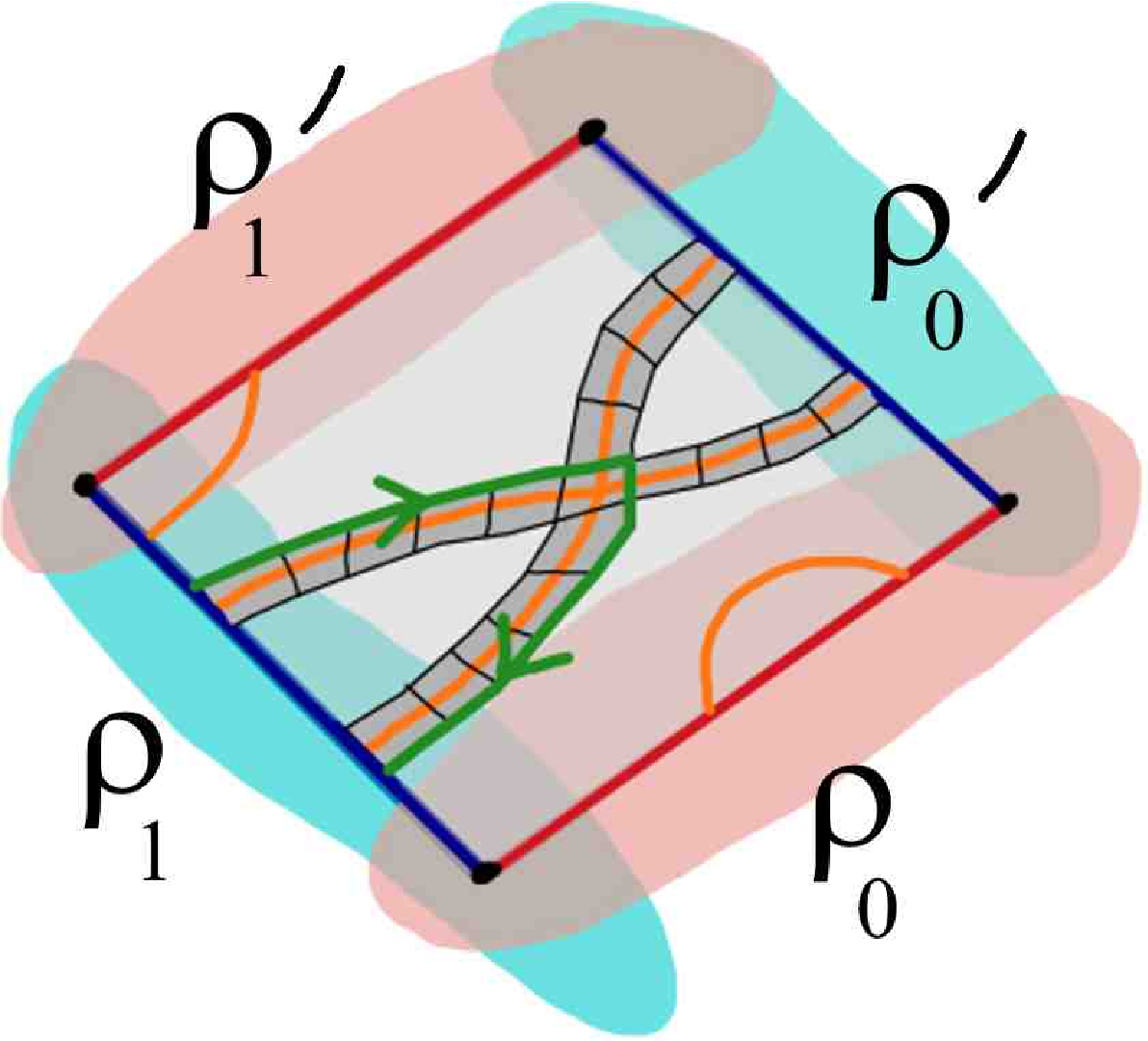}\\
  \caption{Some illegal dual curves, and an illegal crossing, in $D$.}\label{fig:finedisc1}
\end{figure}

This argument shows that $|\rho_1|=|\rho'_0|$ and $|\rho_0|=|\rho'_1|$.  If
$|\rho_1|=0$, then $A_0,A_1,A'_1$ pairwise-intersect, and hence $A'_1$
is one of finitely many vertices of $\Gamma$ that can be the third vertex in a
3-cycle containing the edge $\{A_0,A_1\}$.  But $A'_1,A'_0,A_0$ form a 3-cycle
in $\Gamma$, and thus there are only finitely many possible $A'_0$.  In other
words, if $\Gamma$ contains infinitely many 4-cycles containing the edge
$\{A_0,A_1\}$, then all but finitely many of these 4-cycles lead to disc
diagrams with $|\rho_1|=0$.  An identical argument works for $\rho_0$, and
hence $\mathcal V$ and $\mathcal W$ are nonempty for all but finitely many
4-cycles containing $\{A_0,A_1\}$.

Hence suppose that for all $m\geq 0$, there exist vertices $A'_0=A_0'(m),A'_1=A_1'(m)$ of
$\Gamma$ such that $A_0,A_1,A_1',A_0',A_0$ is a 4-cycle in $\Gamma$, and suppose
that for all $m$, the sets $\mathcal V(m),\mathcal W(m)$ defined above are
nonempty.  Note that $\mathcal V(m)\subseteq\mathcal H(A_1)\cap\mathcal
H(A_0'(m))$ and $\mathcal W(m)\subseteq\mathcal H(A_0)\cap\mathcal H(A_1'(m))$.
Moreover, if $V\in\mathcal V(m)$ and $W\in\mathcal W(m)$, then $V$ and $W$
cross, since their corresponding dual curves in the associated disc diagram
cross.

Next, we show that there exists $\xi<\infty$, depending only on $R$, such that $\max\{|\mathcal V(m)|,|\mathcal W(m)|\}\leq\xi$ for all $m$.  $\mathcal W(m)$ is a set of hyperplanes $H$ that cross both $A_1$ and $A'_0(m)$.  If it were possible to choose $A'_0(m)$ in such a way as to make $H(A_1)\cap\mathcal H(A_0'(m))$ have arbitrarily large cardinality, then since $\stabilizer_G(A_1)$ acts cocompactly on $A_1$, there would exist some $A'_0(m)$ with $H(A_1)\cap\mathcal H(A_0'(m))$ infinite, contradicting the fact that distinct translates of the various $\mathbf C_i$ have disjoint simplicial boundaries.

By cocompactness of the action of $\stabilizer_G(A_0)$, we can assume that $\rho_0(m)\cap\rho_1(m)$ lies
in a fixed compact set in $A_0$, of diameter $d<\infty$, and hence each
$A'_0(m)$ and $A'_1(m)$ come within $d+\xi$ of $\rho_0(1)\cap\rho_1(1)$.
There can only be finitely many such $A_0'(m)$ or $A_1'(m)$, and we conclude
that each edge of $\Gamma$ is contained in at most finitely many distinct
3-cycles or 4-cycles.

(Alternatively, we see that $|\mathcal V(m)|$ and $|\mathcal W(m)|$ must both be
unbounded as $m\rightarrow\infty$, and deduce that there exist infinite sets
$\mathcal V_{\infty}\subset\mathcal H(A_1)$ and $\mathcal
W_{\infty}\subset\mathcal H(A_0)$, with each $V\in\mathcal V$ crossing each
$W\in\mathcal W$.  Thus $\simp\mathbf X$ contains a 1-simplex joining a
0-simplex of $\simp A_1=g\mathbf S_i$ to a 0-simplex of $\simp A_0=h\mathbf
S_j$, and this is impossible.)

\textbf{Large cycles:}  Let $p\geq 4$. Let $A_0,A_1$ be a pair of vertices of $\Gamma$ connected by an edge. Let
$A_2,A_p$ be distinct vertices which are disjoint from $A_0,A_1$, and such that
$\{A_1,A_2\}$ and $\{A_p,A_0\}$ are edges of $\Gamma$.  Let $\sigma$ be an
embedded path of length at least 1 in $\Gamma$ joining $A_2$ to $A_p$ and not
containing $A_0$ or $A_1$; for $2\leq i\leq p$, let $A_i$ denote the subcomplex corresponding to the $(i-1)^{th}$ vertex of $\sigma$.  For each $0\leq i\leq p$, let $\rho_i\rightarrow A_i$ be a combinatorial
geodesic path such that $\rho_0\ldots\rho_p$ is a closed path in $\mathbf X$,
bounding a disc diagram $D$ that is minimal in the same sense as above (the details are identical to the $4$-cycle case).  Then

every dual curve in $D$ travels from some $\rho_i$ to some $\rho_j$ with $i\neq
j$.  For $0\leq\ell\leq p$, let $\mathcal V_{\ell}$ be the set of distinct
hyperplanes corresponding to dual curves emanating from $\rho_{\ell}$.  For each
$\ell$, there exists $\ell'$ such that $|\mathcal V_{\ell}\cap\mathcal
V_{\ell'}|\geq\frac{|\mathcal V_{\ell}|}{p-2}$, since there are $p$ possible
destinations for each of the dual curves emanating from $\rho_{\ell}$
(minimality of $D$ implies that such a
dual curve cannot end on $\rho_{\ell\pm1}$).  Now since $\mathcal
V_{\ell}\subset\mathcal H(A_{\ell})$ and $\mathcal V_{\ell'}\subset\mathcal
H(A_{\ell'})$, we have $|\rho_{\ell}|\leq(p-2)\xi$ for all $\ell$.  As
above, this implies that there are only finitely many paths $\rho$ in $\Gamma$
that combine with $\{A_0,A_1\}$ to make a $(p+1)$-cycle.  Thus $\Gamma$ is fine.
\end{proof}

\begin{lem}\label{lem:qi}
There exists $\delta\in[0,\infty)$ such that $\Gamma$ is $\delta$-hyperbolic.
\end{lem}

\begin{proof}
We will verify that the $G$-cocompact graph $\Gamma$ has thin triangles.

\textbf{Superconvexity:}  The arguments supporting fineness work for any sufficiently large
finite $R$.  In particular, we first show that we can choose $R$ large enough
that $\mathfrak
K_R(\mathbf C_i)$ is \emph{superconvex} for $1\leq i\leq k$, i.e., for any
bi-infinite (combinatorial or CAT(0)) geodesic $\gamma$ in $\mathbf X$, either
$\gamma\subset \mathfrak K_R(\mathbf C_i)$, or $\gamma\cap \mathfrak
K_r(\mathfrak K_R(\mathbf C_i))$ is bounded for all $r\geq 0$.  By
cocompactness, for all $r\geq 0$, there exists $m_r<\infty$ such that
$\diam(\gamma\cap\mathfrak K_{R+r}(\mathbf C_i))\leq m_r$ for any bi-infinite
geodesic $\gamma$ not contained in $\mathfrak K_R(\mathbf C_i)$.

To make this choice, suppose that for all $R\geq 0$, there exists a (CAT(0) or
combinatorial)
geodesic ray $\sigma_R$ lying in $\mathfrak K_R(\mathbf C_i)$, with every point
of $\sigma_R$ at distance at least $R-1$ from $\mathbf C_i$.  Applying
cocompactness and a standard disc diagram argument shows that, in this
situation, there is a boundary set $\mathcal U\subset\mathcal H(\mathbf C_i)$,
representing a simplex $u$ of $\mathbf S_i$, and a boundary set $\mathcal
V\subset\mathcal H-\mathcal H(\mathbf C_i)$ representing a simplex $v$ that is
adjacent in $\simp\mathbf X$ to $u$.  But $v\not\subset\mathbf S_i$, since every
simplex of $\mathbf S_i$ is represented by a boundary set consisting of
hyperplanes crossing $\mathbf C_i$.  Hence $v$ lies in some $\mathbf S_j$ that
differs from and intersects $\mathbf S_i$, a contradiction.

\textbf{Non-peripheral rectangular discs:}  Convexity and superconvexity of
$\mathfrak K_R(\mathbf C_i)$ together imply that any isometric flat $\mathbf
F\subset\mathbf X$ lies entirely inside some $\mathfrak K_R(g\mathbf C_i)$.
Cocompactness then implies that there exists $N$ such that if
$D\rightarrow\mathbf X$ is a combinatorial isometric embedding of the CAT(0)
cube complex $[0,m]^2$, then either $m < N$ or the image of $D$ is contained
in exactly one $\mathfrak K_R(g\mathbf C_i)$.

\textbf{Non-peripheral strips:}  $N$ and $R$ can be chosen so that if there exists a subspace $\mathfrak K_R(\mathbf C)$ corresponding to a vertex of $\Gamma$ and an isometrically embedded rectangle $S\cong[0,a]\times[0,b]\subset\mathbf X$ with $[0,a]\times\{0\}\subset\mathfrak K_R(\mathbf C)$ and $a\geq N$, then $S\subset\mathfrak K_R(\mathbf C)$.  This follows from superconvexity of $\mathfrak K_R(\mathbf C)$ and cocompactness of the action of its stabilizer.

\textbf{Representing geodesics in $\Gamma$:}  Let
$\gamma:[0,T]\rightarrow\Gamma$ be a geodesic segment.  For $0\leq i\leq T$, let $A_i=\gamma(i)$ be the $i^{th}$ vertex.  We also denote by $A_i$ the corresponding subcomplex $\mathfrak K_R(\mathbf C)$ of $\mathbf X$.  A combinatorial piecewise-geodesic $\rho$ is said to \emph{represent} the geodesic $\gamma$ in $\Gamma$ if $\rho=\rho_0\rho_1\ldots\rho_{T-1}$, where $\rho_i$ is a combinatorial geodesic of $A_i$ for $0\leq i\leq T-1$.

\textbf{Properties of projection to $\Gamma$:}  The remainder of the proof requires establishing three claims.  We note that there is a map $\mathbf X\rightarrow\Gamma$ sending each point to the vertex corresponding to the vertex corresponding to some $\mathfrak K_R(\mathbf C)$ containing it.  (There are many choices of such a map and we choose one arbitrarily.  Although we don't use this fact, these maps are coarsely the same, since any point of $\mathbf X$ lies in a uniformly bounded number of subcomplexes $\mathfrak K_R(\mathbf C)$.)  Below, we discuss images of paths under this map.  We note that these images need not be paths, but nevertheless are geometrically well-behaved in the following ways.

\renewcommand{\qedsymbol}{\ensuremath{\blacksquare}}

\begin{claim}\label{claim:bigon_thin}
Let $\sigma'\sigma^{-1}$ be a geodesic bigon in $\mathbf X$.  Then there exists $\delta'$ such that the image of $\sigma$ is contained in the $\delta'$-neighborhood of the image of $\sigma'$ and vice versa.
\end{claim}

\begin{proof}
Let $D\rightarrow\mathbf X$ be a minimal-area disc diagram with boundary path $\sigma'\sigma^{-1}$.  Since $\sigma,\sigma'$ are geodesics, every dual curve in $D$ starts on $\sigma$ and ends on $\sigma'$.  Choose $x\in\sigma$ and $x'\in\sigma'$.  Let $\mathcal L$ be the set of dual curves starting on $\sigma$ to the left of $x$ and ending on $\sigma'$ to the right of $x'$, and let $\mathcal R$ be the set of dual curves starting on $\sigma$ to the right of $x$ and ending on $\sigma'$ to the left of $x'$.  Then every dual curve in $D$ separating $x,x'$ belongs to one of these sets, whence $$\done(x,x')\leq|\mathcal L|+|\mathcal R|.$$
If either of $\mathcal L$ or $\mathcal R$ has cardinality at most $N$, then $x$ lies at distance at most $N$ from $\sigma$ and $x'$ lies at distance at most $\epsilon N$ from $\sigma'$.  On the other hand, since each dual curve in $\mathcal L$ crosses each dual curve in $\mathcal R$, if $|\mathcal L|,|\mathcal R|\geq N$, then $\mathbf X$ contains an isometric flat rectangle $F$, each of whose sides has length at least $N$, containing $x,x'$.  The rectangle $F$ is contained in some subcomplex $\mathbf C$ corresponding to a vertex of $\Gamma$, whence the images of $x,x'$ can be joined by a path of length $2$ in $\Gamma$ whose middle vertex is $\mathbf C$.  Hence the image of $\sigma$ is contained in the $\delta'$-neighborhood of the image of $\sigma'$, and vice versa, for $\delta'$ depending only on $N$.
\end{proof}

\begin{claim}\label{claim:thin_triangle}
Let $\gamma\gamma'\gamma''$ be a geodesic triangle in $\mathbf X$.  There exists $\delta$ such that the image of any of $\gamma,\gamma',\gamma''$ in $\Gamma$ lies in the $\delta$-neighborhood of the union of the images of the other two paths.
\end{claim}

\begin{proof}
This follows from the fact that $\mathbf X^{(0)}$, endowed with the metric $\done$, is a median space, together with Claim~\ref{claim:bigon_thin}.  Indeed, let $\gamma\gamma'\gamma''$ be a geodesic triangle in $\mathbf X^{(1)}$.  Then there is a combinatorial geodesic triangle $\alpha\alpha'\alpha''$ such that $\alpha\gamma^{-1},\alpha'(\gamma')^{-1},\alpha''(\gamma'')^{-1}$ are geodesic bigons and each of $\alpha,\alpha',\alpha''$ is contained in the union of the other two (each passes through the median of the three endpoints of $\gamma\cup\gamma'\cup\gamma''$).  Hence, by Claim~\ref{claim:bigon_thin}, the image of each of $\gamma,\gamma',\gamma''$ in $\Gamma$ lies in the $\delta=2\delta'$-neighborhood of the union of the other two.
\end{proof}

\begin{claim}\label{claim:hierarchy_path!}
There exists $\mathfrak L$, independent of $\gamma$, such that a representative $\rho$ can be chosen so that its image in $\Gamma$ is contained in the $\mathfrak L$-neighborhood of the image of a geodesic $\sigma$ of $\mathbf X$.
\end{claim}

\begin{proof}
There are several steps:\\

\emph{Strategy:} Suppose that $\gamma$ has a representative $\rho$, so
that $\rho=\rho_0\rho_1\cdots\rho_{T-1}$ is a piecewise-geodesic with
$\rho_j\rightarrow A_j$ for $0\leq j\leq T-1$ that joins $x_0\in A_0$
to $x_{T}\in A_{T-1}\cap A_T$.  Let $\sigma_0$ be a geodesic joining
$x_0$ to $x_T$.  Let $D\rightarrow\mathbf X$ be a minimal-area disc
diagram bounded by $\rho\sigma_0^{-1}$.  Convexity of the $A_j$ implies
that no dual curve starts on $\rho_j$ and ends on $\rho_{j\pm1}$, for
otherwise we could remove backtracks from the boundary path of $D$.
Similarly, no two dual curves emanating from a common $\rho_j$ can
cross, for otherwise convexity of $A_j$ would enable us to modify
$\rho_j$, without changing its endpoints, to obtain a lower-area
diagram.

If no dual curve in $D$ has both ends on $\rho$, then $\rho$ is a
geodesic and the claim holds by setting $\sigma=\rho$.  Hence, we suppose
that $K$ is a dual curve in $D$ that is \emph{outermost} in the sense
that $K$ is dual to two distinct 1-cubes on $\rho$, and the subpath of
$\rho$ subtended by these 1-cubes is not properly contained in a
subpath subtended by two distinct 1-cubes dual to the same dual curve.
If the image of $K$ under the map $D\rightarrow\mathbf
X\rightarrow\Gamma$ is at uniformly bounded Hausdorff distance from
the image of $\rho$, then we can replace the part of $\rho$ between
and including the 1-cubes dual to $K$ by a path in the carrier of $K$,
yielding a new path $\rho'$, whose image is at uniformly bounded
Hausdorff distance from that of $\rho$, but which has strictly fewer
pairs of 1-cubes dual to a common hyperplane.  Finitely many
repetitions of this procedure then yields the desired $\sigma$.  Hence
it suffices to find $\mathfrak L$ such that the $\mathfrak
L$-neighborhood of the image of $K$ in $\Gamma$ contains the image of
$\rho$.\\

\emph{The subdiagram $D'$:}  To this end, suppose that $K$ starts on $\rho_j$ and ends on $\rho_{j'}$, with $|j-j'|>1$.  Let $P$ be a shortest path in $N(K)\subset D$ starting at $N(K)\cap\rho_j$ and ending at $N(K)\cap\rho_{j'}$, with $P$ separated from the subtended part of $\rho$ by $K$.  Let $\rho'$ be the subtended part of $\rho$, so that $\rho'=\rho'_j\rho_{j+1}\cdots\rho_{j'}'$, where $\rho'_j,\rho'_{j'}$ are respectively subpaths of $\rho_j,\rho_{j'}$.  Let $D'\rightarrow\mathbf X$ be the subdiagram of $D$ bounded by $P$ and $\rho'$.  As before, no dual curve travels from $\rho'_j$ to $\rho_{j+1}$, or $\rho_k$ to $\rho_{k\pm1}$ for $j-1\leq k\leq j'+1$, or from $\rho_{j'-1}$ to $\rho_{j'}'$, and no two dual curves emanating from the same named subpath of $P$ cross.  Every dual curve emanating from $P$ ends on $\rho'$, since $D$ has minimal area for its boundary path and therefore contains no bigon of dual curves (see e.g.~\cite{Sageev95,WiseIsraelHierarchy}).  Note that the images of $\rho'_j$ and $\rho'_{j'}$ in $\Gamma$ are at distance at most 1 from the images of $A_j,A_{j'}$ and hence at distance at most 2 from the image of $P$.\\

\emph{The diagrams $D'_k$:}  For $j+1\leq k\leq j'-1$, we inductively define combinatorial paths $a_k,b_k$ starting on $\rho_k$ and ending on $P$ as follows.  Let $a_{j+1}$ be a shortest path in $D'$ joining a point of $\rho_{j+1}$ to a point of $P$.  Let $b_{j+1}$ be of minimal length among all paths in $D'$ joining a point of $\rho_{j+2}$ to a point of $P$ and not crossing $a_{j+1}$ (these paths are allowed to coincide for some or all of their lengths).  Given $a_k$ joining $\rho_k$ to $P$, let $b_k$ be a minimal path joining $\rho_{k+1}$ to $P$ that does not cross $a_k$, and given $b_{k}$, let $a_{k+1}$ be a minimal path joining $\rho_{k+1}$ to $P$ that does not cross $b_k$.  See Figure~\ref{fig:a_and_b}.

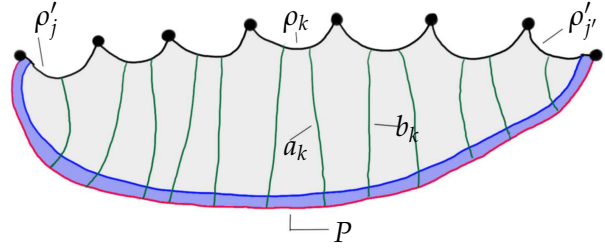
\begin{figure}
\begin{overpic}[width=0.5\textwidth]{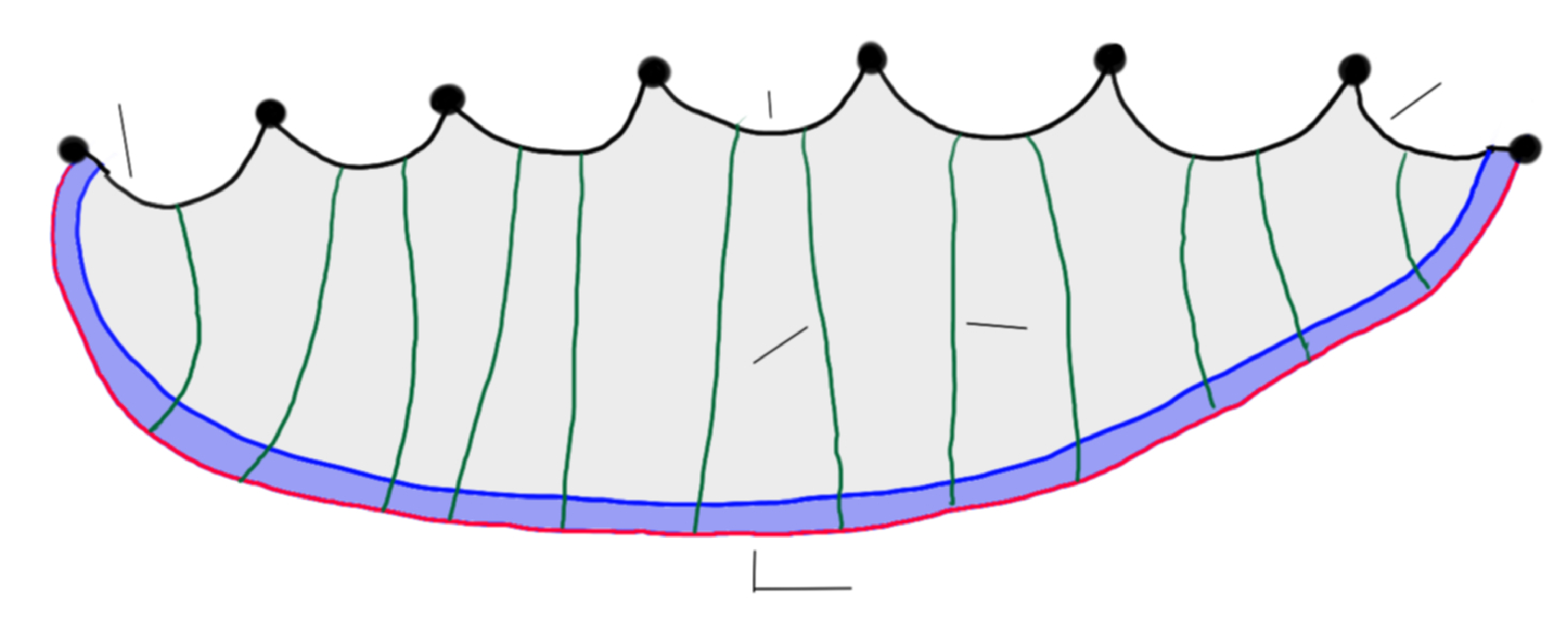}
\put(7,35){$\rho'_j$}
\put(47,35){$\rho_k$}
\put(93,35){$\rho'_{j'}$}
\put(55,0){$P$}
\put(47,14){$a_k$}
\put(65,17){$b_k$}
\end{overpic}
\caption{The diagram $D'$.}\label{fig:a_and_b}
\end{figure}

For each $k$, let $P_k$ be the subpath of $P$ between the endpoints of
$a_k$ and $b_k$.  Let $c_k$ be the subpath of $\rho_k$ between the
initial point of $a_k$ and the terminal point of $\rho_k$, and let
$d_k$ be the part of $\rho_{k+1}$ from the initial point of
$\rho_{k+1}$ to the initial point of $b_k$; these paths are shown in
Figure~\ref{fig:a_and_b}.  Consider the subdiagram $D'_k$ bounded by
$a_k,P_k,b_k,d_k,$ and $c_k$.  Every dual curve in $D'_k$ emanating
from $P_k$ ends on $c_k$ or $d_k$, and no two such dual curves cross.
Indeed, if such a dual curve $C$ ended on $a_k$ (or $b_k$), then we
could have chosen $a_k$ (or $b_k$) to be shorter, as shown in
Figure~\ref{fig:diagram_D_k} at left.  Similarly, no dual curve
travels from $a_k$ to $c_k$ or $b_k$ to $d_k$.  We conclude that
$D'_k$ is the union of two (possibly degenerate) flat rectangles,
$T_k,U_k$ and a subdiagram $V_k$ shown at right in
Figure~\ref{fig:diagram_D_k}.  The subdiagram $V_k$ is formed by the
crossing of the dual curves emanating from $P_k$ with the dual curves
traveling from $a_k$ to $b_k$.  The rectangle $T_k$ is formed from the
dual curves traveling from $a_k$ to $d_k$ crossing those that emanate
from $c_k$.  The rectangle $U_k$ is formed analogously.  Now, if
$|c_k|\geq N$, then the strip $T_k$ actually lies in $A_k$ and we
could have chosen $\rho_k$ to yield a lower-area diagram $D$.  Hence
$|c_k|<N$ and $|d_k|<N$.  Thus $|P_k|<N$, and there is a path of
length less than $2N$ joining $\rho_k\cap\rho_{k+1}$ to $V_k$.  It
follows that if, for any $\epsilon\geq0$, at most $\epsilon N$ dual
curves travel from $a_k$ to $b_k$, then
$\done(\rho_k\cap\rho_{k+1},P)\leq (2+\epsilon)N$.  The images of
$\rho_k$ and $\rho_{k+1}$ in $\Gamma$ thus lie in the
$[(2+\epsilon)N+1]$-neighborhood of the image of $P$.\\

\begin{figure}[h]
\includegraphics[width=0.5\textwidth]{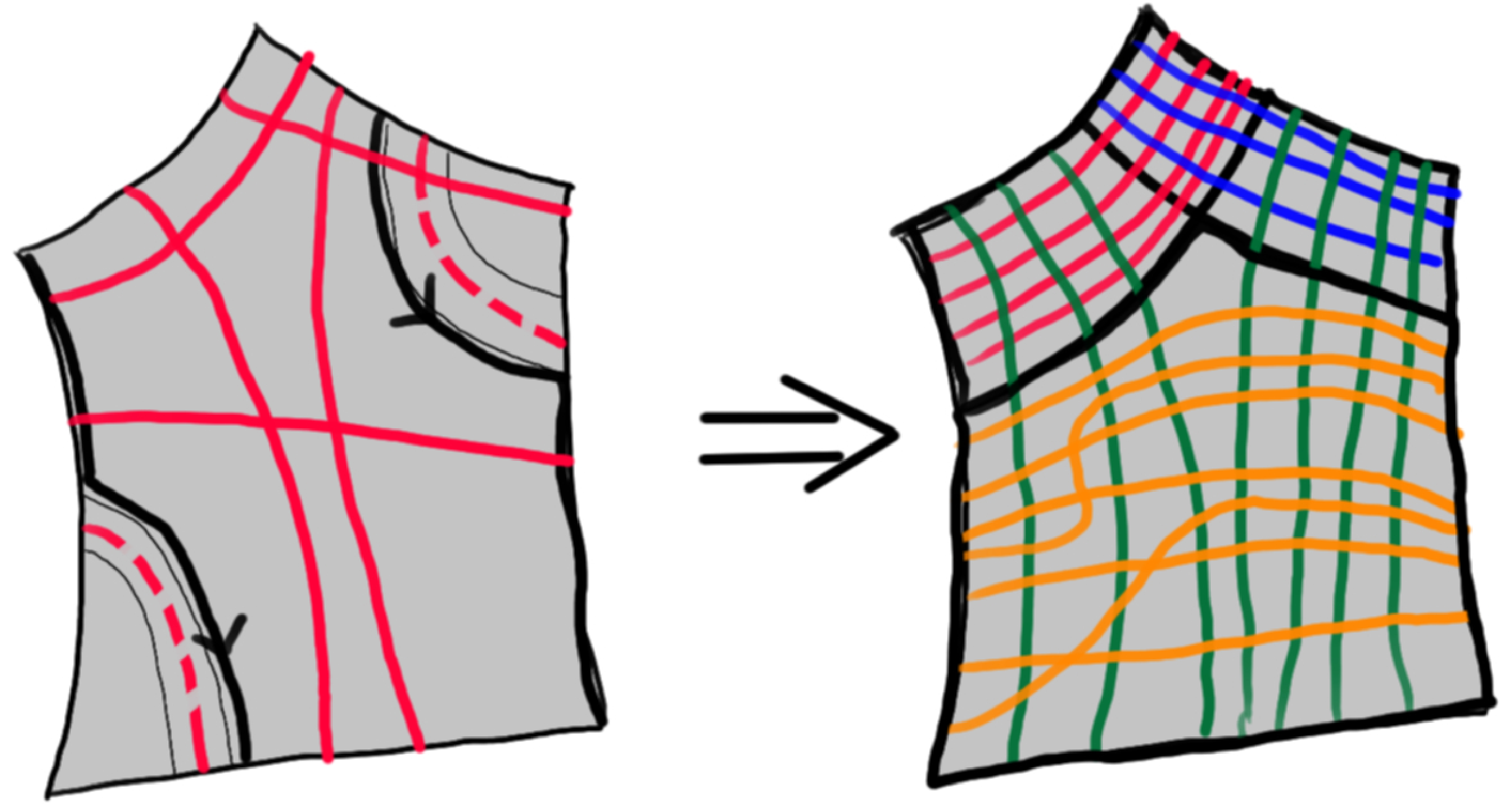}
\caption{Left: the solid dual curves shown in $D'_k$ are all possible.
If either dotted dual curve occurs, then as shown, either $a_k$ or
$b_k$ could be shortened (there are two other similar possibilities
not shown).  This leads to the conclusion at left: the rectangles
$T_k,U_k$ intersect in the smaller rectangle at the top of $D'_k$,
each of whose sides has length less than $N$, and the remainder of the
diagram is $V_k$.}\label{fig:diagram_D_k}
\end{figure}

\emph{The diagrams $E_k$:} For each $k$, let $Q_k$ be the subpath of
$P$ between the endpoints of $b_k$ and $a_{k+1}$ and let $e_k$ be the
subpath of $\rho_{k+1}$ between the initial point of $b_k$ and the
initial point of $a_{k+1}$.  The subdiagram $E_k$ bounded by
$e_k,a_{k+1},Q_k,$ and $b_k$ has the property that all dual curves
travel from $b_k$ to $a_{k+1}$ (by the minimality of those paths) or
from $Q_k$ to $e_k$.  If there are at least $N$ dual curves from $b_k$
to $a_{k+1}$, then the convex hull of the image of $E_k$ in $\mathbf
X$ contains an $N\times N$ flat grid. Since this image of $E_{k}$
contains an $N\times N$ flat grid, as we proved above in the
paragraph on ``Non-peripheral rectangular discs'' we then have
$E_{k}$ contained in some $A_{i}$ and thus the distance in $\Gamma$
between the images of $\rho_k$ and $P$ is at most 3.  Thus
$|\rho_k|\leq 3N$ for all $k$ for which the distance between some
point in the image of $\rho_k$ and the image of $P$ is at least 4.

Choose $k_1,k_2$ with $j\leq k_1\leq k_2\leq j'$ such that for all $k\in\{k_1,\ldots,k_2\}$, the diagram $D'_k$ has more than $\epsilon N$ dual curves traveling from $a_k$ to $b_k$, and for $k\in\{k_1,\ldots,k_2-1\}$, the diagram $E_k$ has more than $\epsilon N$ such dual curves, and the distance from some point of each $\rho_k$ to $P$ in $\Gamma$ is at least 4, and such that $k_2-k_1$ is as large as possible.

Define the subdiagram $E$ of $D'$ to consist of the union of the
$D'_k$, for $k_1\leq k\leq k_2$, together with $E_k$ for $j\leq k\leq j'$.  Let $\mathcal V$ be the set of \emph{vertical} dual curves, i.e., those that have an end on some $P_k$ or $Q_k$.  By the above discussion, at most $3N$ vertical dual curves end on each $\rho_k$.  Observe that there is a path of length $2\epsilon N+|\mathcal V|$ joining $\rho_{k_1}$ to $\rho_{k_2}$, and hence a path of length at most $2\epsilon N+|\mathcal V|+2$ in $\Gamma$ joining $A_{k_1}$ to $A_{k_2}$.  Hence $|\mathcal V|\geq k_2-k_1-2\epsilon N-1$.  If $k_2-k_1\leq 2(2\epsilon N+1)$, then we have a uniform bound 
of $2(2\epsilon N+1)+3$ 
on the distance from any point of the image of any $\rho_k$ to the image of $P$, for $k_1\leq k\leq k_2$.  Hence we can assume $|\mathcal V|\geq\frac{k_2-k_1}{2}$.  Since there is a bound of $3N$ on the number of vertical dual curves intersecting each $\rho_k$, there exists an integer $p=p(N)\geq 1$, independent of $\epsilon$, such that any concatenation of $p$ consecutive paths of the form $\rho_k$, with $k_2\leq k\leq k_1$, crosses at least $N$ vertical dual curves.

To conclude, consider a path $\rho_k\rho_{k+1}\cdots\rho_{k+p}$ with
$k_1\leq k\leq k+p\leq k_2$.  This path crosses at least
$N$ vertical dual curves.  There are at least $\epsilon
N-pN$ dual curves in $E$ that cross $b_{k}$ and $a_{k+p}$, and thus
cross each intervening vertical dual curve emanating from $P$, since for
each $k'$ at most $N$ non-vertical dual curves leave the diagram through $c_{k'}$.
Hence take $\epsilon=p+1$.  Then there
are at least $N$ horizontal dual curves, each of which crosses each of
the at least $N$ vertical dual curves, in the subdiagram between
$b_k,a_{k+p}$, and the subtended parts of $\rho$ and $P$.  Hence there
is an $N\times N$ flat grid whose convex hull intersects $P$ and
$\rho_{k},\rho_{k+1},\ldots,\rho_{k+p}$.  Thus each such path projects
to a subspace of the $3$-neighborhood of the image of $P$ in $\Gamma$.
Either every $\rho_k$ is contained in such a path, or $k_2-k_1\leq p$ and we can bound the distance from any $\rho_k$ to $P$ in
$\Gamma$.
\end{proof}

\renewcommand{\qedsymbol}{\ensuremath{\square}}

\textbf{Conclusion:}  Let $\gamma,\gamma',\gamma''\rightarrow\Gamma$ be geodesics forming a triangle in $\Gamma$.  Let $\rho,\rho',\rho''$ be combinatorial paths respectively representing $\gamma,\gamma',\gamma''$ as above, chosen so that $\rho\rho'\rho''$ is a closed path in $\mathbf X$.  For each $p\in\rho$, there is some subspace $\mathbf C$ representing a vertex of $\gamma$ and containing $p$.  Hence the image of $p$ in $\Gamma$ lies at distance at most 1 from $\gamma$.  Similarly, the image of $\rho'$ [respectively $\rho''$] lies in the 1-neighborhood of the image of $\gamma'$ [respectively $\gamma''$].  By Claim~\ref{claim:hierarchy_path!}, there exist geodesics $\sigma,\sigma',\sigma''$, the $\mathfrak L$-neighborhoods of whose images in $\Gamma$ respectively contain the images of $\rho,\rho',\rho''$.  By Claim~\ref{claim:thin_triangle}, the image of the geodesic triangle $\sigma\sigma'\sigma''$ has the property that the image of any side is contained in the $\delta$-neighborhood of the image of the other two sides.  Hence $\gamma\gamma'\gamma''$ is $\delta+(2\mathfrak L+1)$-thin, whence $\Gamma$ is $\delta+2(\mathfrak L+1)$-hyperbolic.
\end{proof}

In particular, when the $\mathbf S_i$ are hyperoctahedra of dimension at least
1 satisfying the hypotheses of Theorem~\ref{thm:relhyp}, then we may conclude that $G$ is hyperbolic
relative to a finite collection of virtually abelian subgroups, as we now explain.  First,
consider the action of $Q_i$ on $\mathbf C_i$. This action is proper and
cocompact, and by Lemma~\ref{lem:boundaryofcore}
and~\cite[Proposition~3.5]{CapraceSageev}, we may assume that this action is
essential.  Now, $\mathbf C_i$ is fully visible because any invisible simplex is non-maximal and contained in a unique maximal simplex, by the proof of~\cite[Theorem~3.19]{HagenBoundary}, and no such simplices exist in a hyperoctahedron.
By~\cite[Theorem~3.30]{HagenBoundary}, the decomposition $\mathbf
S_i\cong\mathbb O_{d-1}\star\mathbb O_0$ corresponds to a decomposition $\mathbf
C_i\cong X_{d-1}\times X_0$, where $\simp X_0\cong\mathbb O_0$ and $\simp
X_{d-1}\cong\mathbb O_{d-1}$.
Since the boundary of $X_0$ is a single pair of points, and $X_0$ is cocompact, there exists a periodic geodesic $\gamma$ such that $X_0$ lies in a finite neighborhood of $\gamma$.  By induction on dimension, $X_{d-1}$ contains a periodic flat $F\cong\reals^{d-1}$ which coarsely contains all of $X_{d-1}$.  Hence $\mathbf C_i$ is coarsely contained in a flat $F\times\gamma$ of dimension $d$ that is stabilized by a finite-index subgroup of $Q_i$.  Thus $Q_i$ is virtually $\integers^d$, by Bieberbach's theorem.

\begin{exmp}\label{exmp:relhyp}
We conclude this section with some examples and non-examples of relatively
hyperbolic cocompactly cubulated groups:

\begin{enumerate}
\item (Right-angled Artin groups) The results
of~\cite{BehrstockCharney} and~\cite{BDM} combine to show that
one-ended right-angled Artin groups are never relatively hyperbolic
since they are all
either thick of order 0 (in the case the group is a direct product) or
thick of order 1 and thus not relatively hyperbolic by~\cite[Corollary~7.9]{BDM}.
Theorem~\ref{thm:relhyp} above provides another proof of non-relative
hyperbolicity for these groups, since the simplicial boundary of a
one-ended right-angled Artin group, $A$, has only one
positive-dimensional connected component.

\item (Hyperbolic relative to a right angled Artin group)
Figure~\ref{fig:hyprelraag} shows a
cubical subdivision of the Salvetti complex $\overline{\mathbf C}$ of
\[F_2\times\integers\cong\langle a,b,t\mid [a,t],[b,t]\rangle\] at left and a
nonpositively-curved cube complex $\overline{\mathbf Y}$ at right that is a
tiling by 2-cubes of a closed, orientable genus-3 surface.  The fundamental
group of $\overline{\mathbf Y}$ is presented by \[\pi_1\mathbf Y\cong\langle
p_1,q_1,p_2,q_2,p_3,q_3\mid[p_1,q_1][p_2,q_2][p_3,q_3]\rangle\] and we form a
compact nonpositively-curved cube complex $\overline{\mathbf X}$ by attaching a
cylinder to $\overline{\mathbf C}$ and $\overline{\mathbf Y}$ as shown, so that
$G=\pi_1\overline{\mathbf X}$ is isomorphic to the following \[(\pi_1\mathbf C\ast\pi_1\mathbf
Y)\slash\langle\langle
b=p_1q_1^{-1}p_1^{-1}p_2q_2p_2^{-1}\rangle\rangle\]

\begin{figure}[h]
\includegraphics[width=0.6\textwidth]{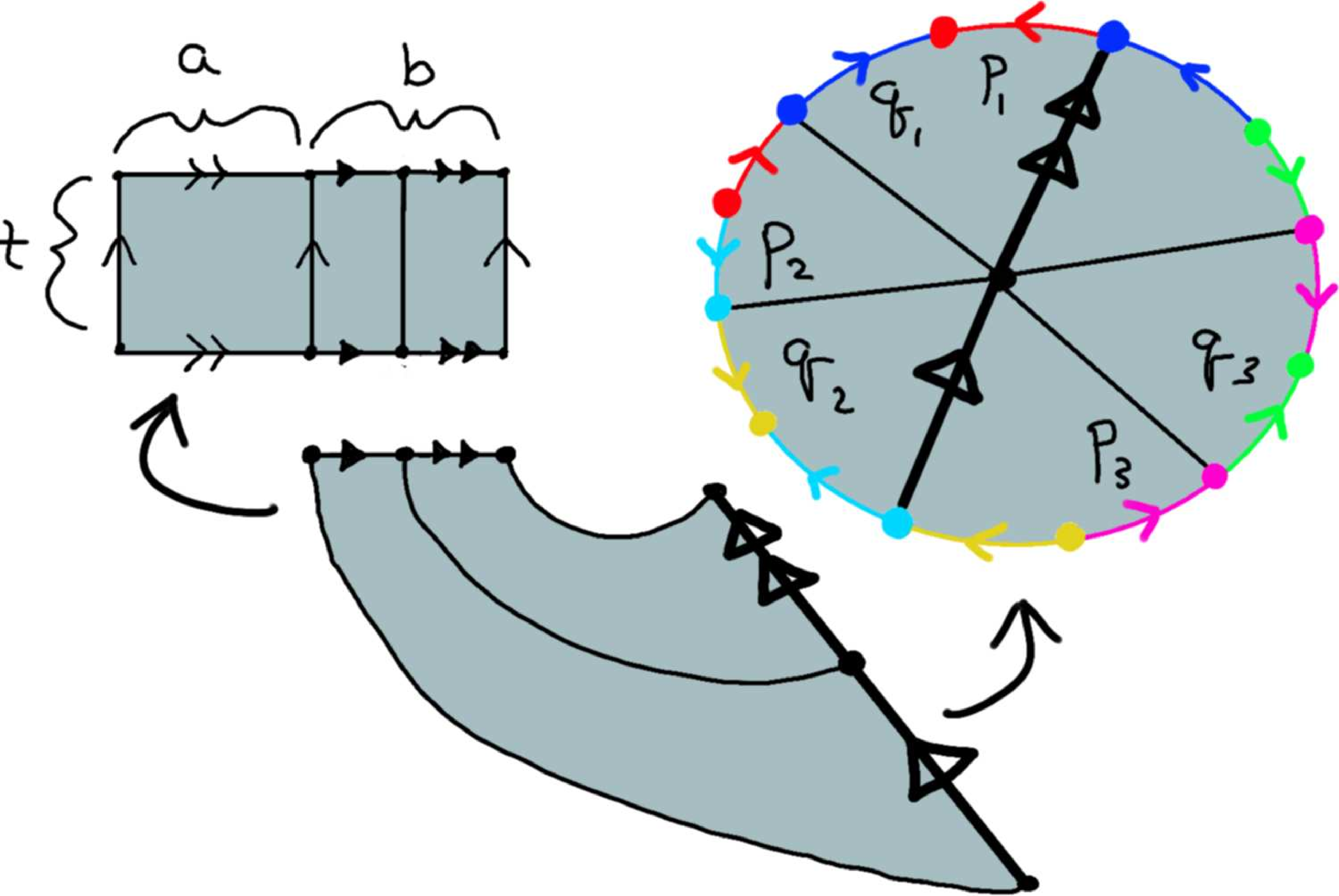}
\caption{}\label{fig:hyprelraag}
\end{figure}

Since the attaching maps of the cylinder are locally convex circles,
$\overline{\mathbf C}$ and $\overline{\mathbf Y}$ are locally convex in
$\overline{\mathbf X}$, and hence the universal cover $\mathbf C$
is a convex, $P\cong\pi_1\overline{\mathbf C}$-cocompact subcomplex of the
universal cover $\mathbf X$.  Now, $\mathbf S=\simp\mathbf C$ is
isomorphic to the join of an infinite discrete set with a pair of 0-simplices,
and $\mathbf S\subset\simp\mathbf X$.  Any two distinct translates
of $\mathbf C$ intersect in a translate of a convex periodic
geodesic lying in a translate of the universal cover $\mathbf Y$, which is a
convex copy of $\mathbb H^2$ in $\mathbf X$. Hence, since cyclic subgroups of
$\pi_1\overline{\mathbf Y}$ are malnormal, $\mathbf S\cap g\mathbf S=\emptyset$
for $g\not\in P$.  Now, every flat orthant in $\mathbf X$ lies in some
translate of $\mathbf C$.  Therefore, $\simp\mathbf X$ is the union of
translates of $\mathbf S$ together with a nonempty set of isolated points
arising from translates of $\simp\mathbf Y$, and
Theorem~\ref{thm:relhypconverse} confirms that $G$ is hyperbolic relative to
$P$.

\item (Cusped hyperbolic 3-manifolds)  There are many cusped, hyperbolic
3-manifolds $\widehat M$ for which $\pi_1\widehat M$ is the fundamental group of
a compact nonpositively-curved cube complex.  Such manifolds arise as finite
covers of finite-volume cusped hyperbolic 3-manifolds that contain a
geometrically finite incompressible
surface~\cite[Theorem~14.29]{WiseIsraelHierarchy}.  In this case, the cusp
subgroups correspond to isolated 4-cycles in the simplicial boundary of the
cocompact cubulation of $\pi_1M$, the remainder of which consists of an
infinite collection of isolated 0-simplices.
\end{enumerate}
\end{exmp}

\section{Unconstricted and wide cube complexes}\label{sec:unconstricted}
We assume throughout this section that $\mathbf X$ is a locally finite,
finite-dimensional CAT(0) cube complex.

$\mathbf X$ is \emph{geodesically complete} if each CAT(0) geodesic segment is
contained in a bi-infinite CAT(0) geodesic.  If $\mathbf X$ is geodesically
complete, then it is \emph{combinatorially geodesically complete} in the sense
that, for any maximal set $W_1,\ldots,W_n$ of pairwise-crossing hyperplanes,
each of the $2^n$ maximal intersections of halfspaces associated to those
hyperplanes contains 0-cubes arbitrarily far from the cube
$\cap_{i=1}^nN(W_i)$.  Equivalently, $\mathbf X$ is combinatorially
geodesically complete if every combinatorial geodesic segment extends to a
bi-infinite combinatorial geodesic, as is shown in~\cite{HagenBoundary}.  If $\mathbf X$ is (combinatorially or CAT(0)) geodesically complete, then $\mathbf X$ satisfies the first requirement of the definition of an unconstricted space, since each point of $\mathbf X$ lies at distance 0 from a bi-infinite (combinatorial or CAT(0)) geodesic and hence lies uniformly close to a CAT(0) quasigeodesic.

Let $\omega$ be an ultrafilter, $(s_n)_{n\geq 1}$ a sequence of scaling
constants, and $(x_n)_{n\geq 1}$ a sequence of observation points in $\mathbf
X$.  Denote by $[y_n]$ the point of $\ascone{\mathbf
X}{(x_n)}{(s_n)}{\omega}$
represented by the sequence $(y_n\in\mathbf X)_{n\geq 1}$.  Since $\mathbf X$ is
finite-dimensional the CAT(0) metric and the path metric on $\mathbf X^{(1)}$
are quasi-isometric, and thus
$\ascone{\mathbf X}{(x_n)}{(s_n)}{\omega}$ is bilipschitz homeomorphic to
$\ascone{\mathbf X^{(1)}}{(x'_n)}{(s_n)}{\omega}$, where $x'_n$ is a closest
0-cube to $x_n$. Where the ultrafilter, scaling constants, and observation
points are understood, we denote this asymptotic cone by $\Xomega$.

We say $\simp\mathbf X$ is \emph{bounded} if its 1-skeleton (with the usual
graph
metric) is finite diameter.

\begin{thm}\label{thm:unconstrictedCC}
Let $\mathbf X$ be a locally finite, finite-dimensional CAT(0) cube
complex such that $|\simp\mathbf X|>1$.  If $\simp\mathbf X$ is
bounded then no asymptotic cone of $\mathbf X$ is separated
by a finite closed ball, in the sense that in no asymptotic cone do
there exist points $\mathbf a,\mathbf b,\mathbf x$ such that $\mathbf
d_{\omega}(\mathbf x,\{\mathbf a,\mathbf b\})>3$ and every path from
$\mathbf a$ to $\mathbf b$
passes through the 1-ball about $\mathbf x$. Under 
the additional hypotheses that every combinatorial 
geodesic segment can be extended to a ray: if $\simp\mathbf X$ is bounded, then
$\mathbf X$ is wide.
\end{thm}

\begin{proof}
Although $\mathbf X$ is not assumed to be fully visible, we always work with visible simplices, justified by the fact that maximal simplices are visible~\cite[Theorem~3.19]{HagenBoundary}.

Let $\alpha,\beta$ be combinatorial geodesics, representing simplices
$h_{\alpha},h_{\beta}$ of $\simp\mathbf X$ respectively.  Without loss of
generality, $\alpha$ and $\beta$ have a common initial point $x_o$.  The
\emph{cubical divergence}, $\dive{\alpha}{\beta}(r)$, is the length of a
shortest
combinatorial path $P_r\rightarrow\mathbf X$ which joins $\alpha(r)$ to
$\beta(r)$ and contains no 0-cube at distance less than $r$ from $x_o$.  Now,
$h_{\alpha}$ and $h_{\beta}$ lie in the same component of $\simp\mathbf X$ if
and only if $\dive{\alpha}{\beta}(r)$ is bounded above by a linear function of
$r$, by~\cite[Theorem~6.8]{HagenBoundary}.  In this case, for all $r\geq 0$,
\[A_1r+B_1\leq\dive{\alpha}{\beta}(r)\leq A_2r+B_2\]
where $A_1,A_2$ depend linearly on the distance between $h_{\alpha}$ and
$h_{\beta}$ in $\simp\mathbf X^{(1)}$ and $B_1,B_2$ are constants depending on
$\alpha$ and $\beta$.  We first exhibit a cut-ball in an asymptotic cone when
$\simp\mathbf X$ is disconnected, and then do the same when $\simp\mathbf X$ is
connected but unbounded.

\textbf{Disconnected $\simp\mathbf X$ implies cut-ball:} Suppose that
$h_{\alpha}$ and $h_{\beta}$ lie in distinct components of $\simp\mathbf X$.
Then, for each $M\geq 1$, there exists a smallest $r_M\geq M$ such that
$\dive{\alpha}{\beta}(r_M)\geq Mr_M$.  From the definition of $r_M$, it follows
immediately that $\dive{\alpha}{\beta}(Kr_M)\geq (2-2K+M)r_M$ for any fixed
$K\geq 1$.  

Consider an asymptotic cone, $\ascone{\mathbf X}{\mathbf x}{(r_n)}{\omega}$,
where the scaling
constants are given by the $(r_n)$ above, and the sequence of observation points
is $\mathbf x=(x_o)$.

For each $n\geq 0$, let $a_n=\alpha(Kr_n)$, where $K\geq 3$ is some fixed
integer, and likewise let $b_n=\beta(Kr_n)$.  Then
$\done(a_n,x_o)r_n^{-1}=K=\done(b_n,x_o)r_n^{-1}$, so that $\mathbf
a=[(a_n)],\mathbf b=[(b_n)]$ define points of $\ascone{\mathbf X}{\mathbf
x}{(r_n)}{\omega}$, and these points are each at distance $K$ from $\mathbf x$.

By construction, any path $P_n$ in $X$ from $a_n$ to $b_n$ either has
length at least $(2-2K+n)r_n$ or travels through the interior of the $r_n$--ball about $x_o$, i.e., through the closed $(r_n-1)$--ball.
We see this as follows. By prepending the part of $\alpha$ joining $\alpha(r_n)$ to $a_n$, and
appending the part of $\beta$ joining $b_n$ to $\beta(r_n)$, to $P_n$, we
obtain a path $P'_n$ of length $2(K-1)r_{n}+|P_n|$ joining $\alpha(r_n)$ to
$\beta(r_n)$.  Either $P'_n$ travels through the interior of the forbidden
$r_n$--ball or else, by our choice of $r_n$, 
$|P'_n|\geq nr_n$ and thus $|P_n|\geq (2-2K+n)r_n$ as claimed.

By construction, $\mathbf d_{\omega}(\mathbf a,\mathbf b)\leq2K$ and, as noted
above, $\mathbf d_{\omega}(\mathbf a,\mathbf x)=\mathbf d_{\omega}(\mathbf
b,\mathbf
x)=K$. We shall show that the closed ball of radius 1 about $\mathbf x$
separates $\mathbf a$ from $\mathbf b$.

Let $\mathfrak P$ be a finite length path in $\ascone{\mathbf X}{\mathbf
x}{(r_n)}{\omega}$ joining $\mathbf a$ to
$\mathbf b$ and let $P_n$ be a path in
$\mathbf X$ joining $a_n$ to $b_n$ for which the $\omega$-limit of these paths
is $\mathfrak P$.  Either $P_n$ passes through the
$(r_n-1)$-ball about $x_o$ for $\omega$-almost all $n$, or
$|P_n|\geq (2-2K+n)r_n$ for $\omega$-almost all $n$.
Now, the latter case can't occur, since if it did then we would have
$\lim_{\omega}|P_n|r^{-1}_n=\infty$, and thus $\mathfrak P$ has infinite length,
contradicting our hypothesis.
 In the former case, by taking the $\omega$--limit of these
balls, we have that $\mathfrak P$ passes through a ball of radius
$\lim_{\omega}\frac{r_n-1}{r_n}=1$.
Taking $K>3$, the claim is proved.

\textbf{Unbounded $\simp\mathbf X$ implies cut-ball:}  By
\cite[Theorem~6.9]{HagenBoundary}, for each $n\geq 0$, we
have $r_n\geq 0$ and combinatorial geodesic rays $\alpha_n,\beta_n$ emanating
from $x_o$ with $\dive{\alpha_n}{\beta_n}(r)\geq nr$ for all $r\geq r_n$. From
this point the argument then finishes exactly as above.

\textbf{Bounded $\simp\mathbf X$ implies no cut-ball:} First we show: if $\simp\mathbf X$ is bounded and $|\simp\mathbf X|>1$, then the combinatorial metric on $\mathbf X^{(1)}$ has linear divergence function.

Let $a,b,c\in\mathbf X^{(1)}$, with $\done(a,b)\leq n$ and
$\done(\{a,b\},c)=r>0$. Choose $\delta\in(0,\frac{1}{2})$ and $\kappa\geq 0$.
Let $\mu$ be the median of $a,b,c$ and let $\gamma$ be a bi-infinite 
path 
with $\gamma(0)=\mu$ and $\gamma(-t_a)=a,\gamma(t_b)=b$ for 
$t_a,t_b\in[0,n]$ and with the property that both 
$\gamma|_{(-\infty,0]}$ and $\gamma|_{[0,\infty)}$ are geodesic rays.  
Here we have used the combinatorial geodesic-ray completeness hypothesis.

Since $X$ is finite-dimensional and locally finite, the hypothesis of
\cite[Theorem~6.8]{HagenBoundary} is satisfied, and thus, since $\simp\mathbf X$
is bounded, the divergence of $\gamma$ is bounded above by a linear function
with uniform additive and multiplicative constants. Note that to use
\cite[Theorem~6.8]{HagenBoundary} implicitly requires $|\simp\mathbf X|\geq 2$,
since a pair of distinct infinite geodesic rays is required in order to apply that theorem.

If $\done(\mu,c)>\delta r-\kappa$, then the subpath of $\gamma$ joining $a$ to
$b$ has length $t_a+t_b\leq n$ and avoids the $(\delta r-\kappa)$--ball about $c$. In this
case we thus have that $\divergence{a}{b}{c}{\delta}{\kappa}=t_a+t_b\leq n$.

Hence, we restrict our attention to the alternate case where
$\done(\mu,c)\leq\delta r-\kappa$. Let $T=2\max\{t_a,t_b\}$. Note that since
$\delta<\frac{1}{2}$ we have $\min\{t_a,t_b\}\geq\frac{r}{2}$. Since, as noted
above, $\gamma$ has linear divergence, there exists a path $P$ connecting
$\gamma(-T)$ to $\gamma(T)$ whose length is linear in $T$ and which avoids the
ball of radius $T$ about $\gamma(0)$, i.e., for each $p\in P$ we have
$\done(p,\mu)\geq T$. Since $\done(\mu,c)\leq\delta r-\kappa$, the triangle
inequality implies that for each $p\in P$ we have $\done(p,c)\geq T-\delta r -
\kappa\geq \delta r-\kappa$. Thus concatenating $P$ with the subpaths of
$\gamma$ from $\gamma(-T)$ to $\gamma(-t_a)$ and from $\gamma(t_b)$ to
$\gamma(T)$ (which are each of length at most $n$), we get a path $P'$
connecting $a$ to $b$, which is of linear length and which avoids the $(\delta
r-\kappa)$--ball about $c$.

Hence, for any choices of $a,b,c$ we have obtained that
$\divergence{a}{b}{c}{\delta}{\kappa}$ is bounded above by a linear function
with uniform constants, as desired.

The remainder of the argument is a routine application of linear divergence. Fix
$\ascone{\mathbf X}{\mathbf x}{(s_n)}{\omega}$. We want to show that for each
closed ball $\mathcal B$ in $\ascone{\mathbf X}{\mathbf x}{(s_n)}{\omega}$ and
distinct points
$\mathbf a,\mathbf b\in\ascone{\mathbf X}{\mathbf x}{(s_n)}{\omega} - \mathcal
B$, there exists a path in $\ascone{\mathbf X}{\mathbf
x}{(s_n)}{\omega}-\mathcal B$ joining $\mathbf a$ to $\mathbf b$. To do this we
fix sequences $(a_n),(b_n)$ representing $\mathbf a, \mathbf b$, respectively,
and let $(c_n)$ be a sequence representing $\mathbf c$, the center of the ball
$\mathcal B$.
Since the divergence of $\mathbf X$ is linear, following the proof
of~\cite[Lemma~3.14]{DrutuMozesSapir} shows that no ball in $\ascone{\mathbf
X}{\mathbf x}{(s_n)}{\omega}$ about $\mathbf c$ of radius less than $\delta$ can
separate $\mathbf a$ from
$\mathbf b$.  Any ball of radius at least $r$ about $\mathbf c$ contains an
element of $\{\mathbf a,\mathbf b\}$ and hence cannot separate those points.
\end{proof}

The following corollary is a characterization of wide
cube complexes in a slightly more general framework than we shall
later apply.  Cocompactness of the action of $\Aut(\mathbf X)$ is needed to find
a cut-point in an asymptotic cone given a cut-ball in some other asymptotic
cone.
 For the converse, the failure to be wide implies that the simplicial boundary
is unbounded, and this assumption is unnecessary.  We have hypothesized
finite-dimensionality so that $\mathbf X$ with the CAT(0) metric is
quasi-isometric to $\mathbf X^{(1)}$, which is the natural setting for working
with the simplicial boundary.

\begin{cor}\label{cor:wide}
Let $\mathbf X$ be a locally finite, geodesically complete, finite-dimensional
CAT(0) cube complex on which $\Aut(\mathbf X)$ acts cocompactly.  Then $\mathbf
X$ is wide if and only if $\simp\mathbf X$ is bounded.
\end{cor}

\begin{proof}
By geodesic completeness, every point of $\mathbf X$
lies in a bi-infinite geodesic.  By Theorem~\ref{thm:unconstrictedCC}, if
$\simp\mathbf X$ is
unbounded then some
asymptotic cone of $\mathbf X$ has a finite cut-ball.  More precisely, there
exists $\delta\geq 0$ and points $\mathbf
a,\mathbf b,\mathbf c$ in some asymptotic cone, with $\mathbf
d_{\omega}(\mathbf c,\{\mathbf a,\mathbf b\})>3\delta$, such that the closed
$\delta$-ball about $\mathbf c$ separates $\mathbf a$ from $\mathbf
b$.  By~\cite[Lemma~3.16]{DrutuMozesSapir}, $\mathbf X$ is not wide.

Conversely, if $\mathbf X$ is not wide, then $\simp\mathbf X$ is unbounded, by
Theorem~\ref{thm:unconstrictedCC}.
\end{proof}

In the event of a proper, cocompact, essential group action, that
$\mathbf X$ is wide corresponds to $\simp\mathbf X$ being connected can be
seen without directly analyzing the asymptotic cones. 

\begin{thm}\label{thm:cocompactwide}
Let $\mathbf X$ be a CAT(0) cube complex
on which the group $G$ acts properly and cocompactly.  Then
$\mathbf X$ is wide if and only if $\simp\mathbf X$ is connected.

Hence, if $G$ is a cocompactly cubulated group, then $G$ is wide if and only if
$G$ acts geometrically on a CAT(0) cube complex with connected
simplicial boundary.
\end{thm}

\begin{proof}
    Throughout the proof, by appealing to~\cite[Proposition~3.5]{CapraceSageev} and
    Proposition~\ref{lem:boundaryofcore}, we assume that $G$ acts essentially on
    $\mathbf X$.

    The end-points of the axis stabilized by any rank-one element in
    $G$ are isolated $0$--simplicies in the the boundary.  Thus,
    $\simp\mathbf X$ is connected if and only if $G$ does not contain
    any rank-one elements.  By the rank-rigidity
    theorem~\cite[Theorem~6.3]{CapraceSageev} $G$ does not contain any
    rank-one elements if and only if there exists unbounded convex
    subcomplexes $\mathbf X_1,\mathbf X_2\subset \mathbf X$ satisfying
    $\mathbf X=\mathbf X_1 \times \mathbf X_2$.  If the space $\mathbf X$
    is such a direct product then it has linear divergence; if it has 
    a rank-one element then its divergence is superlinear.  By
    \cite[Proposition~1.1]{DrutuMozesSapir} a space linear divergence
    if and only if it is wide.
\end{proof}

\section{Characterizing thickness of order 1}\label{sec:thickoforderone}
Throughout this section, $\mathbf X$ will denote a CAT(0) cube complex on which
a group $G$ acts properly and cocompactly.  Let $\mathcal I$ denote the subcomplex of
$\simp\mathbf X$ consisting of all isolated 0-simplices.  Since maximal
simplices of $\simp\mathbf X$ are visible, each $v\in\mathcal I$
is represented by a combinatorial geodesic ray that is rank-one in the sense
of~\cite{HagenBoundary}, and conversely, each rank-one geodesic ray represents
an isolated 0-simplex of $\simp\mathbf X$.  In this section, we adopt the following notation: if $Y\subset\mathbf X$ is a subspace, we denote by $\widehat Y$ its cubical convex hull.

\subsection{Simplicial boundaries of algebraically thick cube complexes}\label{sec:thick}
A \emph{cubical flat sector} is a CAT(0) cube complex of the form $\reals^p\times[0,\infty)^q$, with $p+q\geq 2$, tiled in the standard Euclidean fashion by unit $(p+q)$-cubes.  The class of cubical flat sectors includes cubical orthants, half-flats, and flats of dimension at least 2.

Our first theorem describes simplicial boundaries of CAT(0) cube complexes admitting geometric actions by groups that are algebraically thick of order 1:

\begin{thm}\label{thm:algebraicallythickoforder1}
Let $G$ act properly and cocompactly on a fully visible CAT(0) cube complex $\mathbf X$, and suppose that $G$ is algebraically thick of order 1 relative to a collection $\mathbb G$ of quasiconvex wide subgroups.  Then $\mathcal I\neq\emptyset$ and $\simp\mathbf X$ has at least one $G$-invariant positive-dimensional component.
\end{thm}

\begin{rem}
Note that full visibility of $\mathbf X$ is hypothesized.  This hypothesis can be removed if Conjecture~\ref{conj:cocompactfullyvisible} is true.  The conclusion of the above theorem holds in slightly more generality, namely when $\mathbf X$ is thick relative to a $G$-invariant collection of convex subcomplexes with connected simplicial boundaries.  We also note that in many examples the $G$-invariant component is in fact the unique positive-dimensional component.
\end{rem}

\begin{proof}[Proof of Theorem~\ref{thm:algebraicallythickoforder1}]
Throughout the proof, by appealing to~\cite[Proposition~3.5]{CapraceSageev} and
Proposition~\ref{lem:boundaryofcore}, we assume that $G$ acts essentially on
$\mathbf X$.  We can assume that $G$ is one-ended, since otherwise
$\simp\mathbf X$ is disconnected and $G$ is not thick.  

Necessarily, $\mathcal I\neq\emptyset$.  To see this, note first that if
$\mathcal I=\emptyset$, then $G$ cannot contain a rank-one isometry of
$\mathbf X$, since the set of hyperplanes crossing an axis for such an element would
represent a pair of isolated 0-simplices of $\simp\mathbf X$.  In such a case, by
rank-rigidity, $\mathbf X$ decomposes as the product of two convex subcomplexes,
each of which has nonempty simplicial boundary, and therefore $\simp\mathbf X$
decomposes as a nontrivial simplicial join.  Hence we have, in
particular, that $\simp\mathbf X$ is bounded and hence connected;
thus by Theorem~\ref{thm:cocompactwide} $G$ is wide, i.e., strongly algebraically thick of order 0, a contradiction.

\textbf{Representing $\mathbb G$ in $\mathbf X$:}  Fix a 0-cube
$x_o\in\mathbf X$.  For each $H\in\mathbb G$, the orbit $Hx_o$ is
quasiconvex; denote by $S_H$ the convex hull of this orbit.  By
quasiconvexity, $S_H$ is contained in a uniform neighbourhood of
$Hx_o$, and therefore $S_H$ is an $H$-cocompact CAT(0) cube complex.
Let $\mathcal S=\{gS_H:g\in G,\,H\in\mathbb G\}$.  Since $G$ acts
cocompactly on $\mathbf X$, the set $\mathcal S$ coarsely covers
$\mathbf X$ (denote by $\tau$ a constant such that the
$\tau$-neighborhoods of the various $S_H$ together cover $\mathbf X$).
Now, each $H\in\mathbb G$ is wide, so since the property of being wide
is quasi-isometry invariant, $S_H$ is likewise wide.  By
Theorem~\ref{thm:cocompactwide}, $\simp S_H$ is a connected,
positive-dimensional subcomplex of $\simp\mathbf X$.  Finally, since
$G$ is algebraically thick with respect to $\mathbb G$, $\mathbf X$ is
thick with respect to $\mathcal S$, i.e., for all $S,S'\in\mathcal S$,
there exists a sequence $S=S_0,S_1,\ldots,S_k=S'$ such that for
$1\leq i\leq k$, the intersection $\mathcal N_{\tau}(S_{i-1})\cap \mathcal N_{\tau}(S_i)$ is $\tau$-path-connected and coarsely unbounded.  We now make a series of modifications to $\mathcal S$ to put it in a particularly nice form.

\textbf{Thickness relative to flat sectors:}
As stated above, for each $S\in\mathcal S$, $\simp S$ is connected.  Let $\mathcal F_{S}$ be the set of all cubical flat orthants in $S$ of dimension exceeding 1.  If $F,F'\in\mathcal F_S$, then the simplices $v_F,v_F'$ of $\simp S$ are joined by a sequence $v_F=v_0,\ldots,v_n=v_{F'}$ of positive-dimensional simplices such that $v_{i-1}\cap v_i\neq\emptyset$ for $1\leq i\leq n$.  By full visibility of $S$ -- here we use the fact that full visibility is inherited by convex subcomplexes, by definition -- there exists a sequence $F=F_0,\ldots,F_n=F'$ such that $F_i$ and $F_{i-1}$ are crossed by infinitely many common hyperplanes, for $1\leq i\leq n$.  Employing the \emph{Flat Bridge Trick} (Lemma~\ref{lem:flatbridgetrick} below), we can assume that $\widehat F_i\cap\widehat F_{i-1}$ is path-connected and unbounded for all $i$.

For $S,S'\in\mathcal S$, suppose that $F\in\mathcal F_S$ and
$F'\in\mathcal F_{S'}$ are flat orthants of dimension at least 2.  Choose a sequence $S=S_0,\ldots,S_k=S'$ such that the intersection of consecutive terms is coarsely connected and unbounded, i.e., $\mathcal H(S_i)\cap\mathcal H(S_{i+1})$ is infinite for all $i$.  Applying the flat Bridge Trick, we find a cubical flat sector $F_i$ (containing a half-flat) such that the intersection of $F_i$ with each of $S_i$ and $S_{i+1}$ contains a flat orthant.  Hence $F$ can be thickly connected to $F'$ by a chain of flat orthants.  Moreover, any new flat orthant added during an application of the Flat Bridge Trick is coarsely contained in a flat orthant belonging to some $S\in\mathcal S$, and can thus be thickly connected to any other such flat orthant by a sequence of flat orthants.

\textbf{Conclusion:}  Thus $\mathbf X$ contains a collection $\mathcal S'$ of convex subcomplexes $\widehat F$, where each $F$ is a flat orthant of dimension at least 2, such that $\mathbf X=\mathcal N_{\tau}(\cup_{\widehat F\in\mathcal S'}\widehat F)$ and, for all $\widehat F,\widehat F'\in\mathcal S'$, there exists a sequence $F=F_0,F_1,\ldots,F_k=F'$ such that $\widehat F_i\in\mathcal S'$ for all $i$ and $\widehat F_i\cap\widehat F_{i-1}$ is connected and unbounded for $1\leq i\leq k$.

Now, for each $\widehat F\in\mathcal S'$, let $\chi_{\widehat F}$ be the image in $\simp\mathbf X$ of the simplicial boundary of $\widehat F$.  Each $\chi_{\widehat F}$ is a positive-dimensional simplex by Proposition~\ref{prop:orthantsyieldsimplices}.  The above discussion shows that $\cup_{\widehat F\in\mathcal S}\chi_{\widehat F}$ is a connected subcomplex of $\simp\mathbf X$.  Finally, for each $\widehat F\in\mathcal S'$, either $F$ is a flat orthant of some $S\in\mathcal S$, or $F$ is a flat orthant such that, for some $S,S'\in\mathcal S$, each of the intersections $\widehat F\cap S$ and $\widehat F\cap S'$ is unbounded and path-connected.  Since $\mathcal S$ is $G$-invariant -- it is the set of $G$-translates of convex hulls of the various $Hx_o$ -- the set of all such orthants $F$, and hence $\cup_{\widehat F\in\mathcal S}\chi_{\widehat F}$, is likewise $G$-invariant.  The component containing this subcomplex is thus $G$-invariant.
\end{proof}

\begin{lem}[The Flat Bridge Trick]\label{lem:flatbridgetrick}
Let $\mathbf X$ be finite-dimensional, locally finite a CAT(0) cube complex, and let $S,S'\subseteq\mathbf X$ be wide convex subcomplexes with connected simplicial boundaries, such that there exist flat sectors $F\subseteq S,F'\subseteq S'$, and $\mathcal H(S)\cap\mathcal H(S')$ is infinite.  Then there exists a sequence $F=F_0,F_1,\ldots,F_n=F'$ of flat sectors such that for all $i$, the intersection $\widehat F_i\cap\widehat F_{i+1}$ is unbounded and path-connected.

\end{lem}

\begin{proof}
\textbf{Flat bridges for pairs of flat sectors:}  First, let $F_i,F_{i+1}$ be flat sectors such that $\simp\widehat F_j$ contains a simplex $u_j$, for $j\in\{i,i+1\}$, such that $u_i\cap u_{i+1}\neq\emptyset$.  Then there exists for each $i$ an infinite set of hyperplanes $H$ such that $H$ crosses both $\widehat F_i$ and $\widehat F_{i+1}$.  Hence $\mathcal H(F_i)\cap\mathcal
H(F_{i+1})$ contains a boundary set $\mathcal V$ representing a 0-simplex $v\in u_i\cap u_{i+1}$.  Choose disjoint minimal boundary sets $\mathcal W_i\subset\mathcal H(F_i)$ and $\mathcal W_{i+1}\subset\mathcal H(F_{i+1})$.  Then the smallest set of
hyperplanes containing $\mathcal W_i$ and $\mathcal W_{i+1}$ that is closed
under separation is of the form $\mathcal H(\alpha)$ for some bi-infinite
geodesic $\alpha$ containing an infinite ray in $F_i$ and an infinite
ray in $F_{i+1}$.

Now, since $\mathcal V\cap\mathcal W_i=\mathcal V\cap\mathcal
W_{i+1}=\emptyset$, for any geodesic ray $\beta$ in $F_i$ or $F_{i+1}$ with
initial point on $\alpha$ and $\mathcal H(\beta)\subseteq\mathcal V$, every
hyperplane dual to a 1-cube of $\beta$ crosses every hyperplane dual to a
1-cube of $\alpha$, and thus there is an isometric embedding
$\alpha\times\beta\rightarrow\mathbf X$.  The half-flat $\alpha\times\beta$ has the property that its convex hull contains a flat orthant in $F_i$ and a flat orthant in $F_{i+1}$, since $\beta$ has the same set of dual hyperplanes as a ray in $F_i$ and a ray in $F_{i+1}$.  The half-flat $\alpha\times\beta$ is a flat sector $F_{i+\frac{1}{2}}$, and $\widehat F_{i+\frac 12}$ must have unbounded convex intersection with $\widehat F_i,\widehat F_{i+1}$.

\textbf{Flat bridges for subcomplexes with connected boundaries:} The
same argument can be applied to arbitrary wide convex subcomplexes
$S,S'$ with $\mathcal H(S)\cap\mathcal H(S')$ infinite.  Indeed, there
exist combinatorial geodesic rays $\gamma,\gamma'$ in $S,S'$
respectively, such that $\mathcal H(\gamma)=\mathcal H(\gamma')$.
Since $S,S'$ are wide, $\gamma$ and $\gamma'$ can be chosen to lie in
flat sectors $F_i\subseteq S,F_{i+1}\subseteq S'$, and we argue as
above.  If $F\subseteq S,F'\subseteq S'$ are the given flat sectors,
then since $S$ has connected simplicial boundary, we can chain $F$ to
$F_i$ and $F_{i+1}$ to $F'$ by thickly connecting sequences of convex
hulls of flat sectors, and the proof is complete.
\end{proof}

The Flat Bridge Trick is also used in the next section.

\subsection{Identifying thickness and algebraic thickness of order 1}
The goal of this section is to prove Theorem~\ref{thm:thickoforderone}, which allows one to identify thickness of order 1, and algebraic thickness of order 1, of a group $G$ acting geometrically on the cube complex $\mathbf X$ by examining the action of $G$ on the simplicial boundary and on the visual boundary.  For thickness, one only need examine the action on the simplicial boundary, while a convenient statement of hypotheses implying algebraic thickness also involves the action on the visual boundary.

In the following, $f\co\visual\mathbf X\rightarrow\simp\mathbf X$ denotes the surjection defined in Section~\ref{sec:limitvisual}.

\begin{thm}\label{thm:thickoforderone}
Let $G$ be a group which acts properly and cocompactly on a fully visible CAT(0)
cube complex $\mathbf X$.  If $\mathcal I\neq\emptyset$ and $\simp\mathbf X$ has a positive-dimensional $G$-invariant connected subcomplex $\mathfrak C$, then $G$ is thick of order 1 relative to a collection of wide subsets.

Suppose, further, that there is a finite collection $\mathcal A$ of bounded subcomplexes of $\mathfrak C$ such that:
\begin{enumerate}
 \item The stabilizer $H_A$ of $A$ is quasiconvex for all $A\in\mathcal A$.
 \item For all $A\in\mathcal A$, the set $f^{-1}(A)\subset\visual\mathbf X$ is contained in the limit set of $H_A$.
 \item $\mathfrak C=\bigcup_{g\in G,A\in\mathcal A}gA$ and $f^{-1}(\mathfrak C)$ is contained in the limit set of the subgroup of $G$ generated by $\{H_A:A\in\mathcal A\}$.
\end{enumerate}
Then $G$ is algebraically thick of order 1 relative to the collection $\{H_A:A\in\mathcal A\}$ of wide subgroups.
\end{thm}

\begin{rem}
Note that $\simp\mathbf X$ has a connected, positive-dimensional, $G$-invariant subcomplex if and only if $\simp\mathbf X$ has a positive-dimensional, $G$-invariant component.  Theorem~\ref{thm:thickoforderone} is stated in terms of connected subcomplexes, rather than components, since $(3)$ is rarely satisfied if $\mathfrak C$ is required to be an entire component.

Theorem~\ref{thm:thickoforderone} could be stated in terms of the
$G$-action on $\simp\mathbf X$ alone, with each hypothesis about limit
sets in $\visual\mathbf X$ replaced by the appropriate statement about
limit subcomplexes in $\simp\mathbf X$: the appropriate modification
of condition $(2)$ would require each $A$ to lie in the limit
subcomplex of $H_A$ and that of $(3)$ would require $\mathfrak C$ to
lie in the limit subcomplex of $\langle\{H_A\}\rangle$.
\end{rem}

\begin{proof}[Proof of Theorem~\ref{thm:thickoforderone}]
Suppose that $\simp\mathbf X-\mathcal I$ is nonempty and has a $G$-invariant positive-dimensional connected subcomplex $\mathfrak C$.
Since $\dimension(\mathbf X)<\infty$, there is no infinite family of pairwise-crossing hyperplanes and hence every simplex of $\simp\mathbf X$ is contained in a finite-dimensional maximal simplex, by Theorem~3.14.(2) of~\cite{HagenBoundary}.  Let $v$ be a maximal positive dimensional simplex of $\mathfrak C$.  From Proposition~\ref{prop:simplicesflats}, it
follows that there exists an isometrically embedded maximal flat orthant $F_v\cong[0,\infty)^n\subset\mathbf X$, for some $n\geq 2$, whose boundary is $v$.  Hence the set $\mathcal F$ of \emph{flat sectors} whose convex hulls represent positive-dimensional connected subcomplexes of $\mathfrak C$ is nonempty.  Moreover, $\mathcal F$ is $G$-invariant, since $\mathfrak C$ is.  To see this, note that for all $F\in\mathcal F$, the inclusion $g\widehat F\hookrightarrow\mathbf X$ induces an inclusion of simplicial boundaries whose image lies in $g\mathfrak C=\mathfrak C$.  By definition, $gF\in\mathcal F$.  Hence, by cocompactness, there exists
$\tau\geq 0$ such that $\mathbf X=\bigcup_{F\in\mathcal F}N_{\tau}(F)$.

For each $F\in\mathcal F$, let $\widehat F$ be the convex hull of
$N_{\tau}(F)$.  Since $\widehat F$ is convex, it is a CAT(0) cube complex, and
moreover, $\simp\widehat F$ is bounded and positive-dimensional, being a connected subcomplex of the simplicial boundary of a cubical flat of dimension at least 2 and containing the boundary of a flat orthant. Thus $\widehat F$ is wide, by
Theorem~\ref{thm:unconstrictedCC}.
We conclude that $\{\widehat F:F\in\mathcal F\}$ is a set of convex (and hence
uniformly quasiconvex) wide subcomplexes that covers $\mathbf X$.  By
definition, each $\widehat F$ has the property that every $f\in\widehat F$ is
contained in a bi-infinite combinatorial geodesic, and therefore each point in
$\widehat F$ is uniformly close to a bi-infinite CAT(0) quasigeodesic.  To
conclude that $\{\widehat F\}$ is uniformly wide, it remains to check that no
ultralimit of a sequence in $\{\widehat F\}$ has a cut-point; this is the
content of Lemma~\ref{lem:uniformlywide} below.

Let $p,q\in\mathbf X$ be 0-cubes, and choose $F,F'\in\mathcal F$ so that
$p\in N_k(F),q\in N_k(F')$.  By assumption, there exists a sequence
$v_F=u_0,u_1,\ldots,u_n=v_{F'}$ of maximal simplices in $\mathfrak C$ such
that $u_i\cap u_{i+1}\neq\emptyset$ for all $i$.  For each $i$, let $\widehat
F_i\in\mathcal F$ be the convex hull of a maximal
flat sector containing an orthant representing $u_i$.  If for each $i$, there exist
geodesic rays $\gamma\subset\widehat F_i,\gamma'\subset\widehat F_{i+1}$ that
fellow-travel at distance $\tau$, then we have thickly connected $p$ to $q$
using convex hulls of flat orthants.  (Note that the intersection of
CAT(0) $\tau$-neighborhoods of convex subcomplexes is convex.)  Otherwise, for
any pair $F_{i},F_{i+1}$ not containing such a pair of geodesic rays, we construct a third flat orthant $F_{i+\frac{1}{2}}$ whose convex hull has unbounded intersection with $\widehat F_i$ and $\widehat F_{i+1}$, using the Flat Bridge Trick.  Adding the new $\widehat F_{i+\frac{1}{2}}$ for each $i$ yields the desired thickly connecting
sequence.

Thus far, we have shown that for any two points $x,y\in\mathbf X$, and any $\widehat F_0,\widehat F_n$ with $x\in\mathcal N_{\tau}(\widehat F_0)$ and $y\in\mathcal N_{\tau}(\widehat F_n)$, there exists
a sequence $\widehat F_0,\ldots,\widehat F_n$ of subcomplexes, where each
$\widehat F_i$ is the convex hull of a $d$-dimensional flat, where $d\geq 2$,
such that $\mathcal N_{\tau}(\widehat F_i)\cap\mathcal N_{\tau}(\widehat
F_{i+1})$ is unbounded and path-connected for each $i$.  In so doing, we have verified that
$\mathbf X$, and therefore $G$, is thick of order at most 1.  Since $\mathcal I\neq\emptyset$, $\simp\mathbf X$ is disconnected and hence $G$ contains a rank-one isometry of $\mathbf X$, whence $G$ is not unconstricted and is therefore thick of order exactly 1.

\textbf{Obtaining algebraic thickness:}  Fix a base 0-cube $x_o\in\mathbf X$, and let $\mathbf C_A$ denote the cubical convex hull of the quasiconvex orbit $H_Ax_o$, for each $A\in\mathcal A$.  Then $\mathbf C_A$ is an $H_A$-cocompact subcomplex, by quasiconvexity.  By passing to the $H_A$-essential core of $\mathbf C_A$, if necessary, we may assume that $\mathbf C_A$ is a CAT(0) cube complex on which $H_A$ acts properly, cocompactly, and essentially.

Let $u\subseteq\simp\mathbf C_A$ be a simplex represented by $\mathcal H(\gamma)$ for some geodesic ray $\gamma$ emanating from $x_o$.  Since $H_A$ acts cocompactly on $\mathbf C_A$, $\gamma$ is contained in the limit of a sequence of $H_A$-periodic geodesics, from which it is easily verified that $u$ is a limit simplex of $H_A$.  Conversely, if $u$ is a limit simplex of $H_A$ that is not contained in $\mathbf C_A$, then $u$ is represented by $\mathcal H(\gamma)$ for some geodesic ray $\gamma$ that contains points arbitrarily far from $\mathbf C_A$.  There exists a sequence $(h_i\in H_A)$ such that $\mathcal H(\gamma)$ is the set of hyperplanes $H$ such that $H$ separates all but finitely many $h_ix_o$ from $x_o$.  Since $\mathbf C_A$ is convex and $\gamma$ contains points arbitrarily far from $\mathbf C_A$, infinitely many $H\in\mathcal H(\gamma)$ separate points of $\gamma$ from $\mathbf C_A$, whence $(h_ix_o)$ contains points not in $\mathbf C_A$, contradicting the fact that the latter contains $H_Ax_o$ by definition.  Thus $\simp\mathbf C_A$ coincides with the limit complex for $H_A$.  Our hypothesis that $A$ is contained in the limit complex for $H_A$ implies that $A\subseteq\simp\mathbf C_A$.  ($A$ is contained in the limit complex for $H_A$ since $f^{-1}(A)$ is contained in the limit set of $H_A$, by Lemma~\ref{lem:visualboundary}.)

Suppose $H_A$ contains a rank-one isometry $g$ of $\mathbf C_A$.  We shall show that this contradicts the fact that $A$ is bounded.  Let $a\in A$ be a visible 0-simplex (this must exist because $A$ contains a maximal simplex, at least one of whose 0-simplices must be visible, by the proof of~\cite[Theorem~3.19]{HagenBoundary}).  Then either the orbit $\langle g\rangle a$ is unbounded and contained in $A^{(1)}$, by Lemma~\ref{lem:unboundedinboundary} below, or $g$ fixes $a$.  The former possibility contradicts boundedness of $A$.  Hence $g$ fixes each $a\in A^{(0)}$.  This contradicts the fact that $g$ is rank-one unless $A$ consists of a pair of 0-simplices represented by a combinatorial geodesic axis for $g$, which is impossible since $A$ is connected.  Hence $H_A$ contains no rank-one elements.

By Corollary~B of~\cite{CapraceSageev}, $\mathbf C_A$ decomposes as a non-trivial product equal to the limit complex for $H_A$, and thus $\mathbf C_A$ has bounded simplicial boundary that contains $A$ and is contained in $\mathfrak C$.  We may thus assume that $\simp\mathbf C_A=A$, by adding to $A$, if necessary, any simplices of $\mathfrak C$ that lie in $\simp\mathbf C_A$ but not in $A$.

Since $A=\simp\mathbf C_A$ is connected, $\mathbf C_A$, and therefore $H_A$, is wide by Theorem~\ref{thm:cocompactwide}.\\

\emph{Verifying thick connectivity:}  Let $A,A'\in\mathcal A$ and let $H=H_A,H'=H_{A'}$.  Suppose that $A\cap A'\neq\emptyset$.  Then $\mathcal H(\mathbf C_A)\cap\mathcal H(\mathbf C_{A'})$ is infinite, and by cocompactness of the actions of $H$ on $\mathbf C_A$ and $H'$ on $\mathbf C_{A'}$, it follows that $H_A\cap H'_A$ is infinite (the same holds for conjugates of $H,H'$: if the corresponding $G$-translates of $A,A'$ have nonempty intersection, then the corresponding conjugates of $H$ and $H'$ have infinite intersection).  Conversely, if $H\cap H'$ is infinite, then the intersection contains a hyperbolic isometry of $\mathbf X$, and thus each of $\mathbf C_A$ and $\mathbf C_{A'}$ contains a bi-infinite combinatorial geodesic such that these two geodesics are parallel, and hence represent the same simplices of $\mathfrak C$.  Thus $A\cap A'\neq\emptyset$.  Now, without loss of generality, $\cup_{A\in\mathcal A}A$ is a connected subcomplex of $\mathfrak C$, which can be achieved by choosing conjugacy class
representatives of the various $A\in\mathcal A$ so that the corresponding subcomplex is connected, and replacing $\mathcal A$ by this collection of subgroups.  Hence, for any $A,A'\in\mathcal A$, there exists a
sequence $A=A_0,\ldots,A_n=A'$ such that
$A_i\in\mathcal A$ for all $i$ and $A_i\cap A_{i+1}\neq\emptyset$ for $0\leq i\leq n-1$.  Hence $H_{A_i}\cap H_{A_{i+1}}$ is infinite for $0\leq i\leq n-1$.\\

\emph{Verifying that $\cup_{A\in\mathcal A}H_A$ generates:}  To complete the proof of algebraic thickness of $G$, it suffices to show that $G'=\langle\{H_A:A\in\mathcal A\}\rangle$ has finite index in $G$, and, to this end, we will verify that there exists $R\geq 0$ such that $\mathbf X=\mathcal N_R(G'(\cup_{A\in\mathcal A}\mathbf C_A))$.

If the preceding equality does not hold, then for all $r\geq 0$, there exists $x_r\in\mathbf X$ such that $\dtwo(x_r,h\mathbf C_A)>r$ for all $A\in\mathcal A$ and all $h\in G'$.  By cocompactness of the $G$-action on $\mathbf X$, we may choose $\{x_r\}_{r\geq 0}$ so that for some fixed $A\in\mathcal A$ and $g\in G-G'$, each $x_r\in g\mathbf C_A$ and $x_r$ converges to a point $x_{\infty}\in f^{-1}(g\simp\mathbf C_A)\subset\visual\mathbf X$.  Thus $f(x_\infty)\in g\simp\mathbf C_A$, but $x_{\infty}$ fails, by construction, to be a limit point of $G'$, a contradiction.  Hence $\mathbf X$ is contained in a finite neighborhood of the union of $G'$-translates of the various $\mathbf C_A$, and the stabilizer of each $\mathbf C_A$ is a subgroup of $G'$, whence $G'$ generates a finite-index subgroup of $G$, as required.
\end{proof}

\begin{lem}\label{lem:unboundedinboundary}
Let $H$ act properly and cocompactly on the CAT(0) cube complex $\mathbf C$, and let $g\in H$ be a rank-one element.  Then for any simplex $v$ of $\simp\mathbf X$ not stabilized by $g$, the orbit $\langle g\rangle v$ is unbounded in $(\simp\mathbf C)^{(1)}$.
\end{lem}

\begin{proof}
If $v\in\simp\mathbf C$ is an isolated 0-simplex not fixed by $g$, then $\langle g\rangle v$ is disconnected and therefore unbounded.  Hence it suffices to consider a visible 0-simplex $v$ that is not fixed by $g$.  Let $\alpha$ be a combinatorial geodesic axis for $g$, and let $\gamma$ be a ray representing $v$ and emanating from a 0-cube of $\alpha$.  Suppose that there exists $M<\infty$ such that $g^nv$ is joined to $v$ by a path of length at most $M$ in $\simp\mathbf C^{(1)}$, for all $n\in\integers$.

Then, applying the Flat Bridge Trick, we find for each $n\in\integers$ some $m\leq 2M$ and a sequence $F_0,\ldots,F_m$ of flat sectors such that $\widehat F_i\cap\widehat F_{i+1}$ is unbounded and path-connected for all $i$ and such that $v\in\simp\widehat F_0$ and $g^nv\in\simp\widehat F_m$; moreover it is no loss of generality to assume that $F_0$ is always the same flat sector, as can be seen by applying the Flat Bridge Trick between any pair of $F_0$s obtained as above.  Moreover, these flat sectors can be chosen, again using the Flat Bridge Trick, so that $g^nF_0=F_m$ and so that $\gamma$ contains a sub-ray lying in $F_0$.  Hence for all $n$, the distance from $\alpha$ to $F_m$ is uniformly bounded.  It follows that there exists $\eta\geq0$ such that there are hyperplanes $H,H'$ crossing the subpath $\alpha_n$ of $\alpha$ subtended by $\gamma$ and $g^n(\gamma)$, satisfying $\done(N(H)\cap\alpha,N(H')\cap\alpha)\geq |\alpha_n|M^{-1}-\eta$ and $H,H'$ crossing a common $\widehat F_i$.

Since $\alpha$ is a rank-one periodic geodesic, there exists $p<\infty$ such that if $H,H'$ are hyperplanes that cross $\alpha$, either $H\cap H'=\emptyset$ or the subpath of $\alpha$ between the 1-cubes dual to $H,H'$ has length at most $p$ (see~\cite[Section~2]{HagenBoundary}).

Note that if $H,H'$ are hyperplanes crossing $\alpha_n$, and $H,H'$ both cross $\widehat F_i$, and $H,H'$ do not cross, then $\done(N(H),N(H'))\leq q$ for some $q$ depending on $g$ (but independent of $H,H'$ and $F_i$).  Indeed, analyzing a minimal-area disc diagram bounded by geodesics in $N(H),N(H')$, the subtended part of $\alpha$, and a geodesic in a hyperplane of $F_i$ crossing $H,H'$ (as in~\cite[Section~2]{HagenBoundary}) shows that if $H,H'$ can be chosen arbitrarily far apart, then either there are hyperplanes $V,V'$, that cross $\alpha$ arbitrarily far apart and cross each other, or there are arbitrarily large isometric flat discs of the form $[0,N]^2$ with one side on $\alpha$.  This contradicts that $g$ is a rank-one isometry.

Now, there must exist hyperplanes $H,H'$ with $\done(N(H),N(H'))\geq|\alpha_n|M^{-1}$ that both cross $\widehat F_i$ for some $i$.  If $|\alpha_n|>M\max\{p,q\}$, then $H,H'$ cannot cross, and cannot cross a common flat sector, a contradiction.
\end{proof}

\begin{rem}\label{rem:extrahypotheses}
Let $G$ act properly, cocompactly and essentially on $\mathbf X$, and
suppose that $G$ is algebraically thick of order 1 with respect to a
finite collection $\mathbb G=\{H_A:A\in\mathcal A\}$ of quasiconvex,
wide subgroups, as in Theorem~\ref{thm:algebraicallythickoforder1}.
For each $A\in\mathcal A$, let $S_A$ be the $H_A$-cocompact convex
subcomplex constructed in the proof of
Theorem~\ref{thm:algebraicallythickoforder1}.  That proof shows that
$\cup_{A\in\mathcal A,g\in G}\simp S_A=\mathfrak C$ is
positive-dimensional, connected, and $G$-invariant.  Hence
$\simp\mathbf X$ has a positive-dimensional $G$-invariant component,
namely that containing $\mathfrak C$.  Moreover,
Theorem~\ref{thm:cocompactwide} implies that $\simp\mathbf X$ is
disconnected, since $G$ is not wide.

Now, since $H_A$ acts cocompactly on $S_A$ for all $A\in\mathcal A$,
and each $H_A$ is wide, each $\simp S_A$ is connected by
Theorem~\ref{thm:cocompactwide}.  Each $f^{-1}(\simp S_A)$ is contained in
the limit set of $H_A$, by the general fact that bi-infinite geodesics
in proper, cocompact spaces are limits of sequences of periodic
geodesics.  Likewise, since $\{H_A:A\in\mathcal A\}$ generates a
finite-index subgroup $G'\leq G$, by algebraic thickness, $G'$ acts
cocompactly on $\mathbf X$, which is the coarse union of
$G'$-translates of the various $S_A$, and thus
$f^{-1}(\cup_{A\in\mathcal A,g\in G'})=f^{-1}(\mathfrak C)$ is
contained in the limit set of $G'$.

This discussion shows that, if $G$ is algebraically thick of order 1 relative to a finite collection $\{H_A\}$ of quasiconvex, wide subgroups, then $\simp\mathbf X$ has a $G$-invariant component $\mathfrak C$, and a finite collection $\mathcal A$ of connected subcomplexes, satisfying hypotheses $(1)-(3)$ of Theorem~\ref{thm:thickoforderone}.  This conclusion is used in the proof of Theorem~\ref{thm:introthick}.
\end{rem}

The following characterization of convex hulls of flat sectors is immediate from the
definitions.

\begin{lem}\label{lem:convexhulls}
Let $\mathbf X$ be as in Theorem~\ref{thm:thickoforderone}.  For $n\geq 2$, let
$\mathcal A_n$ be the class of CAT(0) cube complexes $A\subseteq\mathbf X$ such that:

\begin{enumerate}
 \item $A$ contains an isometrically embedded cubical flat sector $F$ with
$2\leq\dimension F\leq n$.
 \item Every hyperplane of $A$ crosses $F$.
\end{enumerate}

Then the convex hull of each cubical flat sector $F$ in $\mathbf X$ belongs to
$\mathcal A_n$ for $n=\dimension\mathbf X$.
\end{lem}

\begin{lem}\label{lem:uniformlywide}
$\mathcal A_n$ is uniformly wide.  Equivalently, $\{A^{(1)}:A\in\mathcal
A_n\}$ is uniformly wide.
\end{lem}

\begin{proof}
The two assertions are equivalent since the collection of elements in $\mathcal
A_n$ have
uniformly bounded dimension and are thus each quasi-isometric to their
1-skeleta,
with uniform quasi-isometry constants (see,
e.g.,~\cite[Lemma~2.2]{CapraceSageev}).

Let
$(A_i)_{i\geq 0}$ be a sequence of cube complexes in $\mathcal A_n$, and
denote by $d_i$ the standard path-metric on $A_i^{(1)}$. Recall that $\mathcal
A_n$ is uniformly wide if and only if for any sequence
$(a_i\in A_i)_{i\geq 0}$, any positive sequence $(s_i)_{i\geq 0}$ with
$\lim_is_i=\infty$, and any ultrafilter $\omega$, the ultralimit
$\lim_{\omega}(A_i,a_i,\frac{d_i}{s_i})$ has no cut-point. We will prove that
exhibit a uniform linear bound on the divergences of the
$A_i^{(1)}$, from which the result then follows from
\cite[Proposition~1.1]{DrutuMozesSapir} which relates divergence and
wideness.

Let $a,b,c\in A_i^{(0)}$, with $d_i(a,b)=m$ and
$d_i(\{a,b\},c)=r$.  Choose $\delta\in(0,\frac{1}{2})$ and $\kappa\geq 0$.
Let $\mu$ be the median of $a,b,c$ and let $\gamma$ be a bi-infinite geodesic
with $\gamma(0)=\mu$ and $\gamma(-t_a)=a,\gamma(t_b)=b$ for $t_a,t_b\in(0,m)$.

If $d_i(\mu,c)>\delta r-\kappa$, then the subpath of $\gamma$ joining $a$ to
$b$ has length $m$ and avoids the $(\delta r-\kappa)$-ball about $c$.

Otherwise, $d_i(\mu,c)\leq\delta r-\kappa$, so that for any
$t\in\mathbb{R}$ we have
$d_i(\gamma(t),c)\geq t-\delta r+\kappa$.

Let $T=2\max\{t_a,t_b\}$. Since $\delta<\frac{1}{2}$ we have $T\geq \delta r -
\kappa$.

In the proof of~\cite[Lemma~6.5]{HagenBoundary}, it is shown that there exists
a combinatorial path $P$ connecting $\gamma(-T)$ to $\gamma(T)$, with the
property that each point
of $P$ lies at distance at least $T$ from $\mu$, having length
at most $5T+B$, where $B$ counts a certain set of hyperplanes separating $a$ or
$b$ from $\mu$, whence $B\leq 2T$. Since $d(\mu,c)\leq\delta r-\kappa$, the
triangle inequality implies that for each $p\in P$ we have $d(p,c)\geq T-\delta
r - \kappa\geq \delta r-\kappa$. Thus concatenating $P$ with the subpaths of
$\gamma$ from $\gamma(-T)$ to $\gamma(-t_a)$ and from $\gamma(t_b)$ to
$\gamma(T)$ (which are each of length at most $m$), we get a path $P'$
connecting $a$ to $b$, which is of linear length (at most $11T\leq 22m$) and
which avoids the $(\delta r-\kappa)$--ball about $c$.

Hence,  the
divergence of $a,b,c$ is at most $22m$.  Since the constants for divergence do
not depend on $i$, it follows immediately that the ultralimit of the sequence
$A_i^{(0)}$ does not have any cut-points.
\end{proof}

\subsection{Strong algebraic thickness}\label{sec:strongalgebraicthickness}
Obtaining quadratic divergence bounds using the result in~\cite{BD} requires strong thickness of order 1, which we obtain here as a consequence of strong algebraic thickness.  First, we note that little stands between the conclusion of Theorem~\ref{thm:thickoforderone} and the conclusion of strong algebraic thickness of a cocompactly cubulated group $G$:

\begin{prop}\label{prop:stalgthick}
Let $G$ act properly and cocompactly on the CAT(0) cube complex $\mathbf X$, and suppose that $G$ is algebraically thick of order $n\geq 1$ with respect to a finite collection $\{H_A:A\in\mathcal A\}$ of quasiconvex subgroups, each of which is strongly algebraically thick of order at most $n-1$.  Then $G$ is strongly algebraically thick of order $n$ relative to $\{H_A\}$.

Hence $\mathbf X$ and $G$ are strongly thick of order $n$ and have polynomial divergence of order at most $n+1$.
\end{prop}

\begin{proof}
By hypothesis, each $H_A$ is strongly algebraically thick of order $n-1$.  Moreover, since each $H_A$ acts with an $M_A$-quasiconvex orbit on $\mathbf X$, each $H_A$ is $M=\max_{A\in\mathcal A}$-quasiconvex in $G$.  In particular, each $H_A$ acts properly and cocompactly on a convex subcomplex $\mathbf C_A$ of $\mathbf X$ that is contained in the tubular $M$-neighborhood of the orbit $H_Ax_o$, where $x_o$ is a fixed 0-cube.  Now, if $H_A,H_{A'}$ are among the given finite collection, then by algebraic thickness, there exists a sequence $A=A_0,\ldots,A_n=A'$ such that for $0\leq i<n$, the intersection $H_{A_i}\cap H_{A_{i+1}}$ is infinite.  This implies that $H_{A_i}x_o\cap H_{A_{i+1}}x_o$ is infinite, whence $\mathbf C_{A_i}\cap\mathbf C_{A_{i+1}}$ is unbounded, and hence path-connected, since the intersection of convex subcomplexes of $\mathbf X$ is again convex.  Thus any geodesic segment starting and ending in $H_{A_i}x_o\cap H_{A_{i+1}}x_o$ lies inside of the $M$-neighborhood of $H_{A_i}x_o\cap H_{A_{i+1}}x_o$, whence $H_{A_i}\cap H_{A_{i+1}}$ is $M$-path-connected.  Finally, $\langle\{H_A\}\rangle$ has finite index in $G$ since $G$ is algebraically thick relative to $\{H_A\}$.  Thus $G$ is strongly algebraically thick of order at most $n$.  Obviously, if $G$ were strongly algebraically thick of order $k<n$, then $G$ would be thick of order $k$, a contradiction.  Hence $G$ is strongly algebraically thick of order exactly $n$.

It is now readily verified that $\mathbf X$ is strongly thick of order $n$ relative to the collection $\{g\mathbf C_A:g\in G, A\in\mathcal A\}$.  Thus $\mathbf X$, and $G$, have polynomial divergence of order at most $n+1$ by Corollary~4.17 of~\cite{BD}.
\end{proof}

\begin{cor}\label{cor:quaddiv}
Let $G$ act properly and cocompactly on the fully visible CAT(0) cube complex $\mathbf X$, and suppose that $\simp\mathbf X$ contains isolated 0-simplices and a $G$-invariant connected subcomplex $\mathfrak C=\cup_{g\in G,A\in\mathcal A}gA$, with $\mathcal A$ and the collection $\{H_A=\stabilizer_G(A):A\in\mathcal A\}$ as in Theorem~\ref{thm:thickoforderone}.  Then $G$ has quadratic divergence function.
\end{cor}

\begin{proof}
By Theorem~\ref{thm:thickoforderone}, $G$ is algebraically thick of order 1 relative to $\{H_A\}$, and thus strongly algebraically thick by Proposition~\ref{prop:stalgthick}, from which it also follows that the divergence of $G$ is at most quadratic.  On the other hand, if the divergence is subquadratic, then it is linear~\cite[Proposition~3.3]{KapovichLeeb}, which implies that $\simp\mathbf X$ is connected, contradicting the fact that the set of isolated 0-simplices is nonempty.
\end{proof}

\begin{exmp}[The Croke-Kleiner example]\label{exmp:crokekleinerthick}
The following example confirms that $\mathbf X$ satisfies the conclusions of Theorem~\ref{thm:algebraicallythickoforder1}, Theorem~\ref{thm:thickoforderone}, and Corollary~\ref{cor:quaddiv} when $\mathbf X$ is the universal cover of the Salvetti complex of a right-angled Artin group; here we have chosen the Croke-Kleiner group~\cite{CrokeKleiner}.  The same reasoning applies to any one-ended right-angled Artin group that is not a product, and these are known to be thick of order 1 and have quadratic divergence; see~\cite{BDM} and~\cite{BehrstockCharney}.

Let $\mathbf X$ be the universal cover of the Salvetti complex of the
right-angled Artin group
\[G\cong\langle a,b,c,d\mid[a,b],[b,c],[c,d]\rangle.\]
(This group is studied by Croke-Kleiner in~\cite{CrokeKleiner}.)  $\mathbf
X$ decomposes as a tree $T$ of spaces: the vertex-spaces are the obvious periodic 2-dimensional
cubical flats whose edges are labeled by generators, and the edge-spaces are
bi-infinite combinatorial geodesics representing cosets of $\langle
a\rangle,\langle b\rangle,\langle c\rangle,$ or $\langle d\rangle$.

Each flat $F$ corresponding to a vertex of $T$ is convex in $\mathbf X$, so
$\simp F$ embeds as a subcomplex in $\simp\mathbf X$.  Each $F$ is labeled by a
pair $(x,y)\in\{a,b,c,d\}^2$ of distinct generators corresponding to the labels
of the 1-cubes of the constituent squares of $F$.  The $x$-labeled combinatorial
geodesics in $F$ represent a pair of 0-simplices in $\simp F$, and the same is
true of the $y$-labeled geodesics, and $\simp F$ is a 4-cycle, being the join
of the $x$-labeled 0-simplices and the $y$-labeled 0-simplices.

Now, fix a root of $T$ and let $F_0$ be the corresponding flat; for
concreteness, take $F_0$ to be a flat labeled $(a,b)$. For each $n\geq 0$, let
$\mathcal S_n$ be the set of flats that correspond to vertices of $T$ at
distance $n$ from the vertex corresponding to $F_0$.  Each $F$ corresponding to
a vertex in $\mathcal S_1$ is labeled $(b,c)$, and for each such $F$,
$\simp\mathbf X$ contains a copy of $\simp F$ attached to $\simp F_0$ along the
pair of $b$-labeled 0-simplices.  If $F,F'\in\mathcal S_1$ are distinct, then
the images of their $c$-labeled 0-simplices are distinct.  By induction on $n$,
one checks that the union of the images of all $\simp F$ is connected; this union is clearly $G$-invariant.

Now, for each geodesic ray $\bar\gamma$ in $T$, there exists a rank-one
geodesic ray $\gamma$ in $\mathbf X$ such that $\gamma$ has nonempty
intersection with exactly those $F$ that correspond to vertices of
$\bar\gamma$.  Conversely, each rank-one ray in $\mathbf X$ projects to a
geodesic ray in $T$, and two rays projecting to the same ray in $T$ represent
the same 0-simplex of $\simp\mathbf X$.  Hence $\simp\mathbf X$ contains
exactly one isolated 0-simplex for each point of $\simp T$, as shown in
Figure~\ref{fig:eye}.

\begin{figure}[h]
\includegraphics[width=0.35\textwidth]{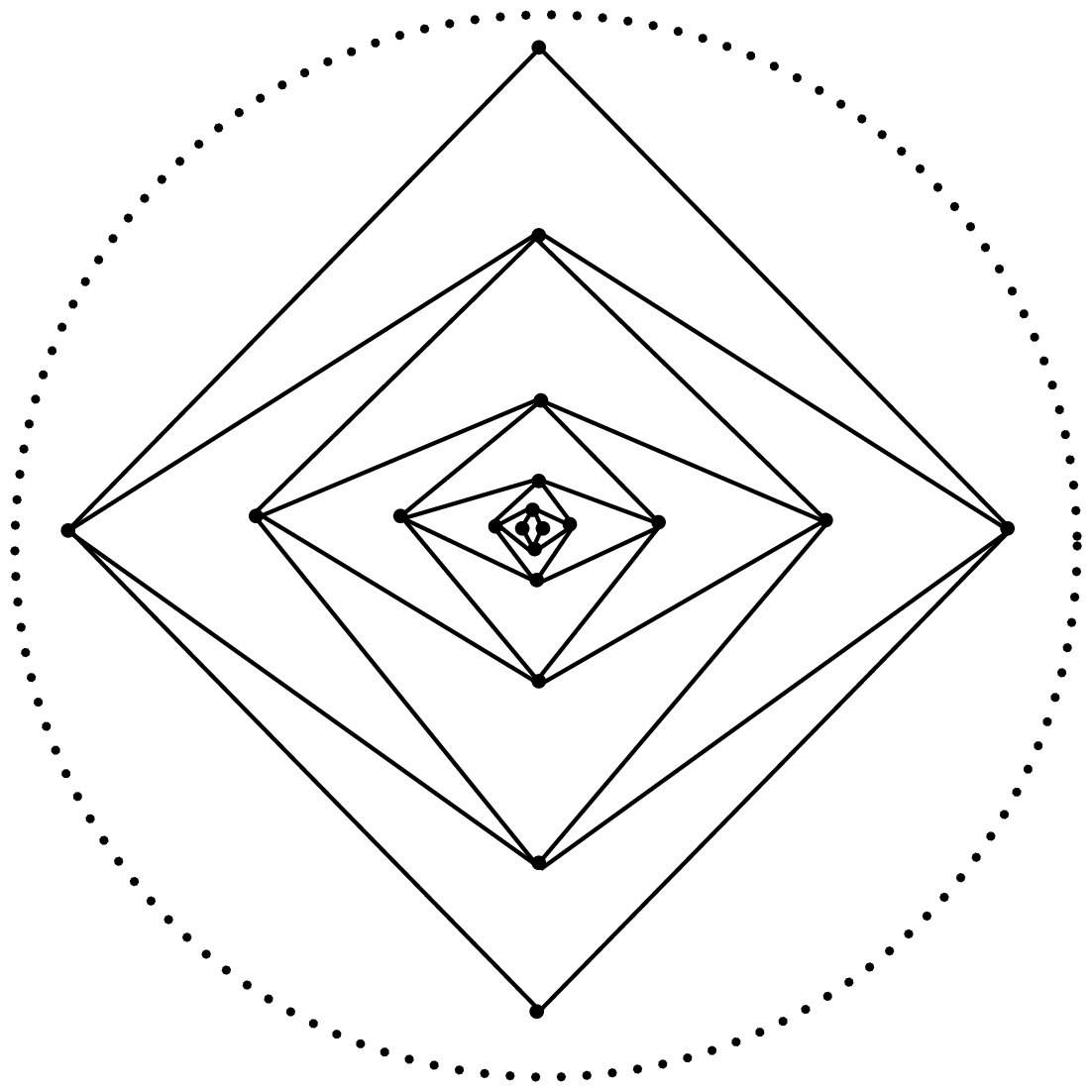}\\
\caption{Part of the simplicial boundary of the universal cover of the Salvetti
complex of the Croke-Kleiner group.}\label{fig:eye}
\end{figure}
\end{exmp}

\subsection{Necessary and sufficient conditions for thickness of order 1}

The following is a culmination of the results of this section.

\begin{thm}\label{thm:introthick}
Let $G$ act properly and cocompactly by isometries on the fully visible CAT(0) cube complex
$\mathbf X$.  If $G$ is algebraically thick of order 1 relative to a collection of quasiconvex wide subgroups, then $\simp\mathbf X$ is disconnected and contains a positive-dimensional, $G$-invariant connected component.  Conversely, if $\simp\mathbf X$ is disconnected, and has a positive-dimensional $G$-invariant component, then $\mathbf X$ is thick of order 1 relative to a collection of wide, convex subcomplexes, whence $G$ is thick of order 1.

Moreover, $G$ is strongly algebraically thick of order 1 if and only if $\simp\mathbf X$ is disconnected and has a positive-dimensional, $G$-invariant connected subcomplex $\mathfrak C=\cup_{A\in\mathcal A,g\in G}gA$, where $\mathcal A$ is a finite collection of bounded subcomplexes such that:
\begin{enumerate}
 \item Each $\stabilizer(A)$ acts on $\mathbf X$ with a quasiconvex orbit.
 \item For each $A\in\mathcal A$, $f^{-1}(A)$ belongs to the limit set of $\stabilizer(A)$.
 \item $f^{-1}(\mathfrak C)$ is contained in the limit set of $\langle\{\stabilizer(A):A\in\mathcal A\}\rangle$.
\end{enumerate}
\end{thm}

\begin{proof}
The first assertion is the content of Theorem~\ref{thm:algebraicallythickoforder1}.  Remark~\ref{rem:extrahypotheses} shows that $\mathfrak C$ satisfies $(1)-(3)$.  The converse is Theorem~\ref{thm:thickoforderone}, with the equivalence of strong algebraic thickness of order 1 is equivalent to algebraic thickness of order 1 relative to quasiconvex wide subgroups being established by Proposition~\ref{prop:stalgthick}.
\end{proof}

\section{Characterizations of thickness and relative hyperbolicity via the Tits boundary}\label{sec:titschar}
When regarding $\mathbf X$ as a combinatorial object, it is natural
to use the simplicial boundary; as a CAT(0) space, $\mathbf X$ also has a Tits
boundary $\partial_T\mathbf X$.  By viewing each simplex of $\simp\mathbf X$ as
a right-angled spherical simplex whose 1-simplices have length $\frac{\pi}{2}$,
one realizes $\simp\mathbf X$ as a piecewise-spherical CAT(1) space.
Proposition~3.37 of~\cite{HagenBoundary} asserts that, when $\mathbf X$ is fully visible, there is an isometric embedding $I\co\simp\mathbf X\rightarrow\partial_T\mathbf X$ such that $\partial_T\mathbf X\subseteq\mathcal N_{\frac{\pi}{2}}(\image I)$.  (The map $I$ is an isometric embedding with respect to the piecewise-spherical CAT(1) metric on $\simp\mathbf X$.)  This map sends each 0-simplex $v$ --- which, by full visibility, is represented by some CAT(0) geodesic ray $\gamma$  --- to the point of $\simp\mathbf X$ represented by $\gamma$.
It follows that $I$ is $G$-equivariant, and induces a bijection from the set of components (respectively, the set of isolated 0-simplices) of $\simp\mathbf X$ to the set of components (respectively, the set of isolated points) of the Tits boundary.

Moreover, $I$ is a section of a surjective map $R\co\partial_T\mathbf X\rightarrow\simp\mathbf X$ such that the $R$-preimage of any point is connected, has diameter at most $\frac{\pi}{2}$, and consists of points represented by rays that represent the same simplex in $\simp\mathbf X$. Furthermore, in the cocompact case, if the simplicial boundary contains infinitely many isolated points, then so does the Tits boundary.

\begin{cor}\label{cor:relhyptits1}
Let the group $G$ act geometrically on the fully visible CAT(0) cube complex
$\mathbf X$.

Suppose that $G$ is hyperbolic relative to a collection
$\mathbb P$ of peripheral subgroups.  Then $\partial_T\mathbf X$ consists of a
nonempty set of disjoint closed balls of radius less than $\frac{\pi}{2}$, together with a collection $\{g\mathbf T_P:P\in\mathbb P,g\in G\}$ of subspaces such that $\stabilizer(\mathbf T_P)=P$ for all $P\in\mathbb P$ and $g\mathbf T_P\cap h\mathbf T_{P'}=\emptyset$ unless $P=P'$ and $gh^{-1}\in P$.

Conversely, suppose that the set of isolated points of $\partial_T\mathbf X$ is
nonempty, and that there is a pairwise-disjoint, $G$-finite collection
$G(\{\mathbf S_i\}_{i=1}^k)$ of subspaces of $\partial_T\mathbf X$ such that
each $P_i=\stabilizer_G(\mathbf S_i)$ is quasiconvex and of infinite index in
$G$, each $\mathbf S_i$ contains the limit set for the action of $P_i$ on $\visual\mathbf X$, and every point of $\partial_T\mathbf X$ lies in some
$g\mathbf S_i$ or in some isolated ball of radius less than $\frac{\pi}{2}$.  Then $G$ is hyperbolic relative to $\{P_i\}_{i=1}^k$.
\end{cor}

\begin{proof}
If $G$ is relatively hyperbolic, then each $\mathbf T_P=R^{-1}(\mathbf S_P)$, where $\mathbf S_P$ is one of the subcomplexes arising from Theorem~\ref{thm:relhyp}.  It is easily verified that the resulting family of subspaces has the desired properties.  Every other point in $\partial_T\mathbf X$ lies in $R^{-1}(p)$ for some isolated 0-simplex $p$.  Any two points in the preimage of the same isolated point correspond to rays that are almost-equivalent and thus represent points at Tits distance strictly less than $\frac{\pi}{2}$.

Conversely, suppose that $\partial_T\mathbf X=\mathbf B\cup(\sqcup_{g\in G,P\in\mathbb P}g\mathbf T_P)$, where $\mathbf B$ is the disjoint union of the isolated balls.  Then for each $g,P$, let $g\mathbf S_P=R(g\mathbf T_P)=gR(\mathbf T_P)$.  This is a $P^g$-invariant subcomplex, and any two of these subcomplexes are disjoint.  For each $b\in\mathbf B$, $R(b)$ must be an isolated 0-simplex, and it follows from Theorem~\ref{thm:relhypconverse} that $G$ is hyperbolic relative to $\mathbb P$.
\end{proof}

\begin{cor}\label{cor:thicktits}
Let $G$ act properly and cocompactly on the fully visible CAT(0) cube complex $\mathbf X$.
If $G$ is algebraically thick of order 1, then $\partial_T\mathbf X$ has a proper $G$-invariant connected
component.

Conversely, if $\partial_T\mathbf X$ has this feature, then $G$ is thick of order 1 relative to a collection of wide subsets.  Suppose, in addition, that $\partial_T\mathbf X$ has a connected $G$-invariant subspace $\mathfrak C=\cup_{g\in G,A\in\mathcal A}$, where $\mathcal A$ is a finite set of connected subspaces satisfying:
\begin{enumerate}
 \item For all $A\in\mathcal A$, the stabilizer $H_A$ of $A$ is quasiconvex.
 \item For all $A\in\mathcal A$, the limit set of $H_A$ (in the cone topology on $\visual\mathbf X$) contains $A$.
 \item The limit set of $\langle\{H_A:A\in\mathcal A\}\rangle$ contains $\mathfrak C$.
\end{enumerate}
Then $G$ is strongly algebraically thick of order 1 relative to a collection of quasiconvex, wide subgroups, and $G$ has polynomial divergence function of order exactly 2.
\end{cor}

\begin{proof}
If $G$ is algebraically thick of order 1, then $\simp\mathbf X$ has a $G$-invariant connected subspace $\mathfrak C'$ that is properly contained in $\simp\mathbf X$, by Theorem~\ref{thm:algebraicallythickoforder1}.  Let $\mathfrak C=R^{-1}(\mathfrak C')$.  The definition of $R$ implies that $\mathfrak C$ is connected: each simplex has connected $R$-preimage.  Also, $\mathfrak C$ does not contain all of $\partial_T\mathbf X$ since $R$ is surjective and distance-nonincreasing, and $\simp\mathbf X$ has more than one component.

Conversely, if $\mathfrak C$ is a $G$-invariant connected subspace of $\partial_T\mathbf X$, then $R(\mathfrak C)$ is a $G$-invariant connected subspace of $\simp\mathbf X$, whence $\mathbf X$ is thick by Theorem~\ref{thm:thickoforderone}.  It is easily verified that $\{R(A):A\in\mathcal A\}$ satisfies the hypotheses of Theorem~\ref{thm:thickoforderone}, from which strong algebraic thickness of order 1 follows.
\end{proof}

\section{Cubulated groups with arbitrary order of
thickness}\label{sec:arbitrary}
The goal of this section is to produce cocompactly cubulated groups of any
order of thickness; in fact, the groups we produce will be strongly algebraically
thick of the desired order.

\begin{notation}
For $n\geq 1$, we will let $\mathbb G_n$ denote the class of groups such that each $G\in\mathbb G_n$
acts properly and cocompactly on a CAT(0) cube complex, is strongly algebraically
thick of order at most $n$, and has polynomial divergence of order $n+1$.
\end{notation}

Note that $\mathbb G_n$ does not contain any groups of dimension $1$, since a $1$-dimensional CAT(0) cube complex is a tree, and hence such a group could not have polynomial divergence.

\begin{lem}\label{lem:thickoforder1}
For each dimension $k>1$, the class $\mathbb G_1$ has an infinite subclass
of pairwise non-quasi-isometric groups of geometric dimension $k$.
\end{lem}

\begin{proof}
Let $\Gamma$ be a connected graph with at least two vertices that does not
decompose as a nontrivial join.  The universal cover $\mathbf X_{\Gamma}$ of
the Salvetti complex of the associated right-angled Artin group $G(\Gamma)$ is
a combinatorially geodesically complete CAT(0) cube complex on which
$G(\Gamma)$ acts properly, cocompactly, and essentially.

According to~\cite{BehrstockCharney}, the right-angled Artin group $G(\Gamma)$
is algebraically thick of order 1 and has quadratic divergence, since
$\Gamma$ is not a nontrivial join.

A connected graph $\Gamma$ is~\emph{atomic} if it has no leaves, if its girth
is at least 5, and no vertex-star is separating.  It is shown in~\cite{BKS}
that, if $\Gamma_1,\Gamma_2$ are atomic graphs, then $G(\Gamma_1)$ and
$G(\Gamma_2)$ are quasi-isometric if and only if $\Gamma_1\cong\Gamma_2$.
Since there are obviously infinitely many isomorphism types of finite atomic
graphs, it follows that $\mathbb G_1$ contains infinitely many pairwise non-quasi-isometric groups each of dimension 2.

For each $k>2$, the irreducible $k$--tree groups constructed in \cite{BehrstockJanuszkiewiczNeumann:highdimartin} provide an infinite family of $k$--dimensional right-angled Artin groups which are all algebraically thick of order 1. Further, it was shown in \cite{BehrstockJanuszkiewiczNeumann:highdimartin} that this family contains infinitely many pairwise non-quasi-isometric groups.
\end{proof}

\begin{thm}\label{thm:thickofanyorder}
For each dimension $k>1$ and each $n\geq 1$, the class $\mathbb G_n$ contains an infinite class of
pairwise non-quasi-isometric groups of geometric dimension $k$.
\end{thm}

\begin{proof}
The claim holds when $n=1$ by Lemma~\ref{lem:thickoforder1}.
For $n\geq 1$, by
induction there exists a group $G_n\in\mathbb G_n$ acting freely,
cocompactly, and essentially on a $k$--dimensional CAT(0) cube complex $\mathbf X_n$ that is
algebraically thick of order $n$ and has divergence of order $n+1$.

\textbf{Construction of $G_{n+1}$ and $\mathbf X_{n+1}$:}
By~\cite[Corollary~B]{CapraceSageev}, there exists $g\in G_n$ acting on
$\mathbf X_n$ as a rank-one isometry.
Let $\gamma\subset\mathbf X_n$ be a
CAT(0) geodesic axis for $g$.
By induction, we can choose $g$ so that $\gamma$ has
divergence of order at least $n+1$.  Since $g$ is rank-one, the cubical convex hull
$K_n$ of $\gamma$ lies in a finite neighborhood of $\gamma$.  Hence the
stabilizer $C_n\leq G$ of $K_n$ contains $\langle\gamma\rangle$ as a
finite-index subgroup.

Let $G_{n+1}=G_n\ast_{C_n}G_n$, and denote by $T_n$ the associated Bass-Serre
tree.  The space $\mathbf X_{n+1}$ is defined to be the total space of the tree
of spaces whose underlying tree is $T_n$, whose vertex-spaces are copies of
$\mathbf X_n$ and whose edge-spaces are copies of $K_n$ corresponding to cosets
of $C_n$.  The attaching maps are inclusions.  Since $\mathbf X_{n+1}$ is
obtained by gluing CAT(0) cube complexes along convex subcomplexes, it is
nonpositively curved and therefore a CAT(0) cube complex, by virtue of being
simply connected.  There is an obvious free, cocompact, essential action of
$G_{n+1}$ on $\mathbf X_{n+1}$, where the vertex-stabilizers are conjugate to
$G_n$ and the edge-stabilizers are conjugate to $C_n$.

We remark that collapsing each edge-space $K_n\times[-1,1]$ to $K_n$ within
$\mathbf X_{n+1}$ yields a new $G_{n+1}$-cocompact CAT(0) cube complex $\mathbf
X_{n+1}'$ with $\dimension\mathbf X_{n+1}'=\dimension\mathbf X_n$.  Although we
work in $\mathbf X_{n+1}$ for convenience, this observation shows, by induction
on $n$, that $G_{n+1}$ can always be chosen to act properly and cocompactly on
a CAT(0) cube complex of dimension $\dimension\mathbf X_1$, where $\mathbf X_1$
corresponds to some $G_1\in\mathbb G_1$.  To prove that $\mathbb G_n$ contains
infinitely many quasi-isometry types of $k$-dimensionally cocompactly cubulated
groups, one needs only to add to the induction hypothesis that
$\dimension\mathbf X_n=k$ and note that Lemma~\ref{lem:thickoforder1} has already accounted for the base case.

\textbf{An upper bound on order of thickness:}  By Lemma~\ref{lem:convex}
below, $h\mathbf X_n$ is a convex subcomplex of $\mathbf X_{n+1}$ for each
$h\in G$, and $h\mathbf X_n$ is thick of order $n$.

By construction, $\mathbf X_{n+1}$ is contained in the 1-neighborhood of
$G_{n+1}\mathbf X_n$.  Therefore, for any $x,y\in\mathbf X_{n+1}$, there exist
$h_0,h_m\in G$ such that $\done(x,h_0\mathbf X_n)\leq 1$ and $\done(y,h_m\mathbf
X_n)\leq 1$.  Let $h_0\mathbf X_n,h_1\mathbf X_n,\ldots,h_m\mathbf X_n$ be the
sequence of vertex-spaces corresponding to the sequence of vertices in the
projection to $T_n$ of a geodesic in $\mathbf X_{n+1}$ joining $x$ to $y$.  By
construction $h_i\mathbf X_n\cap h_{i+1}\mathbf X_n$ is a translate of $K_n$
for $0\leq i\leq m-1$.  Since $K_n$ is unbounded, the set $\{h\mathbf X_n:h\in
G_{n+1}\}$ is thickly connecting, whence $\mathbf X_{n+1}$, and therefore
$G_{n+1}$, is thick of order at most $n+1$.  Since each translate of $\mathbf
X_n$ is stabilized by a conjugate of one of the two vertex groups in the
splitting $G_{n+1}\cong G_n\ast_{C_n}G_n$, and $K_n$ has infinite stabilizer,
we see that $G_{n+1}$ is algebraically thick of order at most $n+1$.

\textbf{A lower bound on divergence:}  By Lemma~\ref{lem:malnormal}, $C_n$ is a
malnormal subgroup of $G_{n+1}$ and the action of $G_{n+1}$ on $T_n$ is
acylindrical by Lemma~\ref{lem:acylindrical}. The proof
of~\cite[Proposition~5.2]{BD} can now be repeated almost verbatim to show that
for any $g'\in G_{n+1}$ acting axially on $T_n$, any geodesic axis for $g'$ in
$\mathbf X_{n+1}$ has divergence of order at least $n+2$.  The only difference
is that the ``separating geodesics'' discussed in~\cite{BD} are replaced here
by tubular neighborhoods of $\gamma$ that contain $K_n$ and therefore separate
$\mathbf X_{n+1}$.

\textbf{Infinitely many quasi-isometry types:}  Denote by $A$ and $B$ the
copies of $G_n$ that are vertex groups of the splitting $G_{n+1}\cong
G_n\ast_{C_n}G_n$, so that $\{A,B\}$ is a set of subgroups showing that
$G_{n+1}$ has order of algebraic thickness at most $n+1$.  Let
$G'_{n+1}\in\mathbb G_{n+1}$ and define $A',B'\leq\mathbb G_{n+1}$ analogously
(so that $A'$ and $B'$ are both isomorphic to some $G'_n\in\mathbb G_n$).  If
$q:G_{n+1}\rightarrow G'_{n+1}$ is a quasi-isometry, then $q(A)$ and $q(B)$
are respectively coarsely equal to $A$ and $B$ (or $B$ and $A$), as in the
construction in~\cite[Section~5]{BD}, because of quasi-isometry
invariance of the splitting over $\integers$, which follows
from~\cite[Theorem~7.1]{PapasogluZsplittings}. Hence $G_n$ and $G'_n$ are
quasi-isometric, and therefore the set of quasi-isometry types represented in
$\mathbb G_{n+1}$ has cardinality at least that of the set of quasi-isometry
types represented in $\mathbb G_1$, and the latter is infinite by
Lemma~\ref{lem:thickoforder1}.
\end{proof}

\begin{lem}\label{lem:convex}
$\mathbf X_n$ and $K_n$ are convex subcomplexes of $\mathbf X_{n+1}$.
\end{lem}

\begin{proof}
$\mathbf X_{n+1}$ is the union of copies of $\mathbf X_n$ and copies of
$K_n\times[-1,1]$.  We denote by $K_n$ the subspace $K_n\times\{-1\}$ of
$\mathbf X_n$.

Since $\mathbf X_{n+1}$ is CAT(0), it is sufficient to verify that $\mathbf
X_n$ and $K_n$ are locally convex.  Suppose to the contrary that $s$ is a
2-cube whose boundary path is a 4-cycle $abcd$ with $ab\subset K_n$.  If
$s\subset\mathbf X_n$, then $cd$ is a combinatorial geodesic segment in $\mathbf
X_n$ starting and ending on $K_n$, whence $s\subset K_n$ since $K_n$ is convex
in $\mathbf X_n$.  Otherwise, $s$ lies in the copy of $K_n\times[-1,1]$
projecting to the edge of $T_n$ corresponding to $K_n$.  The unique possibility
in this case is that $s\subset K_n$.  Hence $K_n$ is convex.

A 2-cube with two
consecutive boundary 1-cubes in $\mathbf X_n$ has two consecutive boundary
1-cubes in some $\stabilizer(\mathbf X_n)$-translate of $K_n$, and must
therefore lie in $K_n\subset\mathbf X_n$.  Thus $\mathbf X_n$ is convex.
\end{proof}

\begin{lem}\label{lem:malnormal}
$C_n$ is a malnormal subgroup of $G$.
\end{lem}

\begin{proof}
If $C_n$ fails to be malnormal, then there exists $h\in
G_{n+1}-C_n$ and nonzero integers $r,s$ such that $g^r=hg^sh^{-1}$.  Since $C$
is quasi-isometrically embedded in $G_{n+1}$, we must have $|r|=|s|$, so that
without loss of generality, $r=s$ and $h$ is a hyperbolic isometry of $\mathbf
X_{n+1}$.  There is thus a $\langle g^r,h\rangle$-invariant flat in $\mathbf
X_{n+1}$ coarsely containing $\gamma$, and this contradicts the fact that $g$
is a rank-one isometry of $\mathbf X_{n}$ and that $\mathbf X_{n+1}$ is a tree of
spaces where the vertex spaces are copies of $\mathbf X_{n}$ and the edge
spaces are each contained in finite neighborhoods of copies of the
axis of $g$.
\end{proof}

\begin{defn}[Acylindrical]\label{defn:acylindrical}
The isometric action of the group $G$ on the graph $Y$ is \emph{acylindrical}
if for some $\ell>0$, there exists $M<\infty$ such that
$|\stabilizer(x)\cap\stabilizer(y)|\leq M$ whenever $x$ and $y$ are at distance
at least $\ell$ in $Y$.
\end{defn}

\begin{lem}\label{lem:acylindrical}
The action of $G_{n+1}$ on $T_n$ is acylindrical.
\end{lem}

\begin{proof}
Let $x,y$ be vertices corresponding to $h_xG_n$ and $h_yG_n$, with
$d_{T_n}(x,y)=2$.  Let $z$ be the midpoint of the unique geodesic joining $x$
to $y$, and denote by $h_zG_n$ the corresponding coset.  If $k\in G_n^{h_x}\cap
G_n^{h_y}$, then $k$ stabilizes two distinct edges in $T_n$, corresponding to
distinct translates of $K_n$ in $h_z\mathbf X_n$.  Lemma~\ref{lem:malnormal}
implies that $k=1$.

If $d_{T_n}(x,y)>2$, then since geodesics in trees are unique, there exist
$x',y'$ between $x,y$ such that $d_{T_n}(x',y')=2$ and every element of
$G_{n+1}$ stabilizing $x$ and $y$ must also stabilize $x'$ and $y'$.
\end{proof}

\bibliographystyle{alpha}
\bibliography{ThickCubes1130}

\end{document}